\documentclass[11pt]{amsart}
\usepackage{etex}

\usepackage[margin=1in]{geometry}

\usepackage{amscd,latexsym,amsthm,amsfonts,amssymb,amsmath,amsxtra}
\usepackage[mathscr]{eucal}
\usepackage{pictexwd,dcpic}
\usepackage{xcolor}
\usepackage[normalem]{ulem}
\usepackage{hyperref}
\usepackage[all,cmtip]{xy}
\usepackage{enumitem}

\renewcommand\theequation{\thesection.\arabic{equation}}

\newcommand{\BA}{{\mathbb {A}}}

\newcommand{\BC}{{\mathbb {C}}}

\newcommand{\BR}{{\mathbb {R}}}

\newcommand{\CA}{{\mathcal {A}}}

\newcommand{\CF}{{\mathcal {F}}}

\newcommand{\CL}{{\mathcal {L}}}

\newcommand{\CP}{{\mathcal {P}}}

\newcommand{\CT}{{\mathcal {T}}}

\newcommand{\Fa}{{\mathfrak {a}}}

\newcommand{\Fl}{{\mathfrak {l}}}

\newcommand{\Fo}{{\mathfrak {o}}}

\newcommand{\Fs}{{\mathfrak {s}}}

\newcommand{\RU}{{\mathrm {U}}}

\newcommand{\GL}{{\mathrm{GL}}}

\newcommand{\GSp}{{\mathrm{GSp}}}

\newcommand{\GO}{{\mathrm{GO}}}
\newcommand{\GSO}{{\mathrm{GSO}}}
\newcommand{\GSpin}{{\mathrm{GSpin}}}
\newcommand{\Spin}{{\mathrm{Spin}}}

\newcommand{\mm}[4]{\left(\begin{smallmatrix} #1 & #2\\ #3 & #4\end{smallmatrix}\right)}

\newcommand{\Hom}{{\mathrm{Hom}}}

\newcommand{\PGL}{{\mathrm{PGL}}}

\newcommand{\SL}{{\mathrm{SL}}}

\newcommand{\GE}{{\mathrm{GE}}}

\newcommand{\SO}{{\mathrm{SO}}}

\newcommand{\Sp}{{\mathrm{Sp}}}

\newcommand{\Span}{{\mathrm{Span}}}

\newcommand{\tr}{{\mathrm{tr}}}

\newcommand{\ul}{\underline}

\newcommand{\back}{\backslash}

\newcommand{\calF}{\mathcal{F}}

\newcommand{\ago}{\mathfrak{a}}
\newcommand{\hDelta}{\widehat \Delta}
\newcommand{\htau}{\widehat \tau}
\newcommand{\sbs}{\subset}

\newcommand{\smin}{\smallsetminus}
\newcommand{\al}{\alpha}

\newcommand{\La}{\Lambda}
\newcommand{\bsl}{\backslash}
\newcommand{\calA}{\mathcal{A}}

\def\C{{\mathbb C}}

\def\R{{\mathbb R}}
\def\Z{{\mathbb Z}}
\def\Q{{\mathbb Q}}
\def\A{{\mathbb A}}
\def\BA{{\mathbb A}}
\def\BR{{\mathbb R}}
\def\BC{{\mathbb C}}

\def\back{{\backslash}}

\newtheorem{thm}{Theorem}[section]
\newtheorem{cor}[thm]{Corollary}
\newtheorem{lem}[thm]{Lemma}
\newtheorem{claim}[thm]{Claim}
\newtheorem{prop}[thm]{Proposition}

\newtheorem {assu}[thm]{Assumption}
\newtheorem {ques/conj}[thm]{Question/Conjecture}

\newtheorem{defn}[thm]{Definition}
\newtheorem{rmk}[thm]{Remark}

\makeatletter

\newcommand{\Rmnum}[1]{\expandafter\@slowromancap\romannumeral #1@}
\makeatother

\begin{document}
\renewcommand{\theequation}{\arabic{equation}}
\numberwithin{equation}{section}

\title{On the residue method for period integrals}

\author{Aaron Pollack}
\address{Department of Mathematics\\
Duke University\\
Durham, NC 27708, USA}

\email{apollack@math.duke.edu}

\author{Chen Wan}
\address{Department of Mathematics\\
Massachusetts Institute of Technology\\
Cambridge, MA 02139, USA}
\email{chenwan@mit.edu}

\author{Micha\l{} Zydor}
\address{Department of Mathematics\\
University of Michigan\\
Ann Arbor, MI 48109, USA}
\email{zydor@umich.edu}

\begin{abstract}
By applying the residue method for period integrals and Langlands-Shahidi's theory for residues of Eisenstein series, we study the period integrals for six spherical varieties. For each spherical variety, we prove a relation between the period integrals and certain automorphic L-functions. In some cases, we also study the local multiplicity of the spherical varieties.
\end{abstract}

\maketitle

\section{Introduction and main results}
Let $k$ be a number field, and $\BA$ its ring of adeles. Let $G$ be a reductive group defined over $k$, and let $H$ be a closed subgroup of $G$. Assume that $X=H\back G$ is spherical variety of $G$ (i.e. the Borel subgroup $B\subset G$ acts with a Zariski dense orbit). Let $A_G$ be the maximal split torus of the center of $G$ and $A_{G,H}=A_G\cap H$. Let $\pi$ be a cuspidal automorphic representation of $G(\BA)$ whose central character is trivial on $A_{G,H}(\BA)$. For $\phi\in \pi$, we define the period integral $\CP_{H}(\phi)$ to be
$$\CP_{H}(\phi):=\int_{H(k)A_{G,H}(\BA)\back H(\BA)} \phi(h)dh.$$

One of the most fundamental problems in the relative Langlands program is to find the relation between the period integral $\CP_{H,\chi}(\phi)$ and some automorphic L-function $L(s,\pi,\rho_X)$. Here $\rho_X:{}^LG\rightarrow \GL_n(\BC)$ is a finite dimensional representation of the L-group ${}^LG$ of $G$.

In this paper, by applying the residue method for period integrals and Langlands-Shahidi's theory for residues of Eisenstein series, we study the period integrals for six spherical varieties. For each spherical variety, we prove a relation between the period integrals and some automorphic L-functions. The L-functions that are related to these spherical varieties include the standard L-functions of the general linear group, orthogonal group, unitary group and $\mathrm{GE_6}$, the exterior square L-function of $\GL_{2n}$, and a degree 12 L-function of $\GL_4\times \GL_2$. In some cases, we also study the local multiplicity of the spherical varieties.

\subsection{The main results}
In this subsection, we are going to summarize the main results of this paper. For simplicity, we will only state the main results when $G$ is split (except for the model related to unitary group which is only quasisplit). We refer the readers to later sections for details about the quasi-split and non-split cases. All the L-functions that show up in this paper are the completed Langlands-Shahidi L-functions; in particular, they include local factors from the archimedean places.

\begin{rmk}\label{reductive vs non-reductive}
In this paper, we only consider the case when $H$ is reductive. However, our method can also be applied to the non-reductive case. For instance, in our previous paper \cite{AWZ18}, we studied the period integral for the Ginzburg-Rallis model, which is non-reductive.
\end{rmk}

\subsubsection{The model $(\SO_{2n+1},\SO_{n+1}\times \SO_n)$}
Let $G=\SO_{2n+1}$ be the split odd orthogonal group, $\rho_X$ be the standard representation of ${}^LG=\Sp_{2n}(\BC)$, and $H=\SO_{n+1}\times \SO_{n}$ be a closed subgroup of $G$ ($H$ does not need to be split or quasi-split).

\begin{thm}\label{SO(2n+1) main}
Let $\pi$ be a cuspidal generic automorphic representation of $G(\BA)$. If the period integral $\CP_{H}(\phi)$ is nonzero for some $\phi\in \pi$, then the L-function $L(s,\pi,\rho_X)$ is nonzero at $s=1/2$.
\end{thm}

Theorem \ref{SO(2n+1) main} will be proved in Section \ref{Section SO(2n+1)}.

\subsubsection{The model $(\SO_{2n},\SO_{n+1}\times \SO_{n-1})$}
Let $G=\SO_{2n}$ be the split even orthogonal group, $\rho_X$ be the standard representation of ${}^LG=\SO_{2n}(\BC)$, and $H=\SO_{n+1}\times \SO_{n-1}$ be a closed subgroup of $G$ ($H$ does not need to be split or quasi-split).

\begin{thm}\label{SO(2n) main}
Let $\pi$ be a cuspidal generic automorphic representation of $G(\BA)$. If the period integral $\CP_{H}(\phi)$ is nonzero for some $\phi\in \pi$, then the L-function $L(s,\pi,\rho_X)$ has a pole at $s=1$.
\end{thm}

Locally, let $F$ be a p-adic field, and $\pi$ be an irreducible smooth representation of $G(F)$. We define the multiplicity
$$m(\pi):=\dim(\Hom_{H(F)}(\pi,1)).$$

\begin{thm}\label{SO(2n) local theorem}
Let $\pi$ be an irreducible generic tempered representation of $G(F)$. If $m(\pi)\neq 0$, then the local L-function $L(s,\pi,\rho_X)$ has a pole at $s=0$.
\end{thm}

Theorem \ref{SO(2n) main} and \ref{SO(2n) local theorem} will be proved in Section \ref{Section SO(2n)}.

\subsubsection{The model $(\RU_{2n},\RU_{n}\times \RU_n)$}
Let $k'/k$ be a quadratic extension, $G=\RU_{2n}$ be the quasi-split unitary group, and $H=\RU_n\times \RU_n$ be a closed subgroup of $G$ ($H$ does not need to be quasi-split). Let $\pi$ be a generic cuspidal automorphic representation of $G(\BA)$ with trivial central character, and let $\Pi$ be the base change of $\pi$ to $\GL_{2n}(\BA_{k'})$. Let $L(s,\pi)$ (resp. $L(s,\Pi)$) be the standard L-function of $\pi$ (resp. $\Pi$). We have $L(s,\pi)=L(s,\Pi)$.

\begin{thm}\label{U(2n) main}
If the period integral $\CP_{H}(\phi)$ is nonzero for some $\phi\in \pi$, then the standard L-function $L(s,\pi)=L(s,\Pi)$ is nonzero at $s=1/2$. Moreover, if we assume that $\Pi$ is cuspidal and there exists a local place $v\in |k|$ such that $k'/k$ splits at $v$ and $\pi_v$ is a discrete series of $G(k_v)=\GL_{2n}(k_v)$, then the exterior square L-function $L(s,\Pi,\wedge^2)$ has a pole at $s=1$ (i.e. $\Pi$ is of symplectic type).
\end{thm}

\begin{rmk}
Combining Theorem \ref{U(2n) main} with the result in \cite{FJ} for the linear model $(\GL_{2n},\GL_n\times \GL_n)$, we know that if the period integral $\CP_{H}(\phi)$ is nonzero on the space of $\pi$, then the $\GL_n(\BA_{k'})\times \GL_n(\BA_{k'})$-period integral is nonzero on the space of $\Pi$.
\end{rmk}

\begin{rmk}
The model $(\RU_{2n},\RU_{n}\times \RU_n)$ was mentioned to the third author
independently by L. Clozel, J. Getz and K. Prasanna during his stay at the IAS.
We would like to thank them for suggesting it to us.

In fact,
our result may be viewed as a special case of a principle put forward by Getz and Wambach in \cite{getWom}.
They conjectured that for any reductive group $H$ and any involution $\sigma$ of $H$, the non-vanishing of the period integrals of the model $(H,H^{\sigma})$ ($H^{\sigma}$ being the group of fixed points of $\sigma$) for a cuspidal automorphic representation $\pi$ of $H(\A)$
should be (roughly) equivalent to the non-vanishing of the period integrals of the model $(G,G^{\sigma})$ for the base change of $\pi$ to $G(\A)$. Here $G=Res_{k'/k}H$ with $k'/k$ quadratic. Our result in Theorem \ref{U(2n) main} confirms one direction of a special case of their conjecture.

On the other hand, the model $(\RU_{4},\RU_{2}\times \RU_2)$ and its twists also
appear in the work of Ichino-Prasanna \cite{ichPras} in the context of algebraic cycles on Shimura varieties.
\end{rmk}

Theorem \ref{U(2n) main} will be proved in Section \ref{Section U(2n)}.

\subsubsection{The Jacquet-Guo model}
Let $k'/k$ be a quadratic extension, $G=\GL_{2n}$, and $H=Res_{k'/k} \GL_n$ be a closed subgroup of $G$ (in particular $H(\BA)=\GL_n(\BA_{k'})$). The model $(G,H)$ is the so called Jacquet-Guo model, and it was first be studied in \cite{G96}. Let $\rho_{X,1}$ (resp. $\rho_{X,2}$) be the standard representation (resp. exterior square representation) of ${}^LG=\GL_{2n}(\BC)$.

\begin{thm}\label{Jacquet-Guo main}
Let $\pi$ be a cuspidal automorphic representation of $G(\BA)$ with trivial central character. If the period integral $\CP_{H}(\phi)$ is nonzero for some $\phi\in \pi$, then the L-function $L(s,\pi,\rho_{X,1})$ is nonzero at $s=1/2$ and the L-function $L(s,\pi,\rho_{X,2})$ has a pole at $s=1$.
\end{thm}

\begin{rmk}
In \cite{FMW}, under some local requirements on $\pi$ (i.e. $\pi$ is supercuspidal at some split place and $H$-elliptic at another place), the authors prove Theorem \ref{Jacquet-Guo main} by the relative trace formula method.
\end{rmk}

Locally, let $F$ be a p-adic field, $E/F$ be a quadratic extension, $G(F)=\GL_{2n}(F)$ and $H(F)=\GL_n(E)$. Let $\pi$ be an irreducible smooth representation of $G(F)$ with trivial central character. We define the multiplicity
$$m(\pi):=\dim(\Hom_{H(F)}(\pi,1)).$$

\begin{thm}\label{Jacquet-Guo local theorem}
Let $\pi$ be an irreducible tempered representation of $G(F)$ with trivial central character. If $m(\pi)\neq 0$, then the local L-function $L(s,\pi,\rho_{X,2})$ has a pole at $s=0$.
\end{thm}

Theorems \ref{Jacquet-Guo main} and \ref{Jacquet-Guo local theorem} will be proved in Section \ref{section Jacquet-Guo}.

\subsubsection{The model $(\mathrm{GE_6}, A_1\times A_5)$}
Let $G=\mathrm{GE_6}$ be the similitude group of the split exceptional group $\mathrm{E_6}$.  Fix a quaternion algebra $B$ over $k$, and define $H = (B^\times \times \GL_3(B))^0:=\{(x,g)\in B^\times \times \GL_3(B): n_B(x) = N_6(g)\}$; here $n_B$ is the degree two reduced norm on $B$ and $N_6$ is the degree six reduced norm on $M_3(B)$. One has a map $H \rightarrow \GE_6$ with $\mu_2$ kernel.  Let $\rho_X$ be a 27 dimensional fundamental representation of ${}^LPGE_6=\mathrm{E_6^{sc}}(\BC)$.

\begin{thm}\label{E_6 main}
Let $\pi$ be a cuspidal generic automorphic representation of $G(\BA)$ with trivial central character. Assume that $L(2,\pi,\rho_X)\neq 0$ (this is always the case if $\pi$ is tempered). If the period integral $\CP_{H}(\phi)$ is nonzero for some $\phi\in \pi$, then the L-function $L(s,\pi,\rho_X)$ has a pole at $s=1$.
\end{thm}

Locally, let $F$ be a $p$-adic field. Given an irreducible smooth representation $\pi$ of $G(F)$ with trivial central character, we define the multiplicity
$$m(\pi):=\dim(\Hom_{H(F)}(\pi,1)).$$

\begin{thm}\label{E_6 local theorem}
Let $\pi$ be an irreducible generic tempered representation of $G(F)$ with trivial central character. If $m(\pi)\neq 0$, then the local L-function $L(s,\pi,\rho_X)$ has a pole at $s=0$.
\end{thm}

Theorem \ref{E_6 main} and \ref{E_6 local theorem} will be proved in Section \ref{Section E_6}.

\subsubsection{The model $(\GL_4\times \GL_2, \GL_2\times \GL_2)$}
Let $G=\GL_4\times \GL_2$, and $H=\left\{\begin{pmatrix}a&0\\0&b\end{pmatrix}\times \begin{pmatrix}a\end{pmatrix}|a,b\in \GL_2 \right\}$ be a closed subgroup of $G$. Let $\rho_X=\wedge^2\otimes std$ be a 12 dimensional representation of ${}^LG=\GL_4(\BC)\times \GL_2(\BC)$.

\begin{thm}\label{GL(4)GL(2)}
Let $\pi$ be a cuspidal automorphic representation of $G(\BA)$ with trivial central character. Assume that $L(3/2,\pi,\rho_X)\neq 0$ (this is always the case if $\pi$ is tempered). If the period integral $\CP_{H}(\phi)$ is nonzero for some $\phi\in \pi$, then $L(1/2,\pi,\rho_X)\neq 0$.
\end{thm}

\begin{rmk}
Since $\pi$ has trivial central character, by the exceptional isomorphism $\mathrm{PGL}_4 \simeq \mathrm{PGSO}_6$, we can view $\pi$ as a cuspidal automorphic representation of $\GSO_6(\BA)$ with trivial central character. Then the L-function $L(s,\pi,\rho_X)$ becomes the tensor L-function of $\GSO_6\times \GL_2$.
\end{rmk}

Theorem \ref{GL(4)GL(2)} will be proved in Section \ref{section GL(4)GL(2)}. Locally, in Section \ref{subsection GL(4)GL(2) local}, we will show that the summation of the multiplicities of the model $(G,H)$ is always equal to 1 over every tempered local Vogan L-packet.

\subsection{Organization of the paper and remarks on the proofs}
The theorems on global $L$-functions are all proved by the residue method, together with the Langlands-Shahidi's theory for residues of Eisenstein series. Recall that in the residue method one relates the period integrals of cuspidal representations to the period integrals of certain residue representations. This method goes back to Jacquet-Rallis \cite{JR92}, and has been applied by Jiang \cite{J98}, Ginzburg-Jiang-Rallis \cite{GJR04},\cite{GJR05},\cite{GJR09}, Ichino-Yamana \cite{IY}, Ginzburg-Lapid \cite{GL07}, and by us in a previous paper \cite{AWZ18}.  In Section 3, which serves as an extended introduction, we explain our strategy of proof in more detail. We will also discuss the connection between the residue method and the dual groups of spherical varieties.

The paper is organized as follows. In Section 2, we set up notations relating to Eisenstein series and truncation operators. Then in Section 3, we explain the strategy of the proofs of the main theorems (i.e. the residue method). In Section 4-9, we prove the main theorems for all six spherical varieties.

\subsection{Acknowledgements} This work was initiated at the Institute for Advanced Study in 2018, when the three of us were members.  We thank the IAS for its hospitality and pleasant working environment.  We also thank the Institute for Mathematical Sciences at the National University of Singapore, where the three of us visited in January 2019.  A.P. thanks the Simons Foundation for its support via Collaboration Grant number 585147, which helped make this work possible.

\section{Eisenstein series and the truncation operators}
\subsection{General notations}
Let $G$ be a connected reductive algebraic group over $k$. We fix a maximal $k$-split torus $A_{0}$ of $G$. Let $P_{0}$ be a minimal parabolic subgroup of $G$ defined over $k$ containing $A_0$, $M_{0}$ be the Levi part of $P_{0}$ containing $A_{0}$ and $U_{0}$ be the unipotent radical of $P_{0}$. Let $\CF(P_{0})$ be the set of parabolic subgroups of $G$ containing $P_{0}$. Elements in $\CF(P_0)$ are called standard parabolic subgroups of $G$. We also use $\CF(M_0)=\CF(A_0)$ (resp. $\CL(M_0)$) to denote the set of parabolic subgroups (resp. Levi subgroups) of $G$ containing $A_0$; these are the semi-standard parabolic subgroups (resp. Levi subgroups).

For $P \in \CF(M_{0})$, we have the Levi decomposition $P = MN$ with $N$ be the unipotent radical of $P$ and $M$ be the Levi subgroup containing $A_{0}$. We use $A_{P} \sbs A_{0}$ to denote the maximal $k$-split torus of the center of $M$. Put
\[
 \ago_{0}^{*} = X(A_{0}) \otimes_{\Z}\R = X(M_{0}) \otimes_{\Z}\R
\]
and let $\ago_{0}$ be its dual vector space.
Here $X(H)$, for any $k$-group $H$, denotes the group of rational characters of $H$.
The inclusions $A_{P} \sbs A_{0}$ and $M_{0} \sbs M$ identify $\ago_{P}$
as a direct factor of $\ago_{0}$, we use $\ago_{0}^{P}$ to denote its complement.
Similarly, $\ago_{P}^{*} = X(A_{P}) \otimes_{\Z}\R$ is a direct factor of $\ago_{0}^{*}$
and we use $\ago_{0}^{P,*}$ to denote its complement. Let $\Delta_P\subset \Fa_{P}^{*}$ be the set of simple roots for the action of $A_P$ on $N$ and we use $\Delta_0$ to denote $\Delta_{P_0}$. Similarly, for $P\subset Q$, we can also define the subset $\Delta_{P}^{Q}\subset \Delta_P$. Then we define the chamber
$$\Fa_{P}^{+}=\{H\in \Fa_P|\; \langle H,\alpha \rangle \;>0,\;\forall \alpha\in \Delta_P\}.$$

Let $\Delta_{0}^{\vee} \sbs \ago_{0}^{G}$ be the set of simple coroots given by the theory
of root systems. For $\al \in \Delta_{0}$ we denote $\al^{\vee} \in \Delta_{0}^{\vee}$
the corresponding coroot. We define $\hDelta_{0} \sbs \ago_{0}^{G, *}$ to be the dual basis of
$\Delta_{0}^{\vee}$, i.e. the set of weights. In particular, we get a natural bijection
between $\Delta_{0}$ and $\hDelta_{0}$ which we denote by $\al \mapsto \varpi_{\al}$.
Let $\hDelta_{P} \sbs \hDelta_{0}$  be the set corresponding to $\Delta_{0} \smin \Delta_{0}^{P}$.

For any subgroup $H \sbs G$, let $H(\A)^{1}$ denote the common kernel of all
characters on $H(\A)$ of the form $|\chi(\cdot)|_{\A}$ where $\chi \in X(H)$
and $|\cdot |_{\A}$ is the absolute value on the ideles of $\A$. Fix $K$ a maximal compact subgroup of $G(\A)$ adapted to $M_{0}$. We define the Harish-Chandra map $H_{P} : G(\A) \to \ago_{P}$ via the relation
\[
 \langle \chi, H_{P}(x) \rangle = |\chi(p)|_{\A}, \quad \forall \chi \in X(P) = \mathrm{Hom}(P,\mathbb{G}_m)
\]
where $x = pk$ is the Iwasawa decomposition $G(\A) = P(\A)K$. Let $A_{P}^{\infty}$ be the connected component of the identity of $\mathrm{Res}_{k/\Q}A_{P}(\R)$. Then $M(\A)^{1}$ is the kernel of $H_{P}$ restricted to $M(\A)$ and we have the direct product decomposition of commuting subgroups $M(\A) = A_{P}^{\infty} M(\A)^{1}$.

For any group $H$ we use $[H]$ to denote $H(k)\back H(\BA)$ and $[H]^1$ to denote $H(k) \back H(\A)^{1}$.

\subsection{Haar measures}
We fix compatible Haar measures on $G(\BA)$, $G(\BA)^1$ and $A_{G}^{\infty}$. For all unipotent subgroups $N$ of $G$, we fix a Haar measure on $N(\A)$ so that $[N]$ is of volume one. On $K$ we also fix a Haar measure of volume $1$. For any $P=MN\in \calF(A_0)$, let $\rho_{P} \in \ago_{P}^{*}$ be the half sum of the weights of the action of $A_{P}$ on $N$. We choose compatible Haar measures on $A_{P}^{\infty}$ and $M_P(\BA)^1$ such that
\[
 \int_{P(k) \bsl H(\A)}f(h) \, dh =
 \int_{K}\int_{[M]^{1}}   \int_{A_{P}^{\infty}} \int_{[U]}
 e^{\langle -2\rho_{P}, H_{P}(a) \rangle}
 f(uamk) \, du da dm dk
\]
for $f \in C_{c}^{\infty}(P(k) \bsl G(\A))$.

\subsection{The computation of $\rho_P$ when $P$ is maximal}\label{ssec:rho}
Let $P \in \calF(P_{0})$ be the maximal parabolic subgroup that corresponds to the simple root $\alpha$, i.e. $\{\al\} = \Delta_{0} \smin \Delta_{0}^{P}$.
Let $\varpi$ be the corresponding weight. We have $\rho_{P} \in \ago_{P}^{G,*}$. Since $P$ is maximal, $\ago_{P}^{G,*}$ is one dimensional. Hence there exists a constant $c \in \R$ such that $\rho_{P} = c \varpi$. In the following proposition, we write down the constant $c$ in five cases. It will be used in later sections. The computation is easy and standard, and hence we will skip it.

\begin{prop}\label{the constant c}
\begin{enumerate}
\item If $G=\SO_n$ and $P$ is the parabolic subgroup whose Levi part is isomorphic to $\SO_{n-2}\times \GL_1$, then $c=\frac{n-2}{2}$.
\item If $G=\Sp_{2n}$ and $P$ is the Siegel parabolic subgroup, then $c=\frac{n+1}{2}$.
\item If $G=\RU_{n}$ and $P$ is the parabolic subgroup whose Levi part is isomorphic to $\RU_{n-2}\times \GL_1$, then $c=\frac{n-1}{2}$.
\item If $G=\SO_{10}$ and $P$ is the parabolic subgroup whose Levi part is isomorphic to $\SO_6\times \GL_2$, then $c=\frac{7}{2}$.
\item If $G=\mathrm{E_7}$ is simply-connected and $P$ is the parabolic subgroup whose Levi part is of type $\mathrm{E_6}$, then $c=9$.
\end{enumerate}
\end{prop}

\subsection{Eisenstein series}\label{ssec:eis}
Let $P=MN$ be a parabolic subgroup of $G$. Given a cuspidal automorphic representation $\pi$ of $M(\A)$, let $\calA_{\pi}$ be the space of automorphic forms $\phi$ on $N(\A)M(k) \bsl G(\A)$ such that $M(\A)^1 \ni m \mapsto \phi(mg) \in L^{2}_{\pi}([M]^{1})$ for any $g \in G(\A)$, where $L^{2}_{\pi}([M]^{1})$ is the $\pi$-isotypic part of $L^{2}([M]^{1})$, and such that
\[
 \phi(ag) = e^{\langle \rho_{P}, H_{P}(a) \rangle} \phi(g), \quad \forall g \in G(\A),\;a\in A_{P}^{\infty}.
\]
Suppose that $P$ is a maximal parabolic subgroup. Let $\varpi \in \hDelta_{P}$ be the corresponding weight. We then define
\[
 E(g, \phi, s) = \sum_{\delta \in P(k) \bsl G(k)}
 \phi(\delta g)e^{\langle s\varpi, H_{P}(\delta g )\rangle}, \quad s \in \C, \
 g \in G(\A).
\]
The series converges absolutely for $s \gg 0$ and admits a meromorphic continuation to all $s\in \BC$.

Suppose moreover that $M$ is stable for the conjugation by the simple reflection in the Weyl group of $G$ corresponding to $P$. We have in this case the intertwining operator $M(s) : \calA_{\pi} \to \calA_{\pi}$ that satisfies $E(M(s)\phi, -s) = E(\phi, s)$ and
\[
E(g, \phi, s)_{P} = \phi(g)e^{\langle s \varpi, H_{P}(g) \rangle} +
 e^{\langle -s\varpi, H_{P}(g) \rangle }M(s)\phi(g) ,\quad g \in G(\A)
\]
where $E( \cdot , \phi, s)_{P}$ is the constant term of $E(\cdot , \phi, s)$ along $P$
\[
 E(g, \phi, s)_{P} :=  \int_{[N]}E(ug, \phi, s) \, du.
\]

When the Eisenstein series $E(g, \phi, s)$ has a pole at $s=s_0$, the intertwining operator also has a pole at $s=s_0$, we use $Res_{s=s_0}E(g, \phi, s)$ (resp. $Res_{s=s_0}M(s)$) to denote the residue of the Eisenstein series (resp. intertwining operator). The poles of Eisenstein series $E(g, \phi, s)$ are simple for $Re(s)>0$ and their residues are square integrable automorphic forms. Also the Eisenstein series, their derivatives and residues are of moderate growth.

\subsection{Arthur-Langlands truncation operator}\label{subsec:trunc}
We continue assuming that $P$ is maximal.
We identify the space $\ago_{P}^{G}$ with $\R$ so that
$T \in \R$ corresponds to an element whose pairing with $\varpi \in \ago_{P}^{*}$
is $T$. We will assume this isomorphism is measure preserving. Let $\htau_{P}$ be the characteristic function of
\[
 \{H \in \ago_{P} \ | \ \varpi(H) > 0 \ \forall \varpi \in \hDelta_{P}\}.
\]
Given a locally integrable function $F$ on $G(k) \bsl G(\A)$
we define its truncation as
\[
 \La^{T}F(g) = F(g) - \sum_{\delta \in P(k) \bsl G(k)}
 \htau_{P}(H_{P}(\delta g) - T) \int_{[N]}F(u\delta g) \, du, \quad g \in G(k) \bsl G(\A),
\]
where $T \in \R$ and the sum is actually finite.

\subsection{The relative truncation operator and the regularized period integral}\label{section relative truncation}
For later applications, we also need the relative truncation operator which was recently introduced by the third author in \cite{Z19}. Let $H\subset G$ be a closed connected reductive subgroup, and let $P=MN\subset G$ still be a maximal parabolic subgroup. With the same notation as in Section \ref{ssec:rho}, let $\varpi$ be the corresponding weight and $c \in \R$ be the constant such that $\rho_{P} = c \varpi$. Fix a maximal split torus $A_{0,H}$ (resp. $A_0$) of $H$ (resp. $G$) such that $A_{0,H}\subset A_0$. Then $\Fa_H:=\Fa_{0,H}$ is a subspace of $\Fa_0$. For simplicity, we assume that $G$ has trivial split center (i.e. $A_G=\{1\}$).

\begin{rmk}
In \cite{Z19}, the author defined the relative truncation operator for general automorphic functions and also for a general pair $(G,H)$ with $H$ reductive ($H$ does not need to be a spherical subgroup). But for our applications in this paper, we only consider the case when when the automorphic function is a cuspidal Eisenstein series induced from a maximal parabolic subgroup.
\end{rmk}

We fix a minimal subgroup $P_{0,H}$ of $H$ with $A_{0,H}\subset P_{0,H}$. This allows us to define the set of standard (resp. semi-standard) parabolic subgroups of $H$.  We will use $\CF_H(P_{0,H})$ (resp. $\CF_H(A_{0,H})$) to denote this set. We can also define the chamber $\Fa_{H}^{+}=\Fa_{P_{0,H}}^{+}$ of $\Fa_{H}$. Let $\bar{\Fa}_{H}^{+}$ be the closure of $\Fa_{H}^{+}$.

\begin{defn}
We use $\CF^G(P_{0,H},P)$ to denote the set of semi-standard parabolic subgroups $Q=LU\in \CF(A_0)$ of $G$ that satisfy the following two conditions.
\begin{enumerate}
\item $Q$ is a conjugate of $P$.
\item $\Fa_{Q}^{+}\cap \bar{\Fa}_{H}^{+}\neq \emptyset$.
\end{enumerate}
\end{defn}

The next proposition was proved in Proposition 3.1 of \cite{Z19}.
\begin{prop}\label{equal rank}
Let $Q=LU$ be a semi-standard parabolic subgroup of $G$ that is conjugate to $P$. Then the following statements hold.
\begin{enumerate}
\item If $Q\in \CF^G(P_{0,H},P)$, then $Q_H=Q\cap H$ is a standard parabolic subgroup of $H$ with the Levi decomposition $Q_H=L_HU_H$ such that $L_H=L\cap H$ and $U_H=U\cap H$.
\item Let $A_L$ be the maximal split torus of the center of $L$. If $A_L\subset A_{0,H}$ (this is always the case when $A_0=A_{0,H}$), then the inverse of (1) holds. In other words, if $Q_H=Q\cap H$ is a standard parabolic subgroup of $H$ with the Levi decomposition $Q_H=L_HU_H$ such that $L_H=L\cap H$ and $U_H=U\cap H$, then $Q\in \CF^G(P_{0,H},P)$.
\end{enumerate}
\end{prop}

\begin{rmk}\label{closed orbit rmk}
Let $F=k_v$ ($v\in |k|$) be a local field. Combining the proposition above and Corollary \ref{closed orbit cor} in Section \ref{Section closed orbit}, we know that for all $Q\in \CF^G(P_{0,H},P)$, $Q(F)H(F)$ is closed in $G(F)$.
\end{rmk}

For $Q\in \CF^G(P_{0,H},P)$, let $A_L\subset A_0$ be the maximal split torus of the center of $L$, and let $\varpi_Q$ be the weight correspond to $Q$. By the definition of the constant $c$, we have
$$\rho_Q=c\varpi_Q.$$
Since $\Fa_{Q}^{+}\cap \bar{\Fa}_{H}^{+}\neq \emptyset$ and $\Fa_Q$ is one-dimensional, we have $\Fa_Q\subset \Fa_H$. We can therefore restrict $\rho_{Q_H}$ to $\Fa_Q$ and we define the real number $c_{Q}^{H}$ to satisfy
$$\rho_{Q_H}|_{\Fa_Q}=c_{Q}^{H}\rho_Q.$$

\begin{rmk}\label{constant c(Q,H)}
In the case when $U$ is abelian, let $n_Q=\dim(U)$ and $n_{Q,H}=\dim(U_H)$. Then
$$c_{Q}^{H}=\frac{n_{Q,H}}{n_Q}.$$
\end{rmk}

We fix a cuspidal automorphic representation $\pi$ of $M(\BA)$ and let $E(g,\phi,s)$ be the Eisenstein series defined in the previous section. For $Q\in \CF^G(P_{0,H},P)$,
let $W(P,Q)$ be the two element set of isometries between $\Fa_{P}$ and $\Fa_{Q}$.
For $w \in W(P,Q)$ let $sgn(w) \in \{-1,1\}$ be such that $w\varpi_{P} = sgn(w) \varpi_{Q}$.
We have then
$$E(g,\phi,s)_Q=
\sum_{w \in W(P,Q)} M(w,s)\phi(g)e^{\langle sgn(w)s\varpi_Q, H_Q(g)\rangle}$$
for some explicit intertwining operators $M(w, s)$, independent of $s$ if $sgn(w) = 1$.

In \cite{Z19}, the author defined a relative truncation operator, denoted by $\Lambda^{T,H}$, on the space $\CA(G)$ of autormorphic forms on $G$, where $T \in \Fa_{H}$.
It depends on a choice of a good maximal compact $K_{H}$ of $H(\BA)$ which we fix now.
For all $\varphi\in \CA(G)$, the truncation $\Lambda^{T,H}\varphi$ is a rapidly decreasing function on $[H]$.
The following is the consequence of Theorem 4.1 of \cite{Z19} and the discussion in Paragraph 4.7 of loc. cit.

\begin{thm}\label{relative truncation}
\begin{enumerate}
\item For all $\phi\in \CA_{\pi}$ and $T \in \Fa_{H}^{+}$ sufficiently regular, the integral
$$\int_{[H]} \Lambda^{T,H} E(h, \phi, s)dh$$
is absolutely convergent for all $s\in \BC$ in the domain of holomorphy of the Eisenstein series $E(\phi,s)$.
Moreover, it defines a meromorphic function on $\BC$.
\item Define the regularized period for $E(\phi,s)$ to be
$$\CP_{H,reg}(E(\phi,s)):=\int_{[H]} \Lambda^{T,H} E(h, \phi, s)dh-\sum_{Q\in \CF^G(P_{0,H},P)} \sum_{w\in W(P,Q)}$$
$$\frac{e^{\langle(sgn(w)s+c(1-2c_{Q}^{H}))\varpi_Q, T\rangle}}{sgn(w)s +c(1-2c_{Q}^{H})} \int_{K_H}\int_{L_H(k)A_{L}^{\infty}\back L_H(\BA)} M(w,s)\phi(mk)dmdk.$$
Then the integrals defining $\CP_{H,reg}(E(\phi,s))$ are absolutely convergent and the functional $\CP_{H,reg}(\cdot)$ is right $H(\BA)$-invariant.
\end{enumerate}
\end{thm}

\subsection{Some nonvanishing results of automorphic L-functions}
\begin{prop}\label{L-function nonzero}
\begin{enumerate}
\item Let $\pi$ be a cuspidal automorphic representation of $\GL_n(\BA)$. Then the standard L-function $L(s,\pi)$ is holomorphic nonzero when $Re(s)>1$.
\item Let $\pi$ be a generic cuspidal automorphic representation of $\SO_n(\BA)$, $\Sp_{2n}(\BA)$ or $\mathrm{U}_n(\BA)$. Then the standard L-function $L(s,\pi)$ is  holomorphic nonzero when $Re(s)>1$.
\item Let $\pi$ be a cuspidal automorphic representation of $\GL_{2n}(\BA)$. Then the exterior square L-function $L(s,\pi,\wedge^2)$ is nonzero when $Re(s)>1$.
\end{enumerate}
\end{prop}

\begin{proof}
By Theorem 5.3 of \cite{JS81} and Theorem 2.4 of \cite{CPS}, the tensor L-function $L(s,\tau\times \sigma)$ is holomorphic nonzero for $Re(s)>1$ for any cuspidal automorphic representation $\tau$ (resp. $\sigma$) of $\GL_{m_1}(\BA)$ (resp. $\GL_{m_2}(\BA)$). This proves (1) by taking $\tau=\pi$ and $\sigma=1$. (2) is a direct consequence of (1) together with the results in \cite{CKPSS} and \cite{KK05}.

For (3), by taking $\sigma=\tau=\pi$, we know that the L-function $L(s,\pi\times \pi)=L(s,\pi,Sym^2)L(s,\pi,\wedge^2)$ is holomorphic nonzero for $Re(s)>1$. Hence it is enough to show that the symmetric square L-function $L(s,\pi,Sym^2)$ is holomorphic when $Re(s)>1$. This follows from Corollary 5.8 of \cite{T14} and Theorem 3.1(3) of \cite{K00}.
\end{proof}

\subsection{A criterion for closed orbit}\label{Section closed orbit}
Let $F$ be a local field, $G$ be a linear algebraic group defined over $F$ and $H,P$ be two closed subgroups of $G$.

\begin{prop}\label{closed orbit}
Suppose that the morphism
\[
H /(H \cap P) \hookrightarrow G/ P
\]
is a closed immersion. Then $H(F)P(F)$ is closed in $G(F)$.
\end{prop}

\begin{proof}
Let $P_H = H \cap P$, and $H \times^{P_H} P $ be the quotient on the right of
$H \times P$ by the diagonally embedded subgroup $P_{H}$. Let $p : H \times^{P_H} P \to G$ be the morphism
induced by the map $H \times P \to G$ given by $(h,p) \mapsto hp^{-1}$. We have the following cartesian diagram
\[
 \xymatrix{ H \times^{P_H} P \ar[r]^{p} \ar[d] & G \ar[d] \\
               H / P_H \ar[r] & G/P }
\]
which shows that $H \times^{P_H} P $ is isomorphic to the fiber product $(H / P_H) \times_{G/P} G$. By assumption, $H / P_H \hookrightarrow G/P $ is a closed immersion. By the property of pullbacks, $p : H \times^{P_H} P \to G$ is also a closed immersion. Hence $(H \times^{P_H} P) (F)$ is closed in $G(F)$.

The set $H(F)P(F)$ is identified with the subset $(H(F) \times P(F)) / P_H(F)$ of $(H \times^{P_H} P) (F)$. The inclusion
$(H(F) \times P(F)) / P_H(F) \hookrightarrow (H \times^{P_H} P) (F)$ is closed by Corollary A.1.6 of \cite{AG}. This proves the proposition.
\end{proof}

\begin{cor}\label{closed orbit cor}
Let $G$ be a connected reductive group, $H\subset G$ a closed connected reductive subgroup, and $P\subset G$ a parabolic subgroup (all defined over $F$).
Assume that $P_H=H\cap P$ is a parabolic subgroup of $H$. Then $H(F)P(F)$ is closed in $G(F)$.
\end{cor}

\begin{proof}
By the proposition above, we only need to show that the morphism $H / P_H \hookrightarrow G/ P$ is a closed immersion. Since $P$ is a parabolic subgroup of $G$ and $P_H$ is a parabolic subgroup of $H$, we know that both $G/P$ and $H/P_H$ are projective. Hence the morphism $H / P_H \hookrightarrow G/ P$ is a closed immersion.
\end{proof}

\section{The strategy of the proof}

\subsection{The first step}
Let $(G,H)$ be one of the six spherical pairs in Section 1. Our goal is to prove a relation between the period integral $\CP_{H}(\phi)$ and certain automorphic L-function $L(s,\phi,\rho_X)$. The first step of our method is to find another spherical variety $\underline{X}=\underline{H}\back\underline{G}$ that satisfies the following three conditions.
\begin{enumerate}
\item $G$ is isomorphic to the Levi component $\underline{M}$ of a maximal parabolic subgroup $\underline{P}=\underline{M}\underline{U}$ of $\underline{G}$ (up to modulo the center).
\item The L-function $L(s,\pi,\rho_X)$ appears in the Langlands-Shahidi L-function for $(\underline{G},\underline{M})$.
\item $\underline{H}\cap \underline{P}=(\underline{H}\cap \underline{M})\ltimes (\underline{H}\cap \underline{N})$ is a maximal parabolic subgroup of $\underline{H}$ such that $\underline{H}\cap \underline{M}$ is isomorphic to the group $H$ (up to modulo the center).
\end{enumerate}
Such a spherical variety does not exist in general. But if it exists, we can use it to prove a relation between the period integral and the automorphic L-function. In particular, for all of the six spherical pairs in Section 1, we can find a spherical pair $(\underline{G},\underline{H})$ that satisfies the conditions above. For simplicity, we assume that $\underline{G}$ has trivial split center.

\begin{rmk}
In \cite{KS}, Knop-Schalke have defined the dual group $G_{X}^{\vee}$ for every affine spherical varieties $X=H\back G$ together with a natural morphism $\iota_X:G_{X}^{\vee}\rightarrow G^{\vee}$ from the dual group of the spherical variety to the dual group of $G$. Let $Cent_{G^{\vee}}(Im(\iota_X))$ be the centralizer of the image of the map $\iota_X$ in $G^{\vee}$. Following the notation in \cite{KS}, we use $\Fl_{X}^{\wedge}$ to denote the Lie algebra of $Cent_{G^{\vee}}(Im(\iota_X))$. We say the spherical variety $X$ is tempered if $\Fl_{X}^{\wedge}=0$.

For all the six spherical pairs $(G,H)$ in Section 1 (as well as all the other known cases), the dual groups of the spherical varieties $X=H\back G$ and $\underline{X}=\underline{H}\back \underline{G}$ satisfy the following two conditions.
\begin{itemize}
\item[(4)] $G_{X}^{\vee}=G_{\underline{X}}^{\vee}$.
\item[(5)] $\Fl_{X}^{\wedge}=0$, $\Fl_{\underline{X}}^{\wedge}=\Fs\Fl_2$ or $\Fs\Fo_3$.
\end{itemize}
In general, we believe that for a given spherical pair $(G,H)$, if we can find another spherical pair $(\underline{G},\underline{H})$ that satisfies Conditions (1),(4) and (5), then it should also satisfies Conditions (2) and (3) (up to conjugating the parabolic subgroup $\underline{P}$).

In this paper, we have considered all the spherical pairs $(G,H)$ with $G$ simple and $H$ reductive (see Table 3 of \cite{KS}) such that there exists another spherical pair $(\underline{G},\underline{H})$ that satisfies Condition (1), (4) and (5) (except those pairs that have already been studied by other people). We also consider a case when $G$ is not simple, i.e. the model $(\GL_4\times \GL_2,\GL_2\times \GL_2)$.
\end{rmk}

\noindent
After we find the spherical pair $(\underline{G},\underline{H})$, we consider the period integral
$$\CP_{\underline{H}}(E(\phi,s)):=\int_{\underline{H}(k)\back \underline{H}(\BA)} E(\phi,s)(\underline{h})  d\underline{h}$$
for a cuspidal Eisenstein series $E(\phi,s)$ of $\underline{G}(\BA)$ induced from the maximal parabolic subgroup $\ul{P}=\ul{M}\ul{U}$ and the cuspidal automorphic representation $\pi$ of $\ul{M}(\BA)\simeq G(\BA)$. This period integral is not convergent in general, hence we need to truncate the Eisenstein series $E(\phi,s)$. As we explained in the previous section, we have two different truncation operators. In the next two subsections, we will explain how to study the truncated period integrals by using these two different truncation operators. Our goal is to prove a relation between the truncated $\ul{H}$-period integral of the residue of the Eisenstein series $E(\phi,s)$ ($\phi\in \CA_{\pi}$) and the $H$-period integral of the cusp forms in $\pi$.

To end this subsection, we give a list of $(\ul{G},\ul{H})$ for the six spherical pairs in Section 1. We refer the reader to Table 3 of \cite{KS} for the dual groups of spherical varieties.

\begin{itemize}
\item When $(G,H)=(\SO_{2n+1},\SO_{n+1}\times \SO_n)$, we let $(\ul{G},\ul{H})=(\SO_{2n+3},\SO_{n+3}\times \SO_n)$. In this case, $G_{X}^{\vee}=G_{\underline{X}}^{\vee}=\Sp_{2n}(\BC)$. This model will be discussed in Section \ref{Section SO(2n+1)}.
\item When $(G,H)=(\SO_{2n},\SO_{n+1}\times \SO_{n-1})$, we let $(\ul{G},\ul{H})=(\SO_{2n+2},\SO_{n+3}\times \SO_{n-1})$. In this case, $G_{X}^{\vee}=G_{\underline{X}}^{\vee}=\SO_{2n-1}(\BC)$. This model will be discussed in Section \ref{Section SO(2n)}.
\item When $(G,H)=(\mathrm{U}_{2n},\mathrm{U}_n\times \mathrm{U}_n)$, we let $(\ul{G},\ul{H})=(\mathrm{U}_{2n+2},\mathrm{U}_{n+2}\times \mathrm{U}_n)$. In this case, $G_{X}^{\vee}=G_{\underline{X}}^{\vee}=\Sp_{2n}(\BC)$. This model will be discussed in Section \ref{Section U(2n)}.
\item When $(G,H)=(\GL_{2n},Res_{k'/k}\GL_n)$, we let $(\ul{G},\ul{H})=(\Sp_{4n},Res_{k'/k}\Sp_{2n})$. In this case, $G_{X}^{\vee}=G_{\underline{X}}^{\vee}=\Sp_{2n}(\BC)$. This model will be discussed in Section \ref{section Jacquet-Guo}.
\item When $(G,H)=(\GE_6,A_1\times A_5)$, we let $\ul{G}=\mathrm{E_7}$ and $\ul{H}$ be the symmetric subgroup of $\ul{G}$ of type $A_1\times D_6$. In this case, $G_{X}^{\vee}=G_{\underline{X}}^{\vee}=\mathrm{F}_4(\BC)$. This model will be discussed in Section \ref{Section E_6}.
\item When $(G,H)=(\GL_4\times \GL_2,\GL_2\times \GL_2)$, we let $(\ul{G},\ul{H})=(\mathrm{GSO}_{10},\mathrm{GSpin}_7\times \GL_1)$. In this case, $G_{X}^{\vee}=G_{\underline{X}}^{\vee}=\GL_4(\BC)\times \GL_2(\BC)$. This model will be discussed in Section \ref{section GL(4)GL(2)}.
\end{itemize}

\subsection{Method 1: relative truncation operator}\label{section method 1}
In this subsection, we will use the relative truncation operator to study the period integral $\CP_{\underline{H}}(E(\phi,s))$. We recall from Theorem \ref{relative truncation} the definition of the regularized period integral:
$$\CP_{\underline{H},reg}(E(\phi,s)):=\int_{[\ul{H}]} \Lambda^{T,\underline{H}} E(h, \phi, s)dh-\sum_{\underline{Q}\in \CF^{\underline{G}}(P_{0,\underline{H}},\underline{P})} \sum_{w\in W(\underline{P},\underline{Q})}$$
$$\frac{e^{\langle(sgn(w)s+c(1-2c_{\underline{Q}}^{\underline{H}}))\varpi_{\underline{Q}}, T \rangle}}{sgn(w)s +c(1-2c_{\underline{Q}}^{\underline{H}})} \int_{K_{\underline{H}}}\int_{L_{\ul{H}}(k)A_{\ul{L}}^{\infty}\back L_{\ul{H}}(\BA)} M(w,s)\phi(mk)dmdk.$$
Then we need to show that for $Re(s)>>0$, the following two statements hold.
\begin{enumerate}
\item The regularized period integral $\CP_{\underline{H},reg}(E(\phi,s))$ is equal to 0.
\item For all $\underline{Q}\in \CF^{\underline{G}}(P_{0,\underline{H}},\underline{P})$ with $\underline{Q}\neq \underline{P}$, $w\in W(P,Q)$ and $\phi\in \CA_{\pi}$, we have
$$\int_{L_{\ul{H}}(k)A_{\ul{L}}^{\infty}\back L_{\ul{H}}(\BA)} M(w,s)\phi(m)dm=0.$$
\end{enumerate}

Assume that we have proved (1) and (2). Then the equation above implies that
\begin{align}\label{1}
\int_{[\underline{H}]} \Lambda^{T,\underline{H}} E(h, \phi, s)dh &=\frac{e^{\langle(s+c(1-2c_{\underline{P}}^{\underline{H}}))\varpi_{\underline{P}}, T \rangle}}{s +c(1-2c_{\underline{P}}^{\underline{H}})} \int_{K_{\underline{H}}}\int_{[H]} \phi(hk)dhdk \\ \nonumber &\;\;\;\;+\frac{e^{\langle (-s+c(1-2c_{\underline{P}}^{\underline{H}}))\varpi_{\underline{P}}, T \rangle}}{-s +c(1-2c_{\underline{P}}^{\underline{H}})} \int_{K_{\underline{H}}}\int_{[H]} M(s)\phi(hk)dhdk.
\end{align}

Here we have used Condition (3) of the pair $(\underline{G},\underline{H})$. \eqref{1} tells us that the truncated $\ul{H}$-period integral of the Eisenstein series $E(\phi,s)$ is equal to the $H$-period integrals of $\phi$ and $M(s)\phi$. Let $s_0=-c(1-2c_{\underline{P}}^{\underline{H}})$. By taking the residue at $s=s_0$ for the above equation, we have
\begin{equation}\label{2}
\int_{[\underline{H}]} \Lambda^{T,\underline{H}} Res_{s=s_0}E(h, \phi, s)dh= \int_{K_{\underline{H}}}\int_{[H]} \phi(hk)dhdk+\frac{e^{\langle -2s_0 \varpi_{\underline{P}}, T \rangle }}{-2s_0} \int_{K_{\underline{H}}}\int_{[H]} Res_{s=s_0}M(s)\phi(hk)dhdk.
\end{equation}
In particular, we get the following proposition.

\begin{prop}\label{residue 3}
If the period integral $\CP_{H}(\phi)$ is non-zero for some $\phi\in \pi$, then the cuspidal Eisenstein series $E(\phi,s)$ and the intertwining operator $M(s)\phi$ has a pole at $s=s_0$.
\end{prop}

\begin{rmk}
When $\pi$ is generic, the above proposition implies that if the period integral $\CP_{H}(\phi)$ is non-zero for some $\phi\in \pi$, the Langlands-Shahidi L-function for $(\underline{G},\underline{M})$ has a pole at $s=s_0$. By assumption (2) of the pair $(\underline{G},\underline{H})$, the L-function $L(s,\pi,\rho_X)$ appears in the Langlands-Shahidi L-function for $(\underline{G},\underline{M})$. Hence the above proposition gives a relation between the period integral $\CP_{H}(\phi)$ and the automorphic L-function $L(s,\pi,\rho_X)$.
\end{rmk}

\begin{rmk}
A similar version of this method has been used by Ichino-Yamana (\cite{IY}) for the unitary Gan-Gross-Prasad model case.
\end{rmk}

Now we discuss the proof of (1) and (2). By Theorem \ref{relative truncation}, the regularized period integral $\CP_{\underline{H},reg}(E(\phi,s))$ is right $\underline{H}(\BA)$-invariant. As a result, in order to prove (1), it is enough to prove the following local statement.

\begin{itemize}
\item[(3)] Let $\pi=\otimes_{v\in |k|}^{'} \pi_v$, and let $\Pi_s=I_{\underline{P}}^{\underline{G}} \pi_s =\otimes_{v\in |k|}^{'} I_{\underline{P} }^{\underline{G}} \pi_{v,s}=\otimes_{v\in |k|}^{'} \Pi_{v,s}$. Here $\pi_s=\pi\otimes \varpi_{\underline{P}}^{s}$, $\pi_{v,s}=\pi_v\otimes \varpi_{\underline{P}}^{s}$, and $I_{\ul{P}}^{\ul{G}}(\cdot)$ is the normalized parabolic induction. Then there exists $v\in |k|$ such that
    $$\Hom_{\underline{H}(k_v)}(\Pi_{v,s},1)=\{0\}$$
    for $Re(s)>>0$. In other words, $\Pi_{v,s}$ is not $\underline{H}(k_v)$-distinguished.
\end{itemize}
Obviously (3) is not true for arbitrary $\pi$. Hence we need to make some assumption.

\begin{assu}\label{weak local assumption}
There exists a non-archimedean place $v\in |k|$ such that $\pi_v$ is generic.
\end{assu}

Now we fix $v\in |k|$ that satisfies the assumption above. In order to prove (3), it is enough to prove the following statement.
\begin{itemize}
\item[(4)] $\Hom_{\underline{H}(k_v)}(\tau_v,1)=\{0\}$ for all generic representations of $\underline{G}(k_v)$.
\end{itemize}

\begin{rmk}\label{local vanishing of orbit}
By the same argument as above, in order to prove (2), it is enough to show that an analogue of statement (4) holds for the pair $(\ul{L},L_{\ul{H}})$. In other words, it is enough to show that
\begin{itemize}
\item[(5)] $\Hom_{L_{\ul{H}}(k_v)}(\tau_v,1)=\{0\}$ for all generic representations of $\underline{L}(k_v)$.
\end{itemize}
\end{rmk}

\vspace{1em}

Under the relative local Langlands conjecture of Sakellaridis-Venkatesh in \cite{SV}, (4) should hold for all spherical varieties that are not tempered (note that by Condition (5) of the pair (\underline{G},\underline{H}), the spherical variety $\underline{X}=\underline{H}\back \underline{G}$ is not tempered). However, only the symmetric pair case has been recently proved by Prasad in \cite{P}. We will briefly recall his result in Section \ref{section Prasad}. Hence in order to prove (4), we need to make a stronger assumption.

\begin{assu}\label{local assumption}
\begin{enumerate}
\item $H$ is a symmetric subgroup of $G$.
\item There exists a non-archimedean place $v\in |k|$ such that $\pi_v$ is generic.
\end{enumerate}
\end{assu}

\noindent
In general, if one can extend Prasad's result to all spherical varieties, then Assumption \ref{weak local assumption} is enough. We want to point out that Assumption \ref{weak local assumption} is also necessary for this method because without the generic assumption, there exist examples such that Condition (1) and (2) fail.

\begin{rmk}
For the six spherical pairs in Section 1, five of them are symmetric. The only exception is the model $(\GL_4\times \GL_2,\GL_2\times \GL_2)$.
\end{rmk}

\subsubsection{Some remarks about the period integrals of the residue representations}

When the Eisenstein series $E(\phi,s)$ has a pole at $s=s_0$, the residue is a square integrable automorphic form with the only non-trivial exponent
$-s_0$ along parabolic subgroups conjugate to $\ul{P}$. By Theorem 4.1 of \cite{Z19}, the regularized period integral of $Res_{s=s_0} E(\phi,s)$ is the constant term in $T$ of the truncated period integral
$$\int_{[\underline{H}]} \Lambda^{T,\underline{H}} Res_{s=s_0}E(h, \phi, s)dh,$$
at least when $s_0\neq c(1-2c_{\ul{Q}}^{\ul{H}})$ for all $\ul{Q}\in \CF^{\underline{G}}(P_{0,\underline{H}},\underline{P})$ (this condition will be satisfied for all the spherical pairs that we consider). As a result, by taking the constant term in $T$ of \eqref{2}, we have
\begin{equation}\label{3}
\CP_{\ul{H},reg}(Res_{s=s_0} E(\phi,s))=\int_{K_{\underline{H}}}\int_{[H]} \phi(hk)dhdk.
\end{equation}
In other words, the regularized $\ul{H}$-period integral of $Res_{s=s_0} E(\phi,s)$ is equal to the $H$-period integral of $\phi$, up to some compact integration. By the same argument as in Lemma 5.8 of \cite{IY}, we obtain that the $H$-period integral is nonzero on the space of $\pi$ if and only if the regularized $\ul{H}$-integral is nonzero on the space $\{Res_{s=s_0} E(\phi,s)|\;\phi\in \CA_{\pi}\}$. This matches the general conjecture of Sakellaridis-Venkatesh \cite{SV} for period integrals because of the conditions (4) and (5) of the pair $(\ul{G},\ul{H})$ in Section 3.1.

Moreover, by Theorem 4.6 of \cite{Z19}, when
\begin{equation}\label{condition on s_0}
s_0> c(1-2c_{\ul{Q}}^{\ul{H}}),\;\forall\ul{Q}\in \CF^{\underline{G}}(P_{0,\underline{H}},\underline{P}),
\end{equation}
the regularized period integral $\CP_{\ul{H},reg}(Res_{s=s_0} E(\phi,s))$ is equal to the actual period integral
$$\int_{[\underline{H}]}  Res_{s=s_0}E(h, \phi, s)dh.$$
In particular, the actual period integral is absolutely convergent. For all the cases we consider, $s_0=-c(1-2c_{\ul{P}}^{\ul{H}})$ is a positive real number (in fact, it is either $1$ or $\frac{1}{2}$), hence the inequality in \eqref{condition on s_0} is automatic when $\ul{Q}=\ul{P}$. In particular, if the set $\CF^{\underline{G}}(P_{0,\underline{H}},\underline{P})$ only contains one element $\ul{P}$, \eqref{condition on s_0} holds. As a result, for those cases, the regularized period integral on the left hand side of \eqref{3} can be replaced by the actual period integral of $Res_{s=s_0}E(h, \phi, s)$. As we will see in later sections, for all the models we consider, the following are the cases when $\CF^{\underline{G}}(P_{0,\underline{H}},\underline{P})=\{\ul{P}\}$.

\begin{itemize}[leftmargin=5mm]
\item[-] $\ul{G}=\SO_{2n+3},\;\ul{H}=\SO_{n+3}\times \SO_n$ and the $\SO_n$-part of $\ul{H}$ is anisotropic, discussed in Section \ref{Section SO(2n+1)}.
\item[-] $\ul{G}=\SO_{2n+2},\;\ul{H}=\SO_{n+3}\times \SO_{n-1}$ and the $\SO_{n-1}$-part of $\ul{H}$ is anisotropic, discussed in Section \ref{Section SO(2n)}.
\item[-] $\ul{G}=\mathrm{U}_{2n+2},\;\ul{H}=\mathrm{U}_{n+2}\times \mathrm{U}_{n}$ and $\mathrm{U}_{n}$-part of $\ul{H}$ is anisotropic, discussed in Section \ref{Section U(2n)}.
\item[-] $\ul{G}=\Sp_{4n}$ and $\ul{H}=Res_{k'/k}\Sp_{2n}$, discussed in Section \ref{section Jacquet-Guo}.
\item[-] $\ul{G} = E_7^{sc}$ the semisimple, simply-connected group of type $\mathrm{E_7}$, $\ul{H}$ the symmetric subgroup of type $D_6\times A_1$, and $\ul{H}$ not split. This will be discussed in Section \ref{Section E_6}.
\end{itemize}

For all the other cases we consider, the set $\CF^{\underline{G}}(P_{0,\underline{H}},\underline{P})=\{\ul{P},\;\ul{P}'\}$ contains two elements. As we will see in later sections, in those cases, the inequality \eqref{condition on s_0} will fail when $\ul{Q}=\ul{P}'$. By Theorem 4.6 of \cite{Z19} again, the period integral
of $Res_{s=s_0}E(h, \phi, s)$ is divergent and the regularization is necessary in those cases. This phenomenon has already been observed for the model $(\Sp_{4n},\Sp_{2n}\times \Sp_{2n})$ by Lapid and Offen in \cite{LO18}.

\subsection{Method 2: Arthur-Langlands truncation operator}\label{section method 2}
In this subsection, we will use the Arthur-Langlands truncation operator to study the period integral $\CP_{\underline{H}}(E(\phi,s))$. We need one assumption.

\begin{assu}\label{assumption for finite orbits}
The double coset $\underline{P}(k)\back \underline{G}(k)/\underline{H}(k)$ has finitely many orbits.
\end{assu}

\begin{rmk}
For the six spherical varieties in Section 1, three of them satisfy this assumption: the Jacquet-Guo model $(\GL_{2n},\mathrm{Res}_{k'/k}(\GL_n))$, $(\GE_6,A_1\times A_5)$, and $(\GL_4\times \GL_2, \GL_2\times \GL_2)$).
\end{rmk}

Let $\{\gamma_i|\;1\leq i\leq l\}$ be a set of representatives for the double coset $\underline{P}(k)\back \underline{G}(k)/\underline{H}(k)$. For $1\leq i\leq l$, let $\underline{H}_i=\underline{H}\cap \gamma_{i}^{-1}\underline{P}\gamma_i$. Without loss of generality, we assume that $\gamma_1=1$.

Consider the truncated period integral
$$\CP_{\underline{H}}(\Lambda^T E(\phi,s)):=\int_{\underline{H}(k)\back \underline{H}(\BA)} \Lambda^T E(\phi,s)(\underline{h})  d\underline{h}$$
where $\Lambda^T$ is the Arthur-Langlands truncation operator. By unfolding the integral, we have
\begin{equation}\label{unfolding}
\CP_{\underline{H}}(\Lambda^T E(\phi,s))=\sum_{i=1}^{l} I_i(\phi,s)+J_i(\phi,s)
\end{equation}
with
$$I_i(\phi,s)=\int_{\ul{H_i}(k)\back \ul{H}(\BA)} (1-\hat{\tau}_{\ul{P}} (H_{\ul{P}}(\gamma_i \ul{h})-T)) e^{<s\varpi_{\ul{P}}, H_{\ul{P}}(\gamma_i \ul{h})>} \phi(\gamma_i \ul{h})  \;d\ul{h},$$
$$J_i(\phi,s)=\int_{\ul{H_i}(k)\back \ul{H}(\BA)} \hat{\tau}_{\ul{P}} (H_{\ul{P}}(\gamma_i \ul{h})-T) e^{<-s\varpi_{\ul{P}}, H_{\ul{P}}(\gamma_i \ul{h})>} M(s)\phi(\gamma_i \ul{h})  \;d\ul{h}.$$
Here $M(s)$ is the intertwining operator, and the factors $\hat{\tau}_{\ul{P}} (H_{\ul{P}}(\gamma_i \ul{h})-T),\; e^{<s\varpi_{\ul{P}}, H_{\ul{P}}(\gamma_i \ul{h})>}$ come from the truncation operator $\Lambda^T$.

The first step is to show that the integrals $I_i(\phi,s)$ and $J_i(\phi,s)$ are absolutely convergent when $Re(s)>>0$. In Section 5 of our previous paper \cite{AWZ18}, we have developed a general argument for proving the absolute convergence. The only thing we need to check is that $(H,H_i)$ is a good pair.  We refer the readers to Section 5.3 of \cite{AWZ18} for the definition of good pair.

After the first step, we need to show that when $Re(s)>>0$, we have
$$I_i(\phi,s)=J_i(\phi,s)=0$$
for $2\leq i\leq l$. For $i=1$, by Condition (3) of the pair $(\underline{G},\underline{H})$, we can show that
$$I_1(\phi,s)= \frac{e^{(s-s_0)T}}{s-s_0} \int_{K_{\underline{H}}}\int_{[H]} \phi(hk)dhdk,\;\; J_1(\phi,s)=\frac{e^{(-s-s_0)T}}{-s-s_0} \int_{K_{\underline{H}}}\int_{[H]} M(s)\phi(hk)dhdk$$
where $s_0=-c(1-c_{\ul{P}}^{\ul{H}})$ as in Method 1. This implies that
\begin{equation}\label{eqn:method2Fin}
\CP_{\underline{H}}(\Lambda^T E(\phi,s))=\frac{e^{(s-s_0)T}}{s-s_0} \int_{K_{\underline{H}}}\int_{[H]} \phi(hk)dhdk+\frac{e^{(-s-s_0)T}}{-s-s_0} \int_{K_{\underline{H}}}\int_{[H]} M(s)\phi(hk)dhdk.
\end{equation}
The above equation is an analogue of equation \eqref{1} for Method 1. Then we can use the same argument as in Method 1 to finish the proof.

\begin{rmk}
This method was first used by Jacquet-Rallis (\cite{JR92}) to study the period integrals of the residue representations for the model $(\GL_{2n},\Sp_{2n})$. Later in \cite{J98}, Jiang used this method to study the trilinear $\GL_2$ model (Jiang was the first one to use this method to study the period integrals of cusp forms). In our previous paper \cite{AWZ18}, we applied this method to the Ginzburg-Rallis model case. A similar version of this method has been used by Ginzburg-Jiang-Rallis (\cite{GJR04},\cite{GJR05},\cite{GJR09}) for the orthogonal Gan-Gross-Prasad model. See also \cite{GL07} for a slightly different approach.
\end{rmk}

\begin{rmk}\label{advantage}
Compared with Method 1, Method 2 has two disadvantages and two advantages. The two disadvantages are
\begin{itemize}
\item We need to assume that the double coset $\underline{P}(k)\back \underline{G}(k)/\underline{H}(k)$ has finitely many orbit (i.e. Assumption \ref{assumption for finite orbits}). In particular, it cannot be applied to the spherical pairs $(\SO_{2n+1},\SO_{n+1}\times \SO_n)$, $(\SO_{2n},\SO_{n+1}\times \SO_{n-1})$ and $(\RU_{2n},\RU_n\times \RU_n)$.
\item As we explained above, in Method 2, we need to show that $I_i(\phi,s)=J_i(\phi,s)=0$ for $2\leq i\leq l$. This requires us to study all the orbits in the double coset $\underline{P}(k)\back \underline{G}(k)/\underline{H}(k)$. On the other hand, for Method 1, we only need to study the closed orbits (see Remark \ref{closed orbit rmk}). In some cases (e.g. the model $(\GE_6,A_1\times A_5)$), those nonclosed orbits can be hard to study because one needs to compute explicitly the image in $\ul{M}=\ul{P}/\ul{N}$ of the intersection $\underline{P}\cap \gamma_{i}\underline{H}\gamma_{i}^{-1}$.
\end{itemize}
The two advantages are
\begin{itemize}
\item Method 2 can be applied to the case when $H$ is not reductive while Method 1 can only be applied to the reductive case (this is due to the fact that the relative truncation operator was only defined in the reductive case). For example, in our previous paper \cite{AWZ18}, we used Method 2 to study the period integrals of the Ginzburg-Rallis model, which is not reductive.
\item Even if $H$ is reductive, as we explained in the previous section, unless one can extend Prasad's result to all the spherical varieties, we can only apply Method 1 when $H$ is a symmetric subgroup. On the other hand, Method 2 can be applied to the non-symmetric case.
\end{itemize}
\end{rmk}

\begin{rmk}
We will use Method 1 to study the following five spherical pairs: $(\SO_{2n+1},\SO_{n+1}\times \SO_n)$, $(\SO_{2n},\SO_{n+1}\times \SO_{n-1})$, $(\RU_{2n},\RU_n\times \RU_n)$, $(\GE_6,A_1\times A_5)$, and the Jacquet-Guo model.  These pairs are all symmetric.  We use Method 2 to study the spherical pair $(\GL_4\times \GL_2,\GL_2\times \GL_2)$, which is not symmetric.  Method 2 can also be used to study the Jacquet-Guo model and the pair $(\GE_6,A_1\times A_5)$, but it will be more complicated than Method 1.
\end{rmk}

\subsection{A local result of Prasad}\label{section Prasad}
In this subsection, we recall a recent result of Prasad for the distinguished representations of symmetric pairs in \cite{P}. Let $F$ be a p-adic field of characteristic 0. Let $G$ be a quasi-split reductive group defined over $F$, $\theta$ be an involution automorphism of $G$ defined over $F$. Let $G^{\theta}$ be the group of fixed points, $H$ be the connected component of identity of $G^{\theta}$, and $H_1=[H,H]$ be the derived group of $H$.

We say $(G,H)$ is quasi-split if there exists a Borel subgroup $B$ of $G(\bar{F})$ such that $B\cap \theta(B)$ is a maximal torus of $G$.

\begin{thm}[Theorem 1 of \cite{P}]\label{prasad}
If $(G,H)$ is not quasi-split, then the Hom space
$$\Hom_{H_1(F)}(\pi,1)$$
is zero for all generic representations $\pi$ of $G(F)$. In other words, there is no $H_1(F)$-distinguished generic representation of $G(F)$.
\end{thm}

In Theorem 1(2) of \cite{P}, the author also gives an easy criterion for one to check whether $(G,H)$ is quasi-split by looking at the real form of $G(\BC)$ associated to the involution $\theta$. We refer the reader to Section 1 of \cite{P} for details. By that criterion, we can easily prove the following corollary.

\begin{cor}\label{prasad cor}
The following symmetric pairs are not quasi-split. In particular, there is no $H_1(F)$-distinguished generic representation of $G(F)$.
\begin{enumerate}
\item $G=\GL_{2n}$ and $H=\GL_{n+1}\times \GL_{n-1}$.
\item $G=\SO_{2n+3}$ the split odd orthogonal group, and $H=\SO_{n+k}\times \SO_{n+3-k}$ with $k\geq 3$.
\item $G=\SO_{2n+2}$ a quasi-split even orthogonal group, and $H=\SO_{n+k}\times \SO_{n+2-k}$ with $k\geq 3$.
\item $G=\SO_{2n}$ the split even orthogonal group ($n\geq 1$), and $H=\GL_n$ be the Levi subgroup of the Siegel parabolic subgroup of $G$.
\item $G=\Sp_{4n}$ and $H=\Sp_{2n}\times \Sp_{2n}$.
\item $G=\GE_6$ the similitude group of the split exceptional group $\mathrm{E_6}$, and $H$ be symmetric subgroup of $G$ of type $D_5\times \GL_1$.
\item $G=\mathrm{E_7^{sc}}$ be the split, simply-connected exceptional group, and $H$ be symmetric subgroup of $G$ of type $D_6\times A_1$.
\end{enumerate}
\end{cor}

\section{The model $(\SO_{2n+1},\SO_{n+1}\times \SO_n)$}\label{Section SO(2n+1)}
\subsection{The result}
Let $W_1$ (resp. $W_2$) be a quadratic space defined over $k$ of dimension $n+1$ (resp. $n$), and $W=W_1\oplus W_2$. Let $G=\SO(W)$ and $H=\SO(W_1)\times \SO(W_2)$. Let $D=Span\{v_{0,+},v_{0,-}\}$ be a two-dimensional quadratic space with $<v_{0,+},v_{0,+}>=<v_{0,-},v_{0,-}>=0$ and $<v_{0,+},v_{0,-}>=1$, $V_1=W_1\oplus D$, $V_2=W_2$, $V=V_1\oplus V_2$, $\underline{G}=\SO(V)$, and $\underline{H}=\SO(V_1)\times \SO(V_2)$.

Let $D_{+}=Span\{v_{0,+}\}$ and $D_{-}=Span\{v_{0,-}\}$. Then $D=D_{+}\oplus D_{-}$ as a vector space. Let $\underline{P}=\underline{M}\underline{N}$ be the maximal parabolic subgroup of $\ul{G}$ that stabilizes the subspace $D_{+}$ with $\underline{M}$ be the Levi subgroup that stabilizes the subspaces $D_{+},\; W$ and $D_{-}$. Then $\underline{M}=\SO(W)\times \GL_1$.

Let $\pi=\otimes_{v\in |k|} \pi_v$ be a cuspidal automorphic representation of $G(\BA)$. Then $\pi\otimes 1$ is a cuspidal automorphic representation of $\underline{M}(\BA)$. To simplify the notation, we will still use $\pi$ to denote this cuspidal automorphic representation. For $\phi\in \CA_{\pi}$ and $s\in \BC$, let $E(\phi,s)$ be the Eisenstein series on $\underline{G}(\BA)$. The goal of this section is to prove the following theorem.

\begin{thm}\label{SO(2n+1) main 1}
Assume that there exists a local non-archimedean place $v\in |k|$ such that $\pi_v$ is a generic representation of $G(k_v)$ (in particular, $G(k_v)$ is split). If the period integral $\CP_{H}(\cdot)$ is nonzero on the space of $\pi$, then there exists $\phi\in \CA_{\pi}$ such that the Eisenstein series $E(\phi,s)$ has a pole at $s=1/2$.
\end{thm}
Theorem \ref{SO(2n+1) main 1} will be proved in the last subsection of this section. The next proposition shows that Theorem \ref{SO(2n+1) main} follows from Theorem \ref{SO(2n+1) main 1}.

\begin{prop}\label{SO(2n+1) prop}
Theorem \ref{SO(2n+1) main 1} implies Theorem \ref{SO(2n+1) main}.
\end{prop}

\begin{proof}
We first recall the statement of Theorem \ref{SO(2n+1) main}. Let $G=\SO_{2n+1}$ be the split odd orthogonal group, $H=\SO_{n+1}\times \SO_{n}$ be a closed subgroup of $G$ (not necessarily split), and $\pi$ be a generic cuspidal automorphic representation of $G(\BA)$. If the period integral $\CP_{H}(\cdot)$ is nonzero on the space of $\pi$, we need to show that the L-function $L(s,\pi,\rho_X)$ is nonzero at $s=1/2$. Here $\rho_X$ is the standard representation of ${}^LG=\Sp_{2n}(\BC)$.

By the Theorem \ref{SO(2n+1) main 1}, if the period integral $\CP_{H}(\cdot)$ is nonzero on the space of $\pi$, there exists $\phi\in \CA_{\pi}$ such that the Eisenstein series $E(\phi,s)$ has a pole at $s=1/2$. In this case, the normalizing factor of the intertwining operator is
\[
\dfrac{L(s, \pi, \rho_X )\zeta_k(2s)}{L(s+1, \pi, \rho_X)\zeta_k(2s+1)}
\]
where $\zeta_k(s)$ is the Dedekind zeta function. By Theorem 4.7 of \cite{KK}, the normalized intertwining operator is holomorphic at $s=1/2$.
By Proposition \ref{L-function nonzero}, $L(3/2,\pi,\rho_X)\neq 0$. It follows that the numerator $L(s, \pi, \rho_X)\zeta_k(2s)$ has a pole at $s=1/2$, which implies that $L(\frac{1}{2},\pi,\rho_X)\neq 0$. This proves Theorem \ref{SO(2n+1) main}.
\end{proof}

\begin{rmk}
For a generic representation $\pi$ of $SO_{2n+1}$, the central value
$L(1/2, \pi)$ is linked to the so called Bessel periods by the Gan-Gross-Prasad conjecture \cite{GGP}
and has been studied in \cite{GJR04, GJR05}. In \cite{JS07a, JS07b} it is linked on the other hand to the first occurrence
problem in theta correspondence.
\end{rmk}

\subsection{The parabolic subgroups}
For $i=1,2$, we fix a maximal hyperbolic subspace (resp. anisotropic subspace) $W_{i,h}$ (resp. $W_{i,an}$) of $W_i$ such that $W_i=W_{i,h}\oplus W_{i,an}$. We fix a basis $\{w_{i,\pm1},\cdots,w_{i,\pm m_i}\}$ of $W_{i,h}$ such that
$$<w_{i,j},w_{i,k}>=<w_{i,-j},w_{i,-k}>=0,\;<w_{i,-j},w_{i,k}>=\delta_{jk},\;\forall 1\leq j\leq k\leq m_i.$$
Let $V_0=W_{1,an}\oplus W_{2,an}$. Fix a maximal hyperbolic subspace (resp. anisotropic subspace) $V_{0,h}$ (resp. $V_{0,an}$) of $V_0$ such that $V_0=V_{0,h}\oplus V_{0,an}$. We fix a basis $\{v_{0,\pm1},\cdots,v_{0,\pm l}\}$ of $V_{0,h}$ such that
$$<v_{0,j},v_{0,k}>=<v_{0,-j},v_{0,-k}>=0,\;<v_{0,-j},v_{0,k}>=\delta_{jk},\;\forall 1\leq j\leq k\leq l.$$
We use capital letters to denote the one-dimensional vector space spanned by vectors in small letters (e.g. $W_{i,1}=Span\{w_{i,1}\}$).

\begin{rmk}
$m_1,\;m_2$ and $l$ may be zero.
\end{rmk}

For $i=1,2$, let $P_{0,i}=M_{0,i}N_{0,i}$ be the parabolic subgroup of $\SO(W_i)$ that stabilizes the filtration
$$Span\{w_{i,1}\}\subset Span\{w_{i,1},w_{i,2}\}\subset \cdots \subset Span\{w_{i,1},\cdots,w_{i,m_i}\},$$
and $M_{0,i}$ be the subgroup of $\SO(W_i)$ that stabilizes the subspaces
$$W_{i,j},\;W_{i,-j},\; W_{i,an},\;\forall 1\leq j\leq m_i.$$
Let $A_{0,i}$ be the split center of $M_{0,i}$ which can be identified with $(\GL_1)^{m_i}$ under the natural isomorphism
$$A_{0,i}\simeq \GL(W_{i,1})\times \cdots \times\GL_{W_{i,m_i}}.$$
Then $P_{0,i}$ is a minimal parabolic subgroup of $\SO(W_i)$ and $A_{0,i}$ is a maximal split torus of $\SO(W_i)$.

Let $\ul{P}_{0,1}=\ul{M}_{0,1}\ul{N}_{0,1}$ be the parabolic subgroup of $\SO(V_1)$ that stabilizes the filtration
$$Span\{v_{0,+}\}\subset Span\{v_{0,+},w_{1,1}\}\subset \cdots \subset Span\{v_{0,+}, w_{1,1},\cdots,w_{1,m_1}\},$$
and $\ul{M}_{0,1}$ be the subgroup of $\SO(V_1)$ that stabilizes the subspaces
$$D_{+},W_{1,j},\;W_{1,-j},\; W_{1,an},\;\forall 1\leq j\leq m_1.$$
Let $\ul{A}_{0,1}$ be the split center of $\ul{M}_{0,1}$ which can be identified with $(\GL_1)^{m_1+1}$ under the natural isomorphism
$$\ul{A}_{0,1}\simeq \GL(D_{+})\times \GL(W_{1,1})\times \cdots \times\GL_{W_{1,m_i}}.$$
Then $\ul{P}_{0,1}$ is a minimal parabolic subgroup of $\SO(V_1)$ and $\ul{A}_{0,1}$ is a maximal split torus of $\SO(V_1)$ with $P_{0,1}\subset \ul{P}_{0,1}$ and $A_{0,1}\subset \ul{A}_{0,1}$.

On the other hand, let $\ul{P}_0=\ul{M}_0 N_0$ be the parabolic subgroup of $\ul{G}$ that stabilizes the filtration
$$Span\{v_{0,+}\}\subset Span\{v_{0,+},w_{1,1}\}\subset \cdots \subset Span\{v_{0,+},w_{1,1},\cdots,w_{1,m_1}\}\subset Span\{v_{0,+},w_{1,1},\cdots, w_{1,m_1},w_{2,1}\}$$
$$\subset \cdots \subset Span\{v_{0,+},w_{1,1},\cdots,w_{1,m_1},w_{2,1},\cdots, w_{2,m_2}\}\subset Span\{v_{0,+},w_{1,1},\cdots,w_{1,m_1},w_{2,1},\cdots,w_{2,m_2},v_{0,1}\}$$
$$\subset \cdots \subset Span\{v_{0,+},w_{1,1},\cdots,w_{1,m_1},w_{2,1},\cdots,w_{2,m_2},v_{0,1},\cdots v_{0,l}\}$$
and $\ul{M}_{0}$ be the subgroup of $\ul{G}$ that stabilizes the subspaces
$$D_{+},W_{i,j},\;W_{i,-j},\; V_{0,k},\;V_{0,an},\;\forall 1\leq i\leq 2, 1\leq j\leq m_i,1\leq k\leq l.$$
Let $\ul{A}_{0}$ be the split center of $\ul{M}_{0}$ which can be identified with $(\GL_1)^{m_1+m_2+l+1}$ under the natural isomorphism
$$\ul{A}_{0}\simeq \GL(D_{+})\times\GL(W_{1,1})\times \cdots \times \GL_{W_{1,m_1}} \times\GL(W_{2,1})\times \cdots \times \GL(W_{2,m_2})\times \GL_{V_{0,1}}\times \cdots \times \GL(V_{0,l}).$$
Then $\ul{P}_0$ is a minimal parabolic subgroup of $\ul{G}$ and $\ul{A}_{0}$ is a maximal split torus of $\ul{G}$ with $\ul{P}_0\subset \ul{P}$ and $\ul{A}_{0,1}\times A_{0,2}\subset \ul{A}_0$.

\begin{defn}
We use $\CF(\ul{M}_{0},\ul{P})$ to denote the set of semi-standard parabolic subgroups $\ul{Q}\in \CF(\ul{M}_0)$ of $\ul{G}$ that are conjugated to $\ul{P}$.
\end{defn}

The following proposition is a direct consequence of the definitions above.

\begin{prop}
Consider the set
$$X_{iso}=\{v_{0,\pm},\;w_{i,\pm j},v_{0,\pm k}|\; 1\leq i\leq 2,1\leq j\leq m_i,1\leq k\leq l\}.$$
For any element $w\in X_{iso}$, let $P_w$ be the stabilizers of $Span\{w\}$ in $\ul{G}$. Then $P_w\in \CF(\ul{M}_{0},\ul{P})$ and this gives us a natural bijection between the sets $\CF(\ul{M}_{0},\ul{P})$ and $X_{iso}$. Moreover, the parabolic subgroup $\ul{P}$ corresponds to the vector $v_{0,+}$ under this bijection.
\end{prop}

Let $\ul{P}_{0,H}=\ul{P}_{0,1}\times P_{0,2}$ be a minimal parabolic subgroup of $\ul{H}=\SO(V_1)\times \SO(V_2)=\SO(V_1)\times \SO(W_2)$. The following corollary is a direct consequence of the discussions above together with the definition of the set $\CF^{\ul{G}}(\ul{P}_{0,H},\ul{P})$ in Section \ref{section relative truncation}.

\begin{cor}\label{SO(2n+1) parabolic}
If $m_2=0$, then $\CF^{\ul{G}}(\ul{P}_{0,H},\ul{P})=\{\ul{P}\}$. If $m_2\neq 0$, then $\CF^{\ul{G}}(\ul{P}_{0,H},\ul{P})=\{\ul{P},\ul{P}'\}$ where $\ul{P}'$ is the parabolic subgroup corresponds to the vector $w_{2,1}$.
\end{cor}

To end this subsection, we discuss the intersections $\ul{P}\cap \ul{H}$ and $\ul{P'}\cap \ul{H}$. Let $P_1=M_1N_1$ be the maximal parabolic subgroup of $\SO(V_1)$ that stabilizes the space $D_{+}$, and $M_1$ be the subgroup of $\SO(V_1)$ that stabilizes the subspaces $D_{+},D_{-}$ and $W_1$. Then $M_1\simeq \SO(W_1)\times \GL(D_{+})$ and we have
$$\ul{P}\cap \ul{H} =P_1\times \SO(W_2),\; \ul{M}\cap \ul{H}=M_1\times \SO(W_2),\;\ul{N}\cap \ul{H}=N_1\times \{1\}.$$

For $\ul{P'}\cap \ul{H}$, let $\ul{P'}=\ul{M'}\ul{N'}$ where $\ul{M'}$ is the subgroup of $\ul{G}$ that stabilizes the subspaces $W_{2,1}$, $W_{2,-1}$ and $V_1\oplus W_2'$ where
$$W_2'=Span\{w_{2,\pm 2},\cdots,w_{2,\pm m_2}\}.$$
Let $P_2=M_2N_2$ be the maximal parabolic subgroup of $\SO(V_2)=\SO(W_2)$ that stabilizes the space $W_{2,1}$, and $M_2$ be the subgroup of $\SO(V_2)$ that stabilizes the subspaces $W_{2,1},W_{2,-1}$ and $W_2'$. Then $M_2\simeq \SO(W_2')\times \GL(W_{2,1})$ and we have
$$\ul{P}'\cap \ul{H} =\SO(V_1)\times P_2,\; \ul{M}'\cap \ul{H}=\SO(V_1)\times M_2,\;\ul{N}'\cap \ul{H}=\{1\}\times N_2.$$

\subsection{The proof of Theorem \ref{SO(2n+1) main 1}}
In this section, we will prove Theorem \ref{SO(2n+1) main 1}. We assume that $m_2\neq 0$, the proof for the case when $m_2=0$ is similar and much easier (this is due to the fact that the set $\CF(\ul{M}_{0},\ul{P})$ only contains one element when $m_2=0$, see Corollary \ref{SO(2n+1) parabolic}). We use Method 1 introduced in Section \ref{section method 1}.

With the same notations as in Theorem \ref{SO(2n+1) main 1} and Section \ref{section method 1}, we want to study the regularized period integral $\CP_{\ul{H},reg}(E(\phi,s))$ for $\phi\in \CA_{\pi}$. First, let's prove statement (1) and (2) in Section \ref{section method 1} for the current case. For (1), by our assumptions on $\pi$ together with the argument in Section \ref{section method 1}, it is enough to show that statement (4) of Section \ref{section method 1} holds for the pair $(\ul{G},\ul{H})$. But this just follows from Corollary \ref{prasad cor}(1). For (2), as we discussed in Remark \ref{local vanishing of orbit}, it is enough to show that the pair
$$(\ul{M}',\ul{M}'\cap \ul{H}')=(\SO(2n+1)\times \GL_1,\SO(n+3)\times \SO(n-2)\times \GL_1)$$
satisfies statement (5) in Remark \ref{local vanishing of orbit}. This also follows from Corollary \ref{prasad cor}(1).

Then we compute the constant $s_0=-c(1-2c_{\ul{P}}^{\ul{H}})$ for the current case. By Remark \ref{constant c(Q,H)}, we have
$$c_{\ul{P}}^{\ul{H}}=\frac{\dim(N_1)}{\dim{\ul{N}}}=\frac{n+1}{2n+1}.$$
By Proposition \ref{the constant c}, we have $c=\frac{2n+1}{2}$. This implies that
$$s_0=-c(1-2c_{\ul{P}}^{\ul{H}})=1/2.$$

Combining the discussions above, equation \eqref{2} in Method 1 becomes
$$\int_{[\underline{H}]} \Lambda^{T,\underline{H}} Res_{s=1/2}E(h, \phi, s)dh= \int_{K_{\underline{H}}}\int_{[H]} \phi(hk)dhdk-e^{\langle - \varpi_{\underline{P}}, T \rangle }\int_{K_{\underline{H}}}\int_{[H]} Res_{s=1/2}M(s)\phi(hk)dhdk$$
for the current case. This finishes the proof of Theorem \ref{SO(2n+1) main 1}.

\begin{rmk}
When $m_2\neq 0$ (i.e. when $W_2$ is not anisotropic), according to our discussion above, the set $\CF^{\ul{G}}(\ul{P}_{0,H},\ul{P})$ contains two elements $\ul{P}$ and $\ul{P}'$. It is easy to see that $c(1-2c_{\ul{P}'}^{\ul{H}})=\frac{2n+1}{2}(1-2\frac{n-2}{2n+1})=\frac{5}{2}>s_0=\frac{1}{2}$. This confirms the discussion in Section 3.2.1.
\end{rmk}

\section{The model $(\SO_{2n},\SO_{n+1}\times \SO_{n-1})$}\label{Section SO(2n)}
\subsection{The global result}
Let $W_1$ (resp. $W_2$) be a quadratic space defined over $k$ of dimension $n+1$ (resp. $n-1$), and $W=W_1\oplus W_2$. Let $G=\SO(W)$ and $H=\SO(W_1)\times \SO(W_2)$. Let $D=Span\{v_{0,+},v_{0,-}\}$ be a two-dimensional quadratic space with $<v_{0,+},v_{0,+}>=<v_{0,-},v_{0,-}>=0$ and $<v_{0,+},v_{0,-}>=1$, $V_1=W_1\oplus D$, $V_2=W_2$, $V=V_1\oplus V_2$, $\underline{G}=\SO(V)$, and $\underline{H}=\SO(V_1)\times \SO(V_2)$.

Let $D_{+}=Span\{v_{0,+}\}$ and $D_{-}=Span\{v_{0,-}\}$. Then $D=D_{+}\oplus D_{-}$ as a vector space. Let $\underline{P}=\underline{M} \underline{N}$ be the maximal parabolic subgroup of $G$ that stabilizes the subspace $D_{+}$ with $\underline{M}$ be the Levi subgroup that stabilizes the subspaces $D_{+},\; W$ and $D_{-}$. Then $\underline{M}=\SO(W)\times \GL_1$.

Let $\pi=\otimes_{v\in |k|} \pi_v$ be a cuspidal automorphic representation of $G(\BA)$. Then $\pi\otimes 1$ is a cuspidal automorphic representation of $\underline{M}(\BA)$. To simplify the notation, we will still use $\pi$ to denote this cuspidal automorphic representation. For $\phi\in \CA_{\pi}$ and $s\in \BC$, let $E(\phi,s)$ be the Eisenstein series on $\underline{G}(\BA)$.

\begin{thm}\label{SO(2n) main 1}
Assume that there exists a local non-archimedean place $v\in |k|$ such that $\pi_v$ is a generic representation of $G(k_v)$ (in particular, $G(k_v)$ is quasi-split). If the period integral $\CP_{H}(\cdot)$ is nonzero on the space of $\pi$, then there exists $\phi\in \CA_{\pi}$ such that the Eisenstein series $E(\phi,s)$ has a pole at $s=1$.
\end{thm}

\begin{proof}
The proof is very similar to the proof of Theorem \ref{SO(2n+1) main 1}, we will skip it here. The only thing worth to point out is that in the case of Theorem \ref{SO(2n+1) main 1}, the constant $-c(1-2c_{\underline{P}}^{\underline{H}})=-\frac{2n+1}{2}(1-\frac{2(n+1)}{2n+1})$ is equal to $1/2$ and this is why we can show that the Eisenstein series $E(\phi,s)$ has a pole at $s=1/2$. For the current case, the constant $-c(1-2c_{\underline{P}}^{\underline{H}})=-\frac{2n}{2}(1-\frac{2(n+1)}{2n})$ is equal to $1$. This is why we can show that the Eisenstein series $E(\phi,s)$ has a pole at $s=1/2$.
\end{proof}

\begin{rmk}
As in the previous case, when $W_2$ is not anisotropic, the set $\CF^{\ul{G}}(\ul{P}_{0,H},\ul{P})$ will contain two elements $\ul{P}$ and $\ul{P}'$. And one can easily show that $c(1-2c_{\ul{P}'}^{\ul{H}})=\frac{2n}{2}(1-2\frac{n-3}{2n})=3>s_0=1$. This confirms the discussion in Section 3.2.1.
\end{rmk}

\begin{rmk}
In \cite{GRS97} the existence of pole at $s=1$ of $L(s,\pi)$
for $\pi$ a generic cuspidal representation of $SO_{2n}$
is linked to a different (non-reductive period) and also to the so called first occurrence problem
in theta correspondence.
\end{rmk}

\begin{prop}
Theorem \ref{SO(2n) main 1} implies Theorem \ref{SO(2n) main}.
\end{prop}

\begin{proof}
We first recall the statement of Theorem \ref{SO(2n) main}. Let $G=\SO_{2n}$ be the split even orthogonal group, $H=\SO_{n+1}\times \SO_{n-1}$ be a closed subgroup of $G$ (not necessarily split), and $\pi$ be a generic cuspidal automorphic representation of $G(\BA)$. If the period integral $\CP_{H}(\cdot)$ is nonzero on the space of $\pi$, we need to show that the L-function $L(s,\pi,\rho_X)$ has a pole at $s=1$. Here $\rho_X$ is the standard representation of ${}^LG=\SO_{2n}(\BC)$.

By Theorem \ref{SO(2n) main 1}, if the period integral $\CP_{H}(\cdot)$ is nonzero on the space of $\pi$, there exists $\phi\in \CA_{\pi}$ such that the Eisenstein series $E(\phi,s)$ has a pole at $s=1$. In this case, the normalizing factor of the intertwining operator is
\[
\dfrac{L(s, \pi, \rho_X )}{L(s+1, \pi, \rho_X)}.
\]
By Theorem 4.7 of \cite{KK}, the normalized intertwining operator is holomorphic at $s=1$. By Proposition \ref{L-function nonzero}, $L(2,\pi,\rho_X)\neq 0$. It follows that the numerator $L(s, \pi, \rho_X)$ has a pole at $s=1$. This proves Theorem \ref{SO(2n) main}.
\end{proof}

\begin{rmk}
By the same argument, we can also prove Theorem \ref{SO(2n) main} when $G$ is quasi-split.
\end{rmk}

\subsection{The local result}\label{subsection SO(2n) local}
Let $F$ be a p-adic field, and $\ul{G},\ul{H},\ul{P}=\ul{M}\ul{N},G,H$ be the groups defined in the previous subsection. Let $\pi$ be an irreducible smooth representation of $G(F)$. We can view $\pi$ as an irreducible smooth representation of $\ul{M}(F)\simeq G(F)\times \GL_1(F)$ by making it trivial on $\GL_1(F)$. By abusing of notation, we still use $\pi$ to denote such representation. We also extend $\pi$ to $\ul{P}(F)$ by making it trivial on $\ul{N}(F)$. For $s\in \BC$, we use $\pi_s$ to denote the representation $\pi\otimes \varpi^{s}$. Here $\varpi=\varpi_{\ul{P}}\in \Fa_{\ul{M}}^{\ast}$ is the fundamental weight associated to $\ul{P}$, and $\varpi^s$ is the character of $\ul{M}(F)$ defined by
$$\varpi^s(m)=e^{\langle s\varpi, H_{\ul{M}}(m)\rangle},\;m\in \ul{M}(F).$$

Let $I_{\ul{P}}^{\ul{G}}(\cdot)$ be the normalized parabolic induction. In other words,

\begin{align*}I_{\ul{P}}^{\ul{G}}(\pi) =\{f:\ul{G}(F)\rightarrow \pi| &\; f\; \text{locally constant},\; f(nmg)=\delta_{\ul{P}}(m)^{1/2}\cdot \pi(m)f(g),\\ &\;\;\;\;
\forall m\in \ul{M}(F), n\in \ul{N}(F),\;g\in \ul{G}(F)\},\end{align*}

and the $\ul{G}(F)$-action is the right translation. The goal of this section is to prove the following theorem.

\begin{thm}\label{SO(2n) local theorem 1}
If $\pi$ is an irreducible representation of $G(F)$ such that the Hom space
$$\Hom_{H(F)}(\pi,1)$$
is nonzero, then the representation $I_{\ul{P}}^{\ul{G}}(\pi_{1})$ is $\ul{H}(F)$-distinguished, i.e. $\Hom_{\ul{H}(F)}(I_{\ul{P}}^{\ul{G}}(\pi_{1}),1)\neq \{0\}$.
\end{thm}

Before we prove the theorem, we first show that Theorem \ref{SO(2n) local theorem 1} implies Theorem \ref{SO(2n) local theorem}.

\begin{prop}\label{SO(2n) local reduction}
Theorem \ref{SO(2n) local theorem 1} implies Theorem \ref{SO(2n) local theorem}.
\end{prop}

\begin{proof}
Assume that $G$ is split. Let $\pi$ be a generic tempered representation of $G(F)$ such that the Hom space
$$\Hom_{H(F)}(\pi,1)$$
is nonzero, we need to show that the local L-function $L(s,\pi,\rho_X)$ has a pole at $s=0$. Here $\rho_X$ is the standard L-function of ${}^LG=\SO_{2n}(\BC)$.

By Theorem \ref{SO(2n) local theorem 1}, we know that the induced representation $I_{\ul{P}}^{\ul{G}}(\pi_{1})$ is $\ul{H}(F)$-distinguished. If $I_{\ul{P}}^{\ul{G}}(\pi_{1})$ is irreducible, then it is generic since $\pi$ is generic. On the other hand, by Corollary \ref{prasad cor}(1), we know that there is no $\ul{H}(F)$-distinguished generic representation of $\ul{G}(F)$. This is a contradiction and hence we know that $I_{\ul{P}}^{\ul{G}}(\pi_{1})$ is reducible.

By Lemma B.2 of \cite{GI} and the Standard Module Conjecture \cite{HO}, we have that $I_{\ul{P}}^{\ul{G}}(\pi_{1})$ is reducible if and only if the local gamma factor $\gamma(s,\pi,\rho_X)=\epsilon(s,\pi,\rho_X)\frac{L(1-s,\pi,\rho_X)}{L(s,\pi,\rho_X)}$ has a pole at $s=1$ (with respect to any non-trivial additive character $\psi$ since it doesn't change the existence or not of a pole at $s=1$). Since $\pi$ is tempered, $L(s,\pi,\rho_X)$ is holomorphic and nonzero when $Re(s)>0$ (Theorem 1.1 of \cite{HO}), which implies that the L-function $L(s,\pi,\rho_X)$ has a pole at $s=0$. This proves the proposition.
\end{proof}

\begin{rmk}
By the same argument, we can also prove Theorem \ref{SO(2n) local theorem} when $G$ is quasi-split.
\end{rmk}

For the rest of this section, we will prove Theorem \ref{SO(2n) local theorem 1}. Let $P_{\ul{H}}=\ul{H}\cap \ul{P},\;M_{\ul{H}}=\ul{H}\cap \ul{M},$ and $N_{\ul{H}}=\ul{H}\cap \ul{N}$. Then $P_{\ul{H}}=M_{\ul{H}}N_{\ul{H}}$ is a maximal parabolic subgroup of $\ul{H}$ with $M_{\ul{H}}\simeq H\times \GL_1$. We need two lemmas.

\begin{lem}\label{SO(2n) closed orbit}
$\ul{P}(F)\ul{H}(F)$ is a closed subset of $\ul{G}(F)$.
\end{lem}

\begin{proof}
This follows from Corollary \ref{closed orbit cor}.
\end{proof}

\begin{lem}\label{SO(2n) character}
We have the following equality of characters of $M_{\ul{H}}(F)$.
$$\delta_{\ul{P}}^{-1/2} \delta_{P_{\ul{H}}}=\varpi.$$
Here $\delta_{\ul{P}}$ and $\varpi$ are characters of $\ul{M}(F)$, and we view them as characters of $M_{\ul{H}}(F)$ by restriction.
\end{lem}

\begin{proof}
By the definition of the constants $c$ and $c_{\ul{P}}^{\ul{H}}$, we have
$$\delta_{\ul{P}}=\varpi^{2c},\;\delta_{P_{\ul{H}}}=\varpi^{2cc_{\ul{P}}^{\ul{H}}}.$$
This implies that
$$\delta_{\ul{P}}^{-1/2} \delta_{P_{\ul{H}}}=\varpi^{-c+2cc_{\ul{P}}^{\ul{H}}}=\varpi^{(-n)+2n\cdot \frac{n+1}{2n}}=\varpi.$$
\end{proof}

\begin{rmk}
The statement in the lemma above is equivalent to the equality
$$s_0=-c(1-2c_{\ul{P}}^{\ul{H}})=1.$$
\end{rmk}

Now we are ready to prove the theorem. Let $\ul{G}(F)_0=\ul{G}(F)-\ul{P}(F)\ul{H}(F)$. By Lemma \ref{SO(2n) closed orbit}, it is an open subset of $\ul{G}(F)$. We realize the representation $I_{\ul{P}}^{\ul{G}}(\pi_1)$ on the space
\begin{eqnarray*}
I_{\ul{P}}^{\ul{G}}(\pi_1)&=&\{f:\ul{G}(F)\rightarrow \pi|\; f\; \text{locally constant},\; f(nmg)=\delta_{\ul{P}}(m)^{1/2}\varpi(m)\cdot \pi(m)f(g),\\
&&\forall m\in \ul{M}(F), n\in \ul{N}(F),\;g\in \ul{G}(F)\}
\end{eqnarray*}
with the $G(F)$-action given by the right translation. Let $V'$ be the subspace of $I_{\ul{P}}^{\ul{G}}(\pi_1)$ consisting of all the functions whose support is contained in $\ul{G}(F)_0$. Then we know that $V'$ is $\ul{H}(F)$-invariant. Moreover, as a representation of $\ul{H}(F)$, we have
$$I_{\ul{P}}^{\ul{G}}(\pi_1)/V'\simeq ind_{P_{\ul{H}}}^{\ul{H}} (\delta_{\ul{P}}^{1/2} \varpi \pi)$$
where $ind$ is the compact induction. By the reciprocity law and Lemma \ref{SO(2n) character}, we have
$$\Hom_{\ul{H}(F)} (I_{\ul{P}}^{\ul{G}}(\pi_1)/V',1) \simeq \Hom_{P_{\ul{H}}(F)} (\delta_{\ul{P}}^{1/2} \varpi \pi,\delta_{P_{\ul{H}}})= \Hom_{P_{\ul{H}}(F)} (\pi,1)$$
$$= \Hom_{M_{\ul{H}}(F)} (\pi,1)=\Hom_{H(F)}(\pi,1)\neq \{0\}.$$
This implies that $\Hom_{\ul{H}(F)}(I_{\ul{P}}^{\ul{G}}(\pi_{1}),1)\neq \{0\}$ and finishes the proof of Theorem \ref{SO(2n) local theorem 1}.

\section{The model $(\RU_{2n},\RU_n\times \RU_n)$}\label{Section U(2n)}
Let $k'/k$ be a quadratic extension, $W_1$ and $W_2$ be two Hermitian spaces of dimension $n$, and $W=W_1\oplus W_2$. Let $G=\RU(W)$ and $H=\RU(W_1)\times \RU(W_2)$. Let $D=Span\{v_{0,+},v_{0,-}\}$ be a two-dimensional Hermitian space with $<v_{0,+},v_{0,+}>=<v_{0,-},v_{0,-}>=0$ and $<v_{0,+},v_{0,-}>=1$, $V_1=W_1\oplus D$, $V_2=W_2$, $V=V_1\oplus V_2$, $\underline{G}=\RU(V)$, and $\underline{H}=\RU(V_1)\times \RU(V_2)$.

Let $D_{+}=Span\{v_{0,+}\}$ and $D_{-}=Span\{v_{0,-}\}$. Then $D=D_{+}\oplus D_{-}$ as a vector space. Let $\underline{P}=\underline{M}\underline{N}$ be the maximal parabolic subgroup of $\ul{G}$ that stabilizes the subspace $D_{+}$ with $\underline{M}$ be the Levi subgroup that stabilizes the subspaces $D_{+},\; W$ and $D_{-}$. Then $\underline{M}=\RU(W)\times Res_{k'/k}\GL_1$.

Let $\pi=\otimes_{v\in |k|} \pi_v$ be a cuspidal automorphic representation of $G(\BA)$ with trivial central character. Then $\pi\otimes 1$ is a cuspidal automorphic representation of $\underline{M}(\BA)$. To simplify the notation, we will still use $\pi$ to denote this cuspidal automorphic representation. For $\phi\in \CA_{\pi}$ and $s\in \BC$, let $E(\phi,s)$ be the Eisenstein series on $\underline{G}(\BA)$. Let $\Pi$ be the base change of $\pi$ to $\GL_{2n}(\BA_{k'})$.

\begin{thm}\label{U(2n) main 1}
Assume that there exists a local non-archimedean place $v\in |k|$ such that $\pi_v$ is a generic representation of $G(k_v)$. If the period integral $\CP_{H}(\cdot)$ is nonzero on the space of $\pi$, then there exists $\phi\in \CA_{\pi}$ such that the Eisenstein series $E(\phi,s)$ has a pole at $s=1/2$.
\end{thm}

\begin{proof}
The proof is very similar to the orthogonal group case (Theorem \ref{SO(2n+1) main 1}), we will skip it here. The only thing worth to mention is the computation of the constant $s_0=-c(1-2c_{\underline{P}}^{\underline{H}})$. By Proposition \ref{the constant c}(3), we have $c=\frac{2n+1}{2}$. On the other hand, although the unipotent radical $\ul{N}$ is not abelian in this case, it is easy to see that $c_{\underline{P}}^{\underline{H}}=\frac{n+1}{2n+1}$. As a result, we have
$$s_0=-c(1-2c_{\underline{P}}^{\underline{H}})=-\frac{2n+1}{2}(1-\frac{2(n+1)}{2n+1})=1/2.$$
\end{proof}

\begin{rmk}
As in the previous cases, when $W_2$ is not anisotropic, the set $\CF^{\ul{G}}(\ul{P}_{0,H},\ul{P})$ will contain two elements $\ul{P}$ and $\ul{P}'$. And one can easily show that $c(1-2c_{\ul{P}'}^{\ul{H}})=\frac{2n+1}{2}(1-2\frac{n-1}{2n+1})=\frac{3}{2}>s_0=\frac{1}{2}$. This confirms the discussion in Section 3.2.1.
\end{rmk}

The next proposition proves the first part of Theorem \ref{U(2n) main}.

\begin{prop}
Assume that $G$ is quasi-split and $\pi$ is generic. If the period integral $\CP_{H}(\cdot)$ is nonzero on the space of $\pi$, then the standard L-function $L(s,\pi)$ is nonzero at $s=1/2$.
\end{prop}

\begin{proof}
By the Theorem \ref{U(2n) main 1}, if the period integral $\CP_{H}(\cdot)$ is nonzero on the space of $\pi$, there exists $\phi\in \CA_{\pi}$ such that the Eisenstein series $E(\phi,s)$ has a pole at $s=1/2$. In this case, the normalizing factor of the intertwining operator is (Section 2.1 and 2.2 of \cite{KK})
\[
\dfrac{L(s, \pi)\zeta_k(2s)}{L(s+1, \pi)\zeta_k(2s+1)}
\]
where $\zeta_k(s)$ is the Dedekind zeta function. By Theorem 4.7 of \cite{KK}, the normalized intertwining operator is holomorphic at $s=1/2$.
By Proposition \ref{L-function nonzero}, $L(3/2,\pi)\neq 0$. It follows that the numerator $L(s, \pi)\zeta_k(2s)$ has a pole at $s=1/2$, which implies that $L(\frac{1}{2},\pi)\neq 0$. This proves Theorem \ref{SO(2n+1) main}.
\end{proof}

Now it remains to prove the second part of Theorem \ref{U(2n) main}. We first recall the statement. Assume that $\Pi$ is cuspidal. Also assume that there exists a local place $v_0\in |k|$ such that $k'/k$ splits at $v_0$ and $\pi_{v_0}$ is a discrete series of $G(k_{v_0})=\GL_{2n}(k_{v_0})$. We need to show that if the period integral $\CP_{H}(\cdot)$ is nonzero on the space of $\pi$, then the exterior square L-function $L(s,\Pi,\wedge^2)$ has a pole at $s=1$ (i.e. $\Pi$ is of symplectic type).

We first show that $\Pi$ is self-dual. Let $\pi=\otimes_{v\in |k|} \pi_v$. By the automorphic Cebotarev density theorem proved in \cite{R15}, in order to show that $\Pi$ is self-dual, it is enough to show that $\pi_v$ is self-dual for all the non-archimedean places $v\in |k|$ such that the quadratic extension $k'/k$ splits at $v$. We fix such a local place $v$. Then $\pi_v$ is an irreducible smooth representation of $G(k_v)=\GL_{2n}(k_v)$. Since the period integral $\CP_{H}(\cdot)$ is nonzero on the space of $\pi$, we know that locally the Hom space
$$\Hom_{H(k_v)}(\pi_v,1)$$
is nonzero. By Theorem 1.1 of \cite{JR96}, we know that $\pi_v$ is self-dual. This proves that $\Pi$ is self-dual. Since $\Pi$ is cuspidal, this implies that $\Pi$ is either of symplectic type or of orthogonal type.

Now in order to show that $\Pi$ is of symplectic type, it is enough to show that at the split place $v_0\in |k|$, $\pi_{v_0}$ is not of orthogonal type. By our assumption, $\pi_{v_0}$ is a discrete series of $G(k_{v_0})=\GL_{2n}(k_{v_0})$, hence it is enough to show that $\pi_{v_0}$ is of symplectic type (this is because a discrete series of $\GL_{2n}$ can not be of symplectic type and orthogonal type at the same time).

By the discussion above, we know that the Hom space $\Hom_{H(k_{v_0})}(\pi_{v_0},1)$ is nonzero. In other words, $\pi_{v_0}$ is distinguished by the linear model. By Theorem 5.1 of \cite{M14}, we know that $\pi_{v_0}$ is distinguished by the Shalika model. Then by Proposition 3.4 of \cite{LM}, we know that $\pi_{v_0}$ is of symplectic type. This finishes the proof of Theorem \ref{U(2n) main}.

\section{The Jacquet-Guo model}\label{section Jacquet-Guo}
\subsection{The global result}
Let $k'=k(\sqrt{\alpha})$ be a quadratic extension of $k$ with $i=\sqrt{\alpha}$. Let $W$ be a $k'$-vector space of dimension $2n$. Fix a basis $\{w_1,\cdots,w_{2n}\}$ of $W$. We define a nondegenerate skew-symmetric $k'$-bilinear form $B$ on $W$ to be
$$B(w_j,w_k)=\delta_{j+k-1,2n},\;B(w_l,w_k)=-\delta_{l+k-1,2n},\; 1\leq j\leq n, n+1\leq l\leq 2n, 1\leq k\leq 2n.$$
In other words, in terms of the basis $\{w_1,\cdots,w_{2n}\}$, $B$ is defined by the skew-symmetric matrix
$$J_{2n}=\begin{pmatrix}0&w_n\\-w_n&0\end{pmatrix}$$
where $w_n$ is the longest Weyl element in $\GL_n$. Then we define the symplectic group $\ul{H}=\Sp(W,B)$. In other words, $\ul{H}=Res_{k'/k}\Sp_{2n}$.

Now we define the group $\ul{G}$. View $W$ as a $k$-vector space of dimension $4n$. Then $\{w_1,iw_1,\cdots,w_{2n},iw_{2n}\}$ is a basis of $W$. We define a non-degenerate skew-symmetric $k$-bilinear form $B_k$ on $W$ to be
$$B_k(v_1,v_2)=\tr_{k'/k}(B(v_1,v_2)),\;v_1,v_2\in W.$$
Then we define $\ul{G}=\Sp(W,B_F)$ (i.e. $\ul{G}=\Sp_{4n}$). We have $\ul{H}\subset \ul{G}$. For $1\leq j\leq 2n$, let $W_j$ be the $k'$-subspace of $W$ spanned by $w_j$, $W_{j,+}$ (resp. $W_{j,-}$) be the $k$-subspace of $W$ spanned by $w_j$ (resp. $iw_j$).

Let $\ul{P}=\ul{M}\ul{N}$ be the Siegel parabolic subgroup of $\ul{G}$ that stabilizes the $k$-subspace $Span_k\{w_j,iw_j|\;1\leq j\leq n\}$, and $\ul{M}$ be the Levi subgroup that stabilizes the $k$-subspaces $Span_k\{w_j,iw_j|\;1\leq j\leq n\}$ and $Span_k\{w_j,iw_j|\;n+1\leq j\leq 2n\}$. On the mean time, let $P_{\ul{H}}=M_{\ul{H}}N_{\ul{H}}$ be the Siegel parabolic subgroup of $\ul{H}$ that stabilizes the $k'$-subspace $Span_{k'}\{w_1,\cdots,w_{n}\}$, and $\ul{M}_{\ul{H}}$ be the Levi subgroup that stabilizes the $k'$-subspaces $Span_{k'}\{w_1,\cdots,w_{n}\}$ and $Span_{k'}\{w_{n+1},\cdots,w_{2n}\}$. Then it is easy to see that $\ul{P}\cap \ul{H}=P_{\ul{H}},\;\ul{M}\cap \ul{H}=M_{\ul{H}}$ and $\ul{N}\cap \ul{H}=N_{\ul{H}}$. We let $G=\ul{M}=\GL_{2n}$ and $H=M_{\ul{H}}=Res_{k'/k}\GL_{n}$.

Let $\pi=\otimes_{v\in |k|} \pi_v$ be a cuspidal automorphic representation of $G(\BA)=\ul{M}(\BA)$ with trivial central character. For $\phi\in \CA_{\pi}$ and $s\in \BC$, let $E(\phi,s)$ be the Eisenstein series on $\underline{G}(\BA)$. The goal of this section is to prove the following theorem.

\begin{thm}\label{Jacquet-Guo main 1}
If the period integral $\CP_{H}(\cdot)$ is nonzero on the space of $\pi$, then there exists $\phi\in \CA_{\pi}$ such that the Eisenstein series $E(\phi,s)$ has a pole at $s=1/2$.
\end{thm}

\begin{rmk}
Since $G=\GL_{2n}$, all the cuspidal automorphic representations of $G(\BA)$ are generic.
\end{rmk}

Theorem \ref{Jacquet-Guo main 1} will be proved in Section \ref{proof of Jacquet-Guo}. The next proposition shows that Theorem \ref{Jacquet-Guo main} follows from Theorem \ref{Jacquet-Guo main 1}.

\begin{prop}\label{Jacquet-Guo prop}
Theorem \ref{Jacquet-Guo main 1} implies Theorem \ref{Jacquet-Guo main}.
\end{prop}

\begin{proof}
We need to show that if the period integral $\CP_{H}(\cdot)$ is nonzero on the space of $\pi$, then the L-function $L(s,\pi,\rho_{X,1})$ is nonzero at $s=1/2$ and the L-function $L(s,\pi,\rho_{X,2})$ has a pole at $s=1$. Here $\rho_{X,1}$ (resp. $\rho_{X,2}$) is the standard representation (resp. exterior square representation) of ${}^L\ul{M}={}^LG=\GL_{2n}(\BC)$.

By Theorem \ref{Jacquet-Guo main 1}, if the period integral $\CP_{H}(\cdot)$ is nonzero on the space of $\pi$, there exists $\phi\in \CA_{\pi}$ such that the Eisenstein series $E(\phi,s)$ has a pole at $s=1/2$. In this case, the normalizing factor of the intertwining operator is
$$\frac{L(s,\pi,\rho_{X,1})L(2s,\pi,\rho_{X,2})}{L(s+1,\pi,\rho_{X,1})L(2s+1,\pi,\rho_{X,2})}.$$
By Theorem 4.7 of \cite{KK}, the normalized intertwining operator is holomorphic at $s=1/2$. By Proposition \ref{L-function nonzero}, $L(3/2,\pi,\rho_{X,1})L(2,\pi,\rho_{X,2})\neq 0$. It follows that the numerator $L(s,\pi,\rho_{X,1})L(2s,\pi,\rho_{X,2})$ has a pole at $s=1/2$, which implies that the L-function $L(s,\pi,\rho_{X,1})$ is nonzero at $s=1/2$ and the L-function $L(s,\pi,\rho_{X,2})$ has a pole at $s=1$.
\end{proof}

\subsection{The parabolic subgroups}
Let $\ul{B}=\ul{A}_0\ul{N}_0$ be the Borel subgroup of $\ul{G}$ that stabilizes the filtration
$$Span_{k}\{w_1\}\subset Span_{k}\{w_1,iw_1\}\subset Span_{k}\{w_1,w_2,iw_1\}\subset \cdots \subset Span_{k}\{w_1,\cdots,w_{n},iw_1,\cdots,iw_n\},$$
and $\ul{A}_0$ be the maximal torus of $\ul{G}$ that stabilizes the subspaces $W_{j,+}$ and $W_{j,-}$ for $1\leq j\leq 2n$. Then we have $\ul{B}\subset \ul{P}$ and $\ul{A}_0\subset \ul{M}$. We identify $\ul{A}_0$ with $(\GL_1)^{2n}$ under the natural isomorphism
\begin{equation}\label{torus Jacquet-Guo}
\ul{A}_0 \simeq \GL(W_{1,+})\times \GL(W_{1,-})\times \cdots \times \GL(W_{n,+})\times \GL(W_{n,-}).
\end{equation}
As in the $(\SO_{2n+1},\SO_{n+1}\times \SO_n)$-case, we use $\CF(\ul{A}_{0},\ul{P})$ to denote the set of semi-standard parabolic subgroups $\ul{Q}\in \CF(\ul{A}_0)$ of $\ul{G}$ that are conjugated to $\ul{P}$. The following proposition gives a description of the set $\CF(\ul{A}_{0},\ul{P})$.

\begin{prop}\label{parabolic Jacquet-Guo}
Let $S$ be the set $\{(a_1,\cdots,a_{2n})|\;a_j=\pm 1\}$. Then there is a natural bijection between $S$ and $\CF(\ul{A}_{0},\ul{P})$ given as follows. For $\ul{a}=(a_1,\cdots,a_{2n})\in S$, $\ul{P}_{\ul{a}}$ will be the Siegel parabolic subgroup of $\ul{G}$ that stabilizes the $k$-subspace
$$Span_k\{v_{(-1)^{a_j}j}|\; |\;1\leq j\leq 2n\}.$$
Here for $1\leq j\leq 2n$, $v_j=w_{\frac{j}{2}}$ and $v_{-j}=w_{2n+1-\frac{j}{2}}$ if $j$ is even; $v_j=iw_{\frac{j+1}{2}}$ and $v_{-j}=iw_{2n+1-\frac{j+1}{2}}$ if $j$ is odd. In particular $\ul{P}$ corresponds to the element $(0,0,\cdots,0)$ in $S$.
\end{prop}

On the mean time, let $\ul{B}_H=\ul{T}_H \ul{N}_{0,H}$ be the Borel subgroup of $\ul{H}$ that stabilizes the filtration
$$Span_{k'}\{w_1\}\subset Span_{k'}\{w_1,w_2\}\subset \cdots \subset Span_{k'}\{w_1,\cdots,w_{n}\},$$
and $\ul{T}_{H}$ be the maximal torus of $\ul{G}$ that stabilizes the subspaces $W_{j}$ for $1\leq j\leq 2n$. Let $\ul{A}_{0,H}=\ul{T}_{H}\cap \ul{A}_0$. Then $\ul{A}_{0,H}$ is a maximal split torus of $\ul{H}$. Under the isomorphism \eqref{torus Jacquet-Guo}, we have
$$\ul{A}_{0,H} \simeq (\GL(W_{1,+})\times \GL(W_{1,-}))^{diag}\times \cdots \times (\GL(W_{n,+})\times \GL(W_{n,-}))^{diag}$$
where $(\GL(W_{j,+})\times \GL(W_{j,-}))^{diag}$ is the set of elements of $\GL(W_{j,+})\times \GL(W_{j,-})$ that act by scalar on $W_{j,+}\oplus W_{j,-}$ for $1\leq j\leq n$.

The following corollary is a direct consequence of the discussions above together with the definition of the set $\CF^{\ul{G}}(\ul{B}_H,\ul{P})$ in Section \ref{section relative truncation}.

\begin{cor}
With the notations above, the set $\CF^{\ul{G}}(\ul{B}_H,\ul{P})$ only contains one element $\ul{P}$.
\end{cor}

\begin{rmk}
The corollary above confirms the discussion in Section 3.2.1.
\end{rmk}

\subsection{The proof of Theorem \ref{Jacquet-Guo main 1}}\label{proof of Jacquet-Guo}
In this section, we will prove Theorem \ref{Jacquet-Guo main 1}. We use Method 1 introduced in Section \ref{section method 1}.

With the same notations as in Theorem \ref{Jacquet-Guo main 1} and Section \ref{section method 1}, we want to study the regularized period integral $\CP_{\ul{H},reg}(E(\phi,s))$ for $\phi\in \CA_{\pi}$. First, let's prove statement (1) in Section \ref{section method 1} for the current case (there is no need to prove statement (2) of Section \ref{section method 1} since in the current case the set $\CF^{\ul{G}}(\ul{B}_H,\ul{P})$ only contains one element $\ul{P}$). For (1), since $\pi$ is generic, together with the argument in Section \ref{section method 1}, it is enough to show that statement (4) of Section \ref{section method 1} holds for the pair $(\ul{G},\ul{H})$. But this just follows from Corollary \ref{prasad cor}(4).

Then we compute the constant $s_0=-c(1-2c_{\ul{P}}^{\ul{H}})$ for the current case. By Remark \ref{constant c(Q,H)}, we have
$$c_{\ul{P}}^{\ul{H}}=\frac{\dim(N_{\ul{H}})}{\dim{\ul{N}}}=\frac{n+1}{2n+1}.$$
By Proposition \ref{the constant c}, we have $c=\frac{2n+1}{2}$. This implies that
$$s_0=-c(1-2c_{\ul{P}}^{\ul{H}})=1/2.$$

Combining the discussions above, equation \eqref{2} in Method 1 becomes

\begin{align*}\int_{[\underline{H}]} \Lambda^{T,\underline{H}} Res_{s=1/2}E(h, \phi, s)dh &= \int_{K_{\underline{H}}}\int_{[H]/Z_G(\BA)} \phi(hk)dhdk \\ &\;\;\;-e^{\langle - \varpi_{\underline{P}}, T \rangle } \int_{K_{\underline{H}}}\int_{[H]/Z_G(\BA)} Res_{s=1/2}M(s)\phi(hk)dhdk\end{align*}

for the current case. This finishes the proof of Theorem \ref{Jacquet-Guo main 1}.

\subsection{The local result}
Let $F$ be a p-adic field, and $E/F$ be a quadratic extension. As in the previous subsections, we can define the groups $\ul{G},\ul{H},\ul{P}=\ul{M}\ul{N},G,H$ over $F$. Let $\pi$ be an irreducible smooth representation of $G(F)=\ul{M}(F)=\GL_{2n}(F)$. We extend $\pi$ to $\ul{P}(F)$ by making it trivial on $\ul{N}(F)$. As in Section \ref{subsection SO(2n) local}, for $s\in \BC$, we use $\pi_s$ to denote the representation $\pi\otimes \varpi^{s}$ and use $I_{\ul{P}}^{\ul{G}}(\pi_s)$ to denote the normalized parabolic induction.

\begin{thm}\label{Jacquet-Guo local theorem 1}
If $\pi$ is an irreducible representation of $G(F)$ such that the Hom space
$$\Hom_{H(F)}(\pi,1)$$
is nonzero, then the representation $I_{\ul{P}}^{\ul{G}}(\pi_{\frac{1}{2}})$ is $\ul{H}(F)$-distinguished, i.e. $\Hom_{\ul{H}(F)}(I_{\ul{P}}^{\ul{G}}(\pi_{\frac{1}{2}}),1)\neq \{0\}$.
\end{thm}

Theorem \ref{Jacquet-Guo local theorem 1} will follows from the exact same argument as the proof of Theorem \ref{SO(2n) local theorem 1} once we have proved the following lemma which is the analogue $\ref{SO(2n) character}$.

\begin{lem}\label{Jacquet-Guo character}
We have the following equality of characters of $M_{\ul{H}}(F)$.
$$\delta_{\ul{P}}^{-1/2} \delta_{P_{\ul{H}}}=\varpi^{1/2}.$$
Here $\delta_{\ul{P}}$ and $\varpi$ are characters of $\ul{M}(F)$, and we view them as characters of $M_{\ul{H}}(F)$ by restriction.
\end{lem}

\begin{proof}
By the same argument as in the proof of Lemma \ref{SO(2n) character}, we have
$$\delta_{\ul{P}}^{-1/2} \delta_{P_{\ul{H}}}=\varpi^{-c+2cc_{\ul{P}}^{\ul{H}}}=\varpi^{\frac{-2n-1}{2}+\frac{2(2n+1)}{2}\cdot \frac{n+1}{2n+1}}=\varpi^{1/2}.$$
This proves the lemma.
\end{proof}

Now we are ready to prove Theorem \ref{Jacquet-Guo local theorem}. Let $\pi$ be a tempered representation of $G(F)$ with trivial central character (in particular, $\pi$ is generic since $G=\GL_{2n}$). Assume that the Hom space $\Hom_{H(F)}(\pi,1)$ is nonzero, we need to show that the local exterior square L-function $L(s,\pi,\rho_{X,2})$ has a pole at $s=0$. By the same argument as in Proposition \ref{SO(2n) local reduction}, we know that the induced representation $I_{\ul{P}}^{\ul{G}}(\pi_{\frac{1}{2}})$ is reducible. Again by applying Lemma B.2 of \cite{GI} and the Standard Module Conjecture \cite{HO}, we have that $I_{\ul{P}}^{\ul{G}}(\pi_{\frac{1}{2}})$ is reducible if and only if the local gamma factor $$\gamma(s,\pi,\rho_{X,1})\gamma(2s,\pi,\rho_{X,2})=\epsilon(s,\pi,\rho_{X,1})\epsilon(2s,\pi,\rho_{X,1}) \frac{L(1-s,\pi,\rho_{X,1})L(1-2s,\pi,\rho_{X,2})}{L(s,\pi,\rho_{X,1})L(2s,\pi,\rho_{X,2})}$$
has a pole at $s=\frac{1}{2}$. Since $\pi$ is tempered, $L(s,\pi,\rho_{X,1})$ and $L(s,\pi,\rho_{X,2})$ are holomorphic and nonzero when $Re(s)>0$ (Theorem 1.1 of \cite{HO}), which implies that the L-function $L(s,\pi,\rho_{X,2})$ has a pole at $s=0$. This finishes the proof of Theorem \ref{Jacquet-Guo local theorem}.

\section{The model $(\mathrm{GE_6}, A_1\times A_5)$}\label{Section E_6}

\subsection{The result}
Fix a quaternion algebra $B$ over the number field $k$.  Denote by $J_B = H_3(B)$ the Hermitian $3 \times 3$ matrices over $B$.  Let $\Theta = B \oplus B$ be an octonion algebra over $k$ defined via the Cayley-Dickson construction and denote by $J_\Theta = H_3(\Theta)$ the Hermitian $3 \times 3$ matrices over $\Theta$.  Then $\dim_k J_B = 15$ and $\dim_k J_\Theta = 27$; $J_\Theta$ is the exceptional cubic norm structure.  The Cayley-Dickson construction induces an identification $J_\Theta = J_B \oplus B^3$.

Let $G = GE_6$ be the group preserving the cubic norm on $J_\Theta$ up to similitude.  Let
\[
H = (\GL_1(B) \times \GL_3(B))^0 = \{(x,g) \in \GL_1(B) \times \GL_3(B): n_B(x) = N_6(g)\}
\]
where $n_B$, resp. $N_6$, denotes the reduced norm on $B$ (of degree two), resp. on $M_3(B)$ (of degree six).  In this section, we will consider $H$-periods of cusp forms on $G$.

Denote by $\underline{G}$ the semisimple, simply-connected group of type $E_7$ defined in terms of $J_\Theta$.  Precisely, $\underline{G}$ is the group preserving Freudenthal's symplectic and quartic form on $W_\Theta = k \oplus J_\Theta \oplus J_\Theta^\vee \oplus k$.  Denote by $W_B:= k \oplus J_B \oplus J_B^\vee \oplus k$ the Freudenthal space associated to the cubic norm structure $J_B$.  We write elements of $W_B$ as ordered four-tuples $(a,b,c,d)$, so that $a, d \in k$, $b \in J_B$ and $c \in J_B^\vee$, and similarly for $W_\Theta$. The  $32$-dimensional space $W_B$ affords one of the half-spin representations of a group of type $D_6$.  The identification $J_\Theta = J_B \oplus B^3$ induces an identification $W_\Theta = W_B \oplus B^6$, which we will use to define $\underline{H}$ and the map $\underline{H} \rightarrow \underline{G}$.

In more detail, let $\underline{H}'$ be the subgroup of elements with similitude equal to $1$ of the group denoted $\widetilde{G}$ in \cite[Appendix A]{P17}.  The group $\underline{H}'$, by its definition in \emph{loc cit}, comes equipped with maps to $D_6^+$ and $U_6(B)$.  Here $D_6^+$ is the semisimple half-spin group of type $D_6$ whose defining representation is $W_B$ and $U_6(B)$ is the subgroup of $\GL_6(B)$ satisfying $g \left(\begin{array}{cc} & 1_3 \\ -1_3 & \end{array}\right) g^* = \left(\begin{array}{cc} & 1_3 \\ -1_3 & \end{array}\right)$ where $g^*$ is the transpose conjugate of $g$.  The group $\underline{H}'$ acts on $W_\Theta = W_B \oplus B^6$ through these two maps, preserving the decomposition. Denote by $B^1$ the subgroup of $\GL_1(B)$ with reduced norm equal to $1$. Let $B^1$ act on $W_\Theta$ by $x (w,v)= (w,xv)$ where $x \in B^1$, $w \in W_B$ and $v \in B^6$.  This action commutes with the action of $\underline{H}'$ on $W_\Theta$ because $\underline{H}'$ acts on the right of $B^6$.  We set $\underline{H} = B^1 \times \underline{H}'$.  The action of $\underline{H}$ on $W_\Theta$ defines a map $\underline{H} \rightarrow \underline{G}$.  This map has a diagonal $\mu_2$-kernel; we abuse notation and also let $\ul{H}$ denote the image of this map in $\ul{G}$.

Denote by $\ul{P}$ the subgroup of $\ul{G}$ that stabilizes the line $k(0,0,0,1)$ in $W_\Theta$.  The group $\ul{P}$ is a parabolic subgroup of $\ul{G}$ with reductive quotient of type $E_6$.  Denote by $\ul{M}$ the subgroup of $\ul{P}$ that also fixes the line $k(1,0,0,0)$.  Then $\ul{M}$ is a Levi subgroup of $\ul{P}$, and one has $\ul{P}=\ul{M}\ul{N}$ with the unipotent radical $\ul{N}$ of $\ul{P}$ abelian; in fact, $\ul{N} \simeq J_\Theta$.  The group $\ul{M}$ is isomorphic to the subgroup $\GL_1 \times GE_6$ consisting of pairs $(\delta,m)$ with $\delta^2 = \nu(m)$, where $\nu: GE_6 \rightarrow \GL_1$ is the similitude.  The map $\underline{M} \rightarrow \GL_1$ defined as $(\delta,m)\mapsto \delta$ is the fundamental weight of $\underline{M}$.  One has that
\[
\ul{M} \cap \ul{H} = \{(x,\delta,g) \in B^1 \times \GL_1 \times \GL_3(B):\delta^2 = N_6(g)\}.
\]
Note that the image of $\ul{M} \cap \ul{H}$ in $\mathrm{PGE}_6$ is contained in the image of $H$ in $\mathrm{PGE}_6$.

Let $\pi = \otimes_{v \in |k|}{\pi_v}$ be a cuspidal automorphic representation of $G(\BA)$ with trivial central character. It can also be viewed as a cuspidal automorphic representation of $\ul{M}(\BA)$ with trivial central character.  For $\phi \in \mathcal{A}_\pi$ and $s \in \BC$, let $E(\phi,s)$ be the Eisenstein series on $\underline{G}(\BA)$.  The goal of this section is to prove the following theorem.

\begin{thm}\label{thm:E7per}
Assume that there exists a local non-archimedean place $v \in |k|$ such that $G(k_v)$ is split and $\pi_v$ is a generic representation of $G(k_v)$.  If the period integral $\mathcal{P}_H(\cdot)$ is nonzero on the space of $\pi$, then there exists $\phi \in \mathcal{A}_\pi$ such that the Eisenstein series $E(\phi,s)$ has a pole at $s=1$.
\end{thm}

\begin{rmk}\label{rmk:M1vsM2}
We will prove the theorem using Method 1. As we explained in Section 3.3, it is also possible to prove the theorem by Method 2, but it is more complicated since we need to study all the orbits in the double coset $\ul{P}\back \ul{G}/\ul{H}$. However, if the quaternion algebra $B$ is not split, there are two elements in the double coset $\ul{P}\back \ul{G}/\ul{H}$, and one can prove the theorem by Method 2 even without the assumption of local genericity.
\end{rmk}

\begin{prop}
Theorem \ref{thm:E7per} implies Theorem \ref{E_6 main}.
\end{prop}

\begin{proof} Let us recall the statement of Theorem \ref{E_6 main}.  Assume that $G$ is split and $\pi$ is a generic cuspidal automorphic representation of $G$ with trivial central character.  Assume that the period integral $\mathcal{P}_H(\cdot)$ is nonzero on the space of $\pi$ and $L(2,\pi,\rho_X)\neq 0$.  We need to show that the $L$-function $L(s,\pi,\rho_X)$ has a pole at $s=1$.  Here $\rho_X$ is the a fundamental $27$-dimensional representation of $^LPGE_6 = E_6^{sc}(\C)$.
	
By Theorem \ref{thm:E7per}, if the period integral $\mathcal{P}_H(\cdot)$ is nonzero on the space of $\pi$, there exists $\phi \in \mathcal{A}_\pi$ such that the Eisenstein series $E(\phi,s)$, and thus the intertwining operator $M(s)$, has a pole at $s=1$. In this case, the normalizing factor of the intertwining operator is
\[
\dfrac{L(s, \pi, \rho_X )}{L(s+1, \pi, \rho_X)}.
\]
By Theorem 4.7 of \cite{KK}, the normalized intertwining operator is holomorphic at $s=1$. Since we have assumed that $L(2,\pi,\rho_X)\neq 0$, it follows that the numerator $L(s, \pi, \rho_X)$ has a pole at $s=1$. This proves Theorem \ref{E_6 main}.
\end{proof}

\subsection{The parabolic subgroups}
Let $A_{0,\ul{H}}$ be a maximal split torus of $\underline{H}$, $M_{0,\ul{H}}$ its centralizer in $\underline{H}$, and $P_{0,\ul{H}}$ a minimal parabolic of $\underline{H}$ with $M_{0,\ul{H}}$ as Levi subgroup.  We will specify a specific choice of $A_{0,\ul{H}}$ momentarily. Let $\ul{A}_{0}$ be a maximal split torus of $\ul{G}$ with $A_{0,\ul{H}}\subset \ul{A}_0$.  In case $B=M_2(k)$ is split, we will specify $A_{0,\ul{H}}$ below; in this case, $\ul{G}$ is also split and we have $A_{0,\ul{H}}=\ul{A}_0$. When $B$ is a division algebra, the choice of $\ul{A}_0$ will turn out to be irrelevant.

Let $\mathcal{F}(\ul{A}_{0},\underline{P})$ be the set of semistandard parabolic subgroups of $\underline{G}$ (i.e. parabolic subgroups that contain $\ul{A}_{0}$) which are conjugate to $\underline{P}$. Let $\mathcal{F}(M_{0,\ul{H}},\underline{P})$ be the set parabolic subgroups $\ul{Q}\in \mathcal{F}(\ul{A}_{0},\underline{P})$ such that $M_{0,\ul{H}}\subset \ul{Q}$. Finally, denote by $\mathcal{F}'(M_{0,\ul{H}},\ul{P})$ the set of parabolic subgroups $\ul{Q}$ of $\ul{G}$ conjugate to $\ul{P}$ such that $\ul{Q}$ contains $M_{0,\ul{H}}$.  Thus $\mathcal{F}(M_{0,\ul{H}},\underline{P}) \subseteq  \mathcal{F}'(M_{0,\ul{H}},\underline{P})$.  The purpose of this subsection is to make explicit the set $\mathcal{F}'(M_{0,\ul{H}},\underline{P})$.  In case the quaternion algebra $B$ is a division algebra, this set contains $8$ elements; in case $B = M_2(k)$ is the split quaternion algebra, this set contains $56$ elements.

We have inclusions $\GL_3 \rightarrow \Sp_6 \rightarrow \underline{H}$ where the first arrow is the Levi of the Siegel parabolic of $\Sp_6$ and the composite arrow $\GL_3 \rightarrow \underline{H}$ is defined in \cite[Appendix A]{P17} in the second paragraph of page 1428.  Denote by $T_6$ the image in $\underline{H}$ of the diagonal maximal torus of $\Sp_6$ or $\GL_3$.  In case $B$ is a division algebra, $T_6=A_{0,\ul{H}}$ is a maximal split torus of $\underline{H}$.  The action of $T_6$ on $W_{\Theta} = k \oplus J_\Theta \oplus J_\Theta^\vee \oplus k$ has $14$ weight spaces.  There are $8$ one-dimensional spaces of the form
\[
(1,0,0,0), (0,e_{ii},0,0), (0,0,e_{ii},0), (0,0,0,1)
\]
where $1 \leq i \leq 3$ and $e_{ii}$ is the element of $H_3(\Theta)$ with a $1$ in the $i^{th}$-position on the diagonal and $0$'s elsewhere.  Denote this above set of $8$ elements of $W_\Theta$ or $W_B$ by $X_{long}$.  The action of $T_6$ on $W_\Theta$ has $6$ eight-dimensional weight spaces of the form $(0,v_k(\Theta),0,0)$ and $(0,0,v_k(\Theta),0)$ for $1 \leq k \leq 3$.  Here $v_1(\Theta)$ is the eight-dimensional subspace of $H_3(\Theta)$ consisting of elements of the form $\left(\begin{array}{ccc} 0 & 0 & 0 \\ 0 &0& * \\ 0 & * & 0\end{array}\right)$ and analogously for $v_2(\Theta)$ and $v_3(\Theta)$.  We denote\footnote{The vector space $W_\Theta$ is the abelianization of the unipotent radical of the Heisenberg parabolic of a group of type $E_8$.  This group has a certain relative root system of type $F_4$; the elements of $X_{long}$ make up the long roots in $W_\Theta$ for this system, whereas the spaces in $X_{short}^\Theta$ make up the short root spaces, hence the ``long" and ``short" labels.} the set of these $6$ eight-dimensional spaces by $X_{short}^{\Theta}$.

Denote by $M_0^B$ the centralizer of $T_6$ in $\underline{H}$.  Note that $M_0^B$ contains the $B^1$-component of $\underline{H}=B^1 \times \underline{H}'$.  In case $B$ is a division algebra, $M_0^B = M_{0,\ul{H}}$.

\begin{lem}\label{lem:Bdiv} For $v \in X_{long}$, denote by $\underline{P}_v$ the subgroup of $\underline{G}$ that stabilizes the line $k v$.  If $B$ is a division algebra, then these eight subgroups $\underline{P}_v$ are the elements of $\mathcal{F}'(M_{0,\ul{H}},\underline{P})$.
\end{lem}
\begin{proof} The subgroups of $\underline{G}$ that are conjugate to $\underline{P}$ are those subgroups that stabilize a rank one line in $W_\Theta$; see, e.g., \cite[section 4.3]{P18}.  It is clear that the eight elements $v$ of $X_{long}$ are rank one, and thus each $\underline{P}_v$ is an element of $\mathcal{F}'(M_{0,\ul{H}},\underline{P})$.
	
The lemma now follows immediately from the following claim:
\begin{claim} Suppose $\ell$ is a rank one line in $W_\Theta$ stabilized by $M_0^B$.  Then $\ell = kv$ for some element $v \in X_{long}$.\end{claim}

To prove the claim, we use the action of $B^1 \subseteq M_0^B$.  Each of the $6$ weight spaces in $X_{short}^{\Theta}$ is a copy of the octonion algebra $\Theta$.  The action of $B^1$ breaks $\Theta$ into $B \oplus B$, where $B^1$ acts on the first $B$ by the identity and the second $B$ by left-multiplication.  Because $B$ is a division algebra, the reduced norm on $B$ is anisotropic, and thus no element of one of these $B$'s in $\Theta$ can be rank one.  Therefore, the only rank one lines in $W_\Theta$ stabilized by $M_0^B$ are spanned by elements in $X_{long}$.  This proves the claim, and with it, the lemma.\end{proof}

\begin{rmk} The subgroups $\ul{P}_v$ of Lemma \ref{lem:Bdiv} may be described as follows.  Recall the maximal diagonal torus $T_6 \subseteq \Sp_6$, which we consider inside of $\ul{G}$.  Recall that the Weyl group of $\Sp_6$ is $(\pm 1)^3 \rtimes S_3$.  Let $w_1, \ldots, w_8$ be elements of the normalizer of $T_6$ in $\Sp_6$ that correspond to the elements $(\pm 1)^3$ of the Weyl group.  Let $\lambda_i: \GL_1 \rightarrow T_6$ be the cocharacters $t \mapsto w_i(\mm{t1_3}{}{}{t^{-1}1_3})$.  Abusing notation, also denote by $\lambda_i$ the cocharacter $\GL_1 \rightarrow T_6 \subseteq A_{0,\ul{H}} \subseteq \ul{A}_0$, where the map $\GL_1 \rightarrow T_6$ is the one just described.  With this notation, the parabolic subgroups of Lemma \ref{lem:Bdiv} are those parabolic subgroups of $\ul{G}$ that correspond to the cocharacters $\lambda_i$.  Consequently, $\ul{P}_v$ contains the centralizer of the image of $\lambda_i$ for some $i$ that depends upon $v$.  In particular, note that each $\ul{P}_v$ contains $\ul{A}_0$ for any choice of maximal split torus $\ul{A}_0$ containing $A_{0,\ul{H}}$.  Thus, in the statement of Lemma \ref{lem:Bdiv}, one can replace the set $\mathcal{F}'(M_{0,\ul{H}},\underline{P})$ with $\mathcal{F}(M_{0,\ul{H}},\underline{P})$, if one so desires.\end{rmk}
	
We now suppose that $B = M_2(k)$ is the split quaternion algebra.  In this case, the decomposition $W_\Theta = W_B \oplus B^6$ is $W_\Theta = W_B \oplus M_{2,12}(k)$ where $M_{2,12}(k)$ denotes the $2 \times 12$ matrices with entries in $k$, $B^1 = \SL_2$ acts by multiplication on the left and $\underline{H}'$ acts by multiplication on the right of $M_{2,12}$ via its map to $O(12)$.

Consider again the split torus $T_6$ in $\underline{H}'$.  It acts on $W_B$ with again $14$ weight spaces, the $8$ one-dimensional spaces in $X_{long}$ and $6$ four-dimensional spaces $X_{short}^{B}$ of the form $(0,v_k(B),0,0)$ and $(0,0,v_k(B),0)$ for $k = 1,2,3$.  Here $v_1(B)$ is the four-dimensional subspace of $H_3(B)$ consisting of elements of the form $\left(\begin{array}{ccc} 0 & 0 & 0 \\ 0 &0& * \\ 0 & * & 0\end{array}\right)$ and analogously for $v_2(B)$ and $v_3(B)$.  We can\footnote{Note that the half-spin representation of $D_6$ is miniscule} and do choose a maximal torus $T'$ of $\underline{H}'$ so that
\begin{itemize}
	\item $T'$ contains $T_6$
	\item the action of $T'$ on $W_B$ breaks up each $v_k(B) \simeq B = M_2(k)$ into four weight spaces, which are spanned by the matrices $\mm{1}{0}{0}{0}, \mm{0}{1}{0}{0}, \mm{0}{0}{1}{0}, \mm{0}{0}{0}{1}$ of $M_2(k)$.
\end{itemize}
Finally, if $T''$ denotes the one-dimensional diagonal torus of $\SL_2 = B^1$, we set $\ul{A}_0=A_{0,\ul{H}} = T'' \times T'$.  This is a maximal split torus of $\underline{H}$ and $\underline{G}$. We also have $M_{0,\ul{H}}=A_{0,\ul{H}}$. The weight spaces for $A_{0,\ul{H}}$ on $W_{\Theta}$ are each one-dimensional, and consist of the $32$ weight spaces in $W_B \subseteq W_\Theta$ and the $24$ spaces in $M_{2,12}$ spanned by the coordinate entries.

The following lemma computes the elements of $\mathcal{F}(A_{0,\ul{H}},\underline{P})=\mathcal{F}(M_{0,\ul{H}},\underline{P})=\mathcal{F}'(M_{0,\ul{H}},\underline{P})$.

\begin{lem}\label{lem:Bsplit}
Let $w$ be one of the weight spaces for $A_{0,\ul{H}}$ on $W_\Theta$, and $\underline{P}_w$ the subgroup of $\underline{G}$ stabilizing $w$.  Then the $56$ subgroups $\underline{P}_w$ are exactly the elements of $\mathcal{F}(A_{0,\ul{H}},\underline{P})$.
\end{lem}

\begin{proof} It is immediate to check that each of these weight spaces is a rank one line in $W_\Theta$, and thus each $\underline{P}_w$ is conjugate to $\underline{P}$ and contains $A_{0,\ul{H}}$.  Conversely, because the weight spaces for $A_0$ on $W_\Theta$ are each one-dimensional and the parabolics of $\underline{G}$ conjugate to $\underline{P}$ stabilize rank one lines in $W_\Theta$, the $\underline{P}_w$ are all of the elements of $\mathcal{F}(M_{0,\ul{H}},\ul{P})$.\end{proof}

\subsection{Proof of Theorem \ref{thm:E7per}}
In this subsection, we prove Theorem \ref{thm:E7per} by using Method 1 introduced in Section \ref{section method 1}. We want to study the regularized period integral $\mathcal{P}_{\ul{H},reg}(E(\phi,s))$.

We first prove statements (1) and (2) in Section \ref{section method 1} for the current case.  For statement (1),  by our assumptions on $\pi$ together with the argument in Section \ref{section method 1}, it is enough to show that statement (4) of Section \ref{section method 1} holds for the pair $(\ul{G},\ul{H})$. But this just follows from Corollary \ref{prasad cor} (7).

To prove statement (2) of Section \ref{section method 1}, we fix $P_{0,\ul{H}}$ to be any minimal parabolic subgroup of $\underline{H}$ that contains $M_{0,\ul{H}}$ and fixes the line $k(0,0,0,1)$ in $W_B \subseteq W_\Theta$. We want to study the set $\CF^{\ul{G}}(P_{0,\ul{H}},\ul{P})$. By Proposition \ref{equal rank}, we have $\CF^{\ul{G}}(P_{0,\ul{H}},\ul{P})\subset \mathcal{F}(M_{0,\ul{H}},\underline{P}) \subseteq \mathcal{F}'(M_{0,\ul{H}},\underline{P})$. So it is enough to consider the elements in $\mathcal{F}'(M_{0,\ul{H}},\underline{P})$.

Suppose first that $B$ is a division algebra.  Because the elements $v$ of Lemma \ref{lem:Bdiv} are fixed by the $B^1$-factor of $\ul{H}$, one has $\ul{P}_v \cap \ul{H} = B^1 \times \left(\ul{P}_v \cap \ul{H}'\right)$.  Because the subgroups $\underline{P}_v$ of Lemma \ref{lem:Bdiv} stabilize rank one lines in $W_B$, the groups $\underline{P}_v \cap \underline{H}'$ are conjugate in $\ul{H}'$ and thus only one of the groups $\underline{P}_v \cap \underline{H}'$ can contain a fixed minimal parabolic of $\ul{H}'$.  In fact, the groups $\underline{P}_v \cap \underline{H}'$ give all 8 of the maximal semistandard parabolic subgroups of $\underline{H}'$ of type $A_5$ in one of the two conjugacy classes with $A_5$-type Levi.  Thus only one can be standard (i.e., the one corresponds to $\ul{P}$). By Proposition \ref{equal rank}(1), we have that $\CF^{\ul{G}}(P_{0,\ul{H}},\ul{P})\subseteq\{\ul{P}\}$.

Note that $\ul{P}$ is the parabolic associated to the cocharacter $\GL_1 \rightarrow T_6 \subseteq A_{0,\ul{H}} \subseteq \ul{A}_0$, with the map $\GL_1 \rightarrow T_6$ given by $t \mapsto \mm{t 1_3}{}{}{t^{-1}1_3}$.  Thus $\ul{P}$ contains the centralizer of the image of this cocharacter, and in particular, contains $\ul{A}_0$ for any choice of $\ul{A}_0$ containing $A_{0,\ul{H}}$.  Consequently, $\CF^{\ul{G}}(P_{0,\ul{H}},\ul{P})=\{\ul{P}\}$, as desired.  Hence statement (2) is trivial in this case.

Now suppose that $B = M_2(k)$ is split.  Recall from Lemma \ref{lem:Bsplit} that there are $56$ semistandard parabolic subgroups $\underline{P}_w$ of $\underline{G}$ conjugate to $\underline{P}$.  Moreover, these $56$ parabolic subgroups are partitioned into two sets, one of size $32$ and the other of size $24$ depending on whether the rank one line $w$ is contained in $W_B$ or $M_{2,12}(k)$.  For the $32$ $\underline{P}_w$ with $w$ contained in $W_B$, we have just as in the division algebra case that $\ul{P}_w \cap \ul{H} = B^1 \times (\ul{P}_w \cap \ul{H}')$ and that the groups $\underline{P}_w \cap \underline{H}'$ give all the maximal semistandard parabolic subgroups of $\underline{H}'$ in one of the two conjugacy classes with Levi of type $A_5$ (there are 32 of them).  Thus only one can be standard (i.e., the one that corresponds to $\ul{P}$).

Similarly, the $24$ $\underline{P}_w$ with $w$ contained in $M_{2,12} = V_2 \otimes V_{12}$ all satisfy that $\underline{P}_w \cap \underline{H} = \underline{P}_w \cap (\SL_2 \times \underline{H}')$ are stabilizers of pure tensors $b \otimes v$ with $b \in V_2$ the two-dimensional representation of $\SL_2$ and $v \in V_{12}$ an isotropic vector in the orthogonal $12$-dimensional representation of $\underline{H}'$.  Thus for any such $\underline{P}_w$, $\ul{P}_w \cap \ul{H} = B \times (\underline{P}_w \cap \underline{H}')$ with $B$ a semistandard Borel subgroup of $\SL_2$ (there are 2 of them) and $\underline{P}_w \cap \underline{H}'$ a maximal semistandard parabolic subgroup of $\ul{H}'$ of type $D_5$ (there are 12 of them). Thus only one can be standard.

Since $\ul{A}_0=A_{0,\ul{H}}$ when $B$ is split, by applying Proposition \ref{equal rank} again, we have $\CF^{\ul{G}}(P_{0,\ul{H}},\ul{P})=\{\ul{P},\ul{P}'\}$ such that $\ul{P}'\cap \ul{H}$ is the parabolic subgroup of type $D_5$. Then statement (2) follows from Corollary \ref{prasad cor}(6).

Finally, we compute the constant $s_0 = -c(1-2c_{\underline{P}}^{\underline{H}})$ for the current case.  We have $c=9$, and, because the unipotent radicals are abelian, $c_{\underline{P},\underline{H}} = \frac{\dim_k J_B}{\dim_k J_\Theta} = \frac{15}{27}.$  Thus $s_0 =1$.  Let $H^0 \subseteq H$ be the image of $\ul{M} \cap \ul{H}$ in $\GE_6$.  Combining the discussions above, equation \eqref{2} in Method 1 becomes
\[\int_{[\underline{H}]} \Lambda^{T,\underline{H}} Res_{s=1}E(h, \phi, s)dh = \int_{K_{\underline{H}}}\int_{[H^0]/Z_G(\BA)} \phi(hk)dhdk  -\frac{e^{\langle - 2\varpi_{\underline{P}}, T \rangle }}{2} \int_{K_{\underline{H}}}\int_{[H^0]/Z_G(\BA)} Res_{s=1}M(s)\phi(hk)dhdk\]
for the current case.  Because $H^0 \subseteq H$, this finishes the proof of Theorem \ref{thm:E7per}.

\begin{rmk}
When $\ul{G}$ and $\ul{H}$ are split, the set $\CF^{\ul{G}}(P_{0,\ul{H}},\ul{P})$ contains two elements $\ul{P}$ and $\ul{P}'$. One can easily show that $c(1-2c_{\underline{P}'}^{\underline{H}})=9(1-2\frac{10}{27})=\frac{5}{3}>s_0=1$. This confirms the discussion in Section 3.2.1.
\end{rmk}

\subsection{The local result}
Let $F$ be a p-adic field. As in the previous subsections, we can define the groups $\ul{G},\ul{H},\ul{P}=\ul{M}\ul{N},G,H$ over $F$. Let $\pi$ be an irreducible smooth representation of $G(F)$ with trivial central character. We can also view $\pi$ as a representation of $\ul{M}(F)$ with trivial central character. We then extend $\pi$ to $\ul{P}(F)$ by making it trivial on $\ul{N}(F)$. As in Section \ref{subsection SO(2n) local}, for $s\in \BC$, we use $\pi_s$ to denote the representation $\pi\otimes \varpi^{s}$ and use $I_{\ul{P}}^{\ul{G}}(\pi_s)$ to denote the normalized parabolic induction.

\begin{thm}
If $\pi$ is an irreducible representation of $G(F)$ with trivial central character. Assume that the Hom space $\Hom_{H(F)}(\pi,1)$ is nonzero, then the representation $I_{\ul{P}}^{\ul{G}}(\pi_{1})$ is $\ul{H}(F)$-distinguished.
\end{thm}

\begin{proof}
The proof follows from the exact same argument as the proof of Theorem \ref{SO(2n) local theorem 1}, we will skip it here.
\end{proof}

Now we are ready to prove Theorem \ref{E_6 local theorem}. Assume that $G$ is split over $F$. Let $\pi$ be a tempered generic representation of $G(F)$ with trivial central character. Assume that the Hom space $\Hom_{H(F)}(\pi,1)$ is nonzero, we need to show that the local L-function $L(s,\pi,\rho_{X})$ has a pole at $s=0$. By the same argument as in Proposition \ref{SO(2n) local reduction}, we know that the induced representation $I_{\ul{P}}^{\ul{G}}(\pi_{1})$ is reducible. Then by applying Lemma B.2 of \cite{GI} and the Standard Module Conjecture \cite{HO}, we have that $I_{\ul{P}}^{\ul{G}}(\pi_{1})$ is reducible if and only if the local gamma factor
$$\gamma(s,\pi,\rho_{X})=\epsilon(s,\pi,\rho_{X}) \frac{L(1-s,\pi,\rho_{X})}{L(s,\pi,\rho_{X})}$$
has a pole at $s=1$. Since $\pi$ is tempered, $L(s,\pi,\rho_{X})$ is holomorphic and nonzero when $Re(s)>0$ (Theorem 1.1 of \cite{HO}), which implies that the L-function $L(s,\pi,\rho_{X})$ has a pole at $s=0$. This finishes the proof of Theorem \ref{E_6 local theorem}.

\section{The model $(\GL_4\times \GL_2,\GL_2\times \GL_2)$}\label{section GL(4)GL(2)}
The purpose of this section is to prove the local and global results for the pair $(\GL_4\times \GL_2,\GL_2 \times \GL_2)$.  The first several subsections are concerned with the global results, while the final subsection concerns the local results.

\subsection{Overview of argument} The purpose of this and the following three subsections is to prove Theorem \ref{GL(4)GL(2)}.  We will do this by Method 2.  The argument is analogous to that of \cite{AWZ18}, but the computations are easier.  In this subsection, we give an overview of the argument used to prove Theorem \ref{GL(4)GL(2)}.

Denote by $E = k \times k$ the split quadratic \'etale extension of $k$.  In fact, almost all of this section is unchanged if $E$ is replaced by a general quadratic \'etale extension of $k$, so we frequently write $E$ instead of $k \times k$.\footnote{It is essentially for the purpose of proving the absolute convergence of certain integrals below that we choose $E$ split.}  Let $\Theta$ be a split octonion algebra over $k$.  Define the quadratic space $V = \Theta \oplus E$ with quadratic form $q(x,\lambda)= n_{\Theta}(x)-n_{E}(\lambda)$ where $x \in \Theta$, $\lambda \in E$, and $n_{\Theta}$ resp. $n_{E}$ denote the quadratic norms on $\Theta$ resp. $E$.  We define $\underline{G} = \GSO(V)$, which by definition is the subgroup of $\GO(V)$ consisting of those $g$ with $\det(g) = \nu(g)^{\dim(V)/2}$, where $\nu: \GO(V) \rightarrow \GL_1$ is the similitude.

In the next subsection, we specify a group $\underline{H}_7$ which is of type $\GSpin(7)$ together with its $8$-dimensional spin representation on $\Theta$.  Set
\begin{align*}
\underline{H} = \underline{H}_7 \boxtimes \mathrm{Res}_{E/k}(\GL_1) :&=\{(h,\lambda)\in \underline{H}_7 \times \mathrm{Res}_{E/k}(\GL_1): \nu(h) = n_{E}(\lambda)\} \\ &=\{(h,\lambda_1,\lambda_2) \in \ul{H}_7 \times \GL_1 \times \GL_1: \nu(h)=\lambda_1 \lambda_2\}.
\end{align*}
Via the representation $t_1: \GSpin(7) \rightarrow \GSO(\Theta)$ specified below, we obtain an inclusion $\underline{H} \rightarrow \underline{G}$.

Denote by $\underline{P} = \underline{M}\underline{N}$ the Heisenberg parabolic of $\underline{G}$, so that the Levi subgroup $\underline{M}$ of $\underline{P}$ is of type $A_1 \times D_3 = A_1 \times A_3$.  Suppose that $\pi = \otimes_{v \in |k|}{\pi_v}$ is a cuspidal automorphic representation of $\underline{M}$ or $G=\GL_2 \times \GL_4$ with trivial central character\footnote{Note that the exterior square representation of $\GL_4$ induces a map of algebraic groups $\GL_1 \times \GL_4 \rightarrow \GSO(6)$ which is surjective on $k$-points and has central kernel.  Thus this map induces an isomorphism $\mathrm{PGL}_4 \simeq \mathrm{PGSO}_6$, so that a cuspidal automorphic representation on $\GL_4$ with trivial central character may be considered as an automorhpic representation of $\GSO(6)$.}.  Suppose that $\phi \in \mathcal{A}_{\pi}$, $s \in \BC$ and $E(\phi,s)$ denotes the associated Eisenstein series on $\mathcal{A}_\pi(\underline{G}(\BA))$.

Let
\begin{align*} H &= (\GL_2 \times \GL_2) \boxtimes \mathrm{Res}_{E/k}(\GL_1) \\ &= \{(g,h,\lambda) \in \GL_2 \times \GL_2 \times \mathrm{Res}_{E/k}(\GL_1): \det(g)\det(h) = N_{E/k}(\lambda)\}.\end{align*}
Denote by $Z \simeq \GL_1 \times \GL_1$ the subgroup of $(\GL_2 \times \GL_2) \boxtimes \mathrm{Res}_{E/F}(\GL_1)$ consisting of the elements $(\mathrm{diag}(z,z),\mathrm{diag}(w,w),(zw,zw))$. From Lemma \ref{lem:gl2gl2} below we obtain an embedding $H \rightarrow \underline{M}$ so that $Z = H \cap Z_{\ul{M}}$ where $Z_{\ul{M}}$ is the center of $\ul{M}$.  Note that the map $H \rightarrow \GL_2 \times \GL_2$ given by $(g,h,(\lambda_1,\lambda_2)) \mapsto (g,\lambda_1^{-1}\det(g)h)$ induces an isomorphism $H/Z \simeq (\GL_2 \times \GL_2)/\Delta(\GL_1)$, with $\Delta(\GL_1)$ the diagonally embedded central $\GL_1$. For a cuspidal automorphic form $\varphi$ of $\underline{M}$ with trivial central character, denote by $\mathcal{P}_H$ the period
\[\mathcal{P}_H(\varphi)= \int_{Z(\BA)H(k)\backslash H(\BA)}{\varphi(h)\,dh}.\]

\begin{thm}\label{thm:GSO(10)} Suppose that the period $\mathcal{P}_H(\cdot)$ is nonvanishing on the space of $\pi$.  Then there exists $\phi \in \mathcal{A}_\pi$ such that $E(\phi,s)$ has a pole at $s=1/2$.\end{thm}

Note that even though $G$, $H$ and $\ul{G}$ are classical groups, our proof of Theorem \ref{thm:GSO(10)} proceeds through the non-classical group $\ul{H}_7 \simeq \GSpin(7)$. This is in complete analogy with the triple product period integral considered by Jiang in \cite{J98}, later generalized by Ginzburg-Jiang-Rallis in \cite{GJR04b}, where one considers $G_2$-periods of certain residual representations on groups of type $D$.

From Theorem \ref{thm:GSO(10)} we obtain Theorem \ref{GL(4)GL(2)} of the introduction.
\begin{prop}
Theorem \ref{thm:GSO(10)} implies Theorem \ref{GL(4)GL(2)}.
\end{prop}

\begin{proof}
We first recall the statement of Theorem \ref{GL(4)GL(2)}. Let $\pi$ be a cuspidal automorphic representation of $\GL_4(\BA)\times \GL_2(\BA)$ with trivial central character. Assume that the $\GL_2\times \GL_2$-period defined in Section 1.1.6 is nonzero on the space of $\pi$. Moreover assume that the L-function $L(s,\pi,\rho_X)$ is nonzero at $s=3/2$ where $\rho_X=\wedge^2\otimes std$ is a 12-dimensional representation of ${}^LG$. Then we need to show that the L-function $L(s,\pi,\rho_X)$ is nonzero at $s=1/2$.

As we explained in the previous page, we can view $\pi$ as a cuspidal automorphic representation of $\ul{M}(\BA)$ with trivial central character. Moreover, by the discussion in the end of Section 9.4, we know that the $\GL_2\times \GL_2$-period integral on $\pi$ (viewed as a cuspidal automorphic representation of $\GL_4(\BA)\times \GL_2(\BA)$) is just the $H$-period integral on $\pi$ (viewed as a cuspidal automorphic representation of $\ul{M}(\BA)$). As a result, we know that the period $\mathcal{P}_H(\cdot)$ is nonvanishing on the space of $\pi$. By Theorem \ref{thm:GSO(10)}, there exists $\phi \in \mathcal{A}_\pi$ such that the Eisenstein series $E(\phi,s)$, and thus the intertwining operator $M(s)$, has a pole at $s=1/2$.

In this case, the normalizing factor of the intertwining operator is
\[
\dfrac{L(s, \pi, \rho_X )\zeta_k(2s)}{L(s+1, \pi, \rho_X)\zeta_k(2s+1)}
\]
where $\zeta_k(s)$ is the Dedekind zeta function. By Theorem 4.7 of \cite{KK}, the normalized intertwining operator is holomorphic at $s=1/2$.
By Proposition \ref{L-function nonzero}, $L(3/2,\pi,\rho_X)\neq 0$. It follows that the numerator $L(s, \pi, \rho_X)\zeta_k(2s)$ has a pole at $s=1/2$, which implies that $L(\frac{1}{2},\pi,\rho_X)\neq 0$. This proves Theorem \ref{GL(4)GL(2)}.
\end{proof}

We now finish this subsection by explaining how Theorem \ref{thm:GSO(10)} is proved and the organization of the next three subsections. Denote by $\Lambda^T E(\phi,s)$ the Arthur-Langlands truncation of $E(\phi,s)$; see subsection \ref{subsec:trunc}.  As explained in subsection \ref{section method 2}, we will compute an $\underline{H}$-period of $\Lambda^T E(\phi,s)$, $\mathcal{P}_{\underline{H}}(\Lambda^T E(\phi,s))$ and essentially reduce the calculation to a period $\mathcal{P}_{H}(\phi)$.

More precisely, the proof of Theorem \ref{thm:GSO(10)} proceeds as follows.
\begin{enumerate}
	\item First, we prove that the double coset space $\underline{P}(k)\backslash \underline{G}(k) / \underline{H}(k)$ is finite.  Denote by $\gamma_1=1, \gamma_2, \ldots, \gamma_\ell$ its elements.
	\item Define $\underline{H}_i = \underline{H} \cap (\gamma_i^{-1}\underline{P} \gamma_i)$.  Then for each $i$, we prove that the pair $(\ul{H},\underline{H}_i)$ is a good pair in the sense of \cite[Section 5]{AWZ18}.
	\item Denote by $I_i(\phi,s)$ and $J_i(\phi,s)$ the integrals specified in Subsection \ref{section method 2}.  Applying Lemma 5.1, Proposition 5.2, and Proposition 5.5 of \cite{AWZ18}, one obtains that each of the finitely many integrals $I_i(\phi,s)$ and $J_i(\phi,s)$ converges absolutely for $Re(s)$ and $T$ sufficiently large.
	\item Finally, we prove that the integrals $I_i(\phi,s)$ and $J_i(\phi,s)$ vanish if $i > 1$.
	\item Computing the integrals $I_1(\phi,s)$ and $J_1(\phi,s)$, we obtain \eqref{eqn:method2Fin} with $s_0 = 1/2$, from which Theorem \ref{thm:GSO(10)} follows immediately.
\end{enumerate}

In the next subsection, we define the group $\underline{H}_7$ precisely, its representation $t_1: \underline{H}_7 \rightarrow \GSO(\Theta)$, and some special subgroups of it.  In subsection \ref{subsec:orbitStab} we prove that the double coset $\underline{P}(k)\backslash \underline{G}(k) / \underline{H}(k)$ is finite and that $(\ul{H},\underline{H}_i)$ is a good pair for all $i$.  In subsection \ref{subsec:van} we prove that the integrals $I_i(\phi,s)$ and $J_i(\phi,s)$ vanish for each $i >1$ and deduce Theorem \ref{thm:GSO(10)}.

\subsection{Non-classical groups} In this subsection we define the group $\underline{H}_7$, specify its Lie algebra concretely, and define certain subgroups of it.  First, recall from \cite{AWZ18} the Zorn model of the octonions $\Theta$.  We will use the notation
\[
\epsilon_1, e_1, e_2, e_3, e_1^*, e_2^*, e_3^*, \epsilon_2
\]
of \cite[section 1.1.1]{AWZ18} to denote a particular basis of $\Theta$.  We write $u_0 = \epsilon_1 - \epsilon_2$.

Define
\[ GT(\Theta) = \{(g_1,g_2,g_3) \in \GSO(\Theta) \times \GSO(\Theta) \times \SO(\Theta): (g_1x,g_2y,g_3z) = (x,y,z) \text{ for all } x,y, z \in \Theta\}.\]
Here $\SO(\Theta)$ is defined via the norm $n_{\Theta}: \Theta \rightarrow F$ and the trilinear form is $(x,y,z) := \tr_{\Theta}(xyz)$.  Denote by $t_1: GT(\Theta) \rightarrow \GSO(\Theta)$ the first projection, and $\nu: GT(\Theta) \rightarrow \GL_1$ the map that is $t_1$ composed with the similitude on $\GSO(\Theta)$.  The subgroup of $GT(\Theta)$ with $\nu=1$ is the group $\Spin(8)=\Spin(\Theta)$.  Define $\underline{H}_7$ to be the subgroup of $GT(\Theta)$ consisting of triples $(g_1,g_2,g_3)$ so that $g_3 \cdot 1 = 1$.  One can check that the similitude $\nu: \underline{H}_7 \rightarrow \GL_1$ is not the trivial map on $k$-points.  We slightly abuse notation and let $t_1:\underline{H}_7 \rightarrow \GSO(\Theta)$ denote the restriction of $t_1$ from $GT(\Theta)$ to $\underline{H}_7$.

We record facts about the group $\underline{H}_7$ that we will need later.  Denote by $\sigma$ the map $x \mapsto x^*$ on $\Theta$. We begin with a simple lemma, whose proof is an exercise.
\begin{lem}\label{lem:GSpinProd} Suppose $(g_1,g_2,g_3) \in \Spin(\Theta)$.
	\begin{enumerate}
		\item Then
		\[ g_1(x) g_2(y) = (\sigma g_3 \sigma)(xy)\]
		for all $x, y \in \Theta$.
		\item Suppose $g=(g_1,g_2,g_3) \in GT(\Theta)$, and define $\nu = \nu(g_1)$.  If $g_3(1) = 1$, then $g_2 = \nu^{-1} \sigma g_1 \sigma$. Consequently, if $(g_1, g_2, g_3) \in GT(\Theta)$ and $g_3(1) = 1$, then
		\[g_1(x) g_1(y)^* = \nu (\sigma g_3 \sigma)(xy^*)\]
		for all $x,y \in \Theta$.
		\item Conversely,  suppose $(g_1,g_2,g_3) \in \Spin(\Theta)$, and $g_2 = \sigma g_1 \sigma$.  Then $g_3$ stabilizes $1$.\end{enumerate}
\end{lem}

Recall that we define the parabolic subgroup $\underline{P}$ of $\underline{G}$ to be the Heisenberg parabolic.  This means that $\underline{P}$ stabilizes an isotropic two-dimensional subspace of $V$ and that the flag variety $\underline{P}(k)\backslash \underline{G}(k)$ is the set of isotropic two-dimensional subspaces of $V$.  In order to compute the double coset space  $\underline{P}(k)\back \underline{G}(k)/\underline{H}(k)$, we will need to use $\underline{H}_7$ to move around various isotropic subspaces of $V$.  To do this, it is helpful to have handy large concrete subgroups of $\ul{H}_7$.  We specify such subgroups now.

The subgroups of $\ul{H}_7$ we will use are $G_2$, $\GSpin(6)$, and $\GL_2 \times \GL_2$.  It is clear that $G_2 \subseteq \ul{H}_7$.  For $\GL_2 \times \GL_2$, we compute inside the Cayley-Dickson construction (see, e.g., \cite[section 1.1]{AWZ18}), so that the multiplication is $(x_1, y_1) (x_2, y_2) = (x_1 x_2 + \gamma y_2'y_1, y_2 x_1 + y_1 x_2')$ and the conjugation is $(x,y)^* = (x',-y)$.  Here for $h = \mm{a}{b}{c}{d} \in M_2(k)$, $h' = \mm{d}{-b}{-c}{a}$ so that $hh'=\det(h)I_2$. Now, suppose $g\in \GL_2$ and $h \in \GL_2$.  Define $(g,h) \cdot (x,y) = (gxh',\mu_h yg')$, where $\mu_h = \mathrm{diag}(\det(h),1)$.
\begin{lem}\label{lem:gl2gl2} This action of $\GL_2 \times \GL_2$ on $\Theta$ defines a map $\GL_2 \times \GL_2 \rightarrow \ul{H}_7$.\end{lem}
\begin{proof} Indeed, if $z_1 = (x_1, y_1)$ and $z_2 = (x_2, y_2)$, then one computes
	\begin{align} \nonumber ((g,h) z_1) \cdot ((g,h) z_2)^* &= (gx_1h',\mu_h y_1g') (gx_2h',\mu_h y_2g')^* \\
 \label{eq:Stab1} &= \det(g)\det(h)(g(x_1 x_2'-\gamma y_2'y_1)g^{-1},\mu_h(-y_2x_1+y_1x_2)h^{-1}) .\end{align}
From this, the lemma is clear.\end{proof}

We now describe the flag variety of Heisenberg parabolic subgroups in $\ul{H}_7$.   For a two-dimensional isotropic subspace $W$ of $\Theta$, define $\kappa'(W) = \{x y^*: x,y \in W\}$.  Then, because $W$ is two-dimensional and isotropic, $\kappa'(W)$ is contained in $V_7 = \Theta^{\tr=0}$, and is either $0$ or a line.  If $\kappa'(W) =0$, we say that $W$ is \emph{null}; otherwise, we say that $W$ is not null.  By Lemma \ref{lem:GSpinProd}, whether or not $W$ is null is an $\ul{H}_7$-invariant. Moreover, it is clear that being isotropic and null is a closed condition on the Grassmanian of two-spaces in $\Theta$, and thus the set of null isotropic two-spaces is a projective variety.

\begin{lem}\label{lem:levi} One has the following:
\begin{enumerate}
	\item The group $\ul{H}_7$ acts transitively on the set of null-isotropic two-spaces $W$ of $\Theta$, and thus the stabilizer $P_W$ of any such $W$ is a parabolic subgroup of $\ul{H}_7$; these are the Heisenberg parabolics.
	\item Denote by $W$ the null isotropic two-dimensional subspace of $\Theta$ that consists of the elements $(0,\mm{*}{*}{0}{0})$.  The map $\GL_2 \times \GL_2 \rightarrow \ul{H}_7$ of Lemma \ref{lem:gl2gl2} identifies $\GL_2 \times \GL_2$ with a Levi subgroup of $P_W$.
\end{enumerate}	
\end{lem}
\begin{proof} The first statement is easily checked, and in any case, is surely well-known.  For the second statement, it is easy to see that this $\GL_2 \times \GL_2$ embeds into $\ul{H}_7$, and that the image preserves $W$.  As the reductive quotient of the parabolic subgroup $P_W$ is exactly $\GL_2 \times \GSpin_3 = \GL_2 \times \GL_2$, the lemma follows.\end{proof}

We next describe the subgroup $\GSpin(6)$ of $\ul{H}_7$ and how it acts on $\Theta$.  Recall the elements $\epsilon_1, \epsilon_2 \in \Theta$ with $1 = \epsilon_1 + \epsilon_2$.  Define $H_6$ to be the subgroup of triples $(g_1,g_2,g_3)$ in $GT(\Theta)$ for which $g_3(\epsilon_j) = \epsilon_j$ for $j = 1,2$.

\begin{lem}\label{lem:H6} The group $H_6$ fixes the four-dimensional subspaces $U_+ = \mm{*}{*}{0}{0}$ and $U_{-} = \mm{0}{0}{*}{*}$ of $\Theta$ under the $t_1$-representation.  Moreover, the image of the map $H_6 \rightarrow \GL(U_+)$ includes $\SL(U_+)$.\end{lem}
\begin{proof} We have the relation $g_3(x) g_1(y) = (\sigma g_2 \sigma)(xy)$ for general triples $(g_1,g_2,g_3) \in \Spin(\Theta)$. Now, $U_{+}$ is the subset of $y \in \Theta$ with $\epsilon_2 y = 0$ and $U_{-} = \{ y \in \Theta: \epsilon_1 y = 0\}$.  The first part of the lemma follows immediately from this.
	
For the second part, it is clear that the image contains the $\SL_3 \subseteq G_2$ that stabilizes $\epsilon_1$ and $\epsilon_2$.  From \eqref{eq:Stab1}, the subgroup $1 \times \SL_2 \subseteq \GL_2 \times \GL_2$ is in $H_6$.  Under the identification of the Cayley-Dickson and the Zorn model of the octonions, the subspace spanned by $e_1, e_2, e_3$ in the Zorn model becomes the subspace of elements $(\mm{0}{*}{0}{0},\mm{*}{0}{*}{0})$ in the Cayley-Dickson model.  Thus, the subgroup $1 \times \SL_2$ of $H_6$ is the $\SL_2$ that acts on the span of $\epsilon_1$ and (say) $e_1$.  Because the image of $H_6$ in $\GL(U_+)$ contains these two subgroups of $\SL(U_+)$, the image must contain all of $\SL(U_+)$, giving the lemma. \end{proof}

Finally, it will be useful to have a concrete realization of the Lie algebra of $\ul{H}_7$ that makes it easy to compute its action on $\Theta$ via the Spin representation $t_1$ (as opposed to the $7$ dimensional orthogonal representation of $\SO(7)$.) Denote by $\ul{H}_7'$ the subgroup of $\ul{H}_7$ with $\nu=1$.  We describe the Lie algebra of $\ul{H}_7'$ as a subalgebra of $\mathfrak{so}(\Theta)$ via the representation $t_1$.

To do this, consider the map $\kappa:\wedge^2 \Theta \rightarrow V_{7} = \Theta^{\tr = 0}$ given by $x \wedge y \mapsto Im(xy^*)$.  Define $K$ to be the kernel of this map.  Note, $K$ is \emph{not} $\wedge^2 V_7$, but it does contain the exceptional Lie algebra $\mathfrak{g}_2$, which is the subspace of $\wedge^2 V_{7}$ in the kernel of $\kappa$.

The following description of the Lie algebra $\ul{H}_7$ is likely well-known.  Not knowing a reference, we include a proof.
\begin{lem} The Lie algebra $\mathfrak{h}_7^0$ of $\ul{H}_7'$ is equal to $K$ as subspaces of $\mathfrak{so}(\Theta) = \wedge^2 (\Theta)$.  In particular, $K$ is a Lie algebra.\end{lem}
\begin{proof} First we claim that $K$ is fixed under the induced action of $\ul{H}_7'$ on $\Theta$.  Indeed, if $x, y \in \Theta$ and $g \in \ul{H}_7'$, then
\[ \kappa(g \cdot (x \wedge y)) = \kappa(g(x) \wedge g(y)) = Im( g(x)g(y)^*) = Im(g'(xy^*)) = g'Im(xy^*) = g'\kappa(x \wedge y).\]
Hence $\kappa$ is equivariant for the action of $\ul{H}_7'$, so $K$ is preserved by $g \in H_7'$.

We have $G_2 \rightarrow \ul{H}_7' \rightarrow \mathrm{Spin}(8)$.  As already mentioned, one has $\mathfrak{g}_2 \subseteq K$.  But because $K$ is closed under the action of $\ul{H}_7'$, $K$ contains $\ul{H}_7' \cdot \mathfrak{g}_2 \subseteq \mathfrak{h}_7^0$.  But $H_7' \cdot \mathfrak{g}_2$ is all of $\mathfrak{h}_7^0$.  Thus $\mathfrak{h}_7^0 \subseteq K$.  But both $\mathfrak{h}_7^0$ and $K$ are $21$-dimensional, thus $K = \mathfrak{h}_7^0$ as claimed.\end{proof}

\subsection{Orbits, stabilizers, and convergence}\label{subsec:orbitStab}  In this subsection we prove that the double coset $\ul{P}(k)\backslash \ul{G}(k)/\ul{H}(k)$ is finite.  We also prove that the integrals $I_i(\phi,s)$ and $J_i(\phi,s)$ associated to these orbits are absolutely convergent.

The variety $\ul{P}(k)\backslash \ul{G}(k)$ is the set of isotropic two-dimensional subspaces $W$ of $V$.  We thus consider the orbits of $\ul{H}(k)$ on these isotropic subspaces.  For $W \subseteq V$ isotropic and two-dimensional, we set $\ul{P}_W$ the stabilizer of $W$ inside $\ul{G}$ and $\ul{H}_W = \ul{H} \cap \ul{P}_W$, the stabilizer of $W$ inside $\ul{H}$.  In this subsection, we will prove the following two statements:

\begin{claim}[Proposition \ref{prop:orbits}] There are finitely many $\ul{H}(k)$-orbits of isotropic two-dimensional subspaces in $V$.\end{claim}

We will then calculate the stabilizers $\ul{H}_W$ for representatives $W$ of these finitely many orbits, and deduce the following.

\begin{claim}[Proposition \ref{prop:goodPairs}] For every isotropic two-dimensional subspace $W$ of $V$, $(\ul{H},\ul{H}_W)$ is a good pair in the sense of \cite[Section 5]{AWZ18}.\end{claim}

As already explained, these two claims imply the equality \eqref{unfolding} and the absolute convergence of the integrals $I_i(\phi,s)$ and $J_i(\phi,s)$ on the right-hand side of this equality.  To compute the orbits, we begin with the following lemma.

\begin{lem}\label{lem:Spin7Move} Suppose $v, w \in \Theta$ are nonzero.  If $n_{\Theta}(v) = n_{\Theta}(w)$, whether $0$ or not, then there exists $g \in \ul{H}_7'$ with $g v = w$. \end{lem}
\begin{proof} Suppose $v \in \Theta$ as above.  We claim that we can use the $H_6'$ inside $\ul{H}_7'$ to move $v$ to $V_7 = \Theta^{\tr=0}$.  From this, the lemma follows from the corresponding fact for $G_2$, by moving both $v$ and $w$ into $V_7$.  To move $v$ into $V_7$, write $v = v_1 + v_2$, with $v_1 \in U_+$ and $v_2 \in U_-$ in the notation of Lemma \ref{lem:H6}.  Thinking about the action of $\SL_4$ on its defining representation and its dual, it is clear that we can simultaneously move $v_1$ into the three-dimensional subspace $\mm{0}{*}{0}{0}$ of $\Theta$ and $v_2$ into the three-dimensional subspace $\mm{0}{0}{*}{0}$.  This proves the lemma.\end{proof}

Given an isotropic two-dimensional subspace $W$ of $V$, it falls into one of four (broad) classes.  Define $pr_{E}: V\rightarrow E$ the orthogonal projection of $V$ onto $E$.
\begin{enumerate}
	\item $pr_{E}(W) = E$
	\item $pr_{E}(W)$ is one-dimensional and anisotropic
	\item $pr_{E}(W)$ is one-dimensional and isotropic
	\item $pr_{E}(W)$ is $0$.\end{enumerate}

Clearly, the class of such $W$ is invariant under the action of $\ul{H}$.  Moreover, let us record now that there will be two subclasses in the case $pr_{E}(W) = 0$: those for which $W \subseteq \Theta$ is null, and those for which $W \subseteq \Theta$ is non-null.

\begin{lem}\label{lem:Spin72} There are two $\ul{H}_7'$ orbits on isotropic two-dimensional subspaces of $\Theta$, consisting of the orbit of null isotropic spaces and of non-null spaces.\end{lem}
\begin{proof} To see that there are at least two orbits, note that $\Span\{ \epsilon_1, e_3 \}$ is not null, whereas the span $\Span\{e_1, e_3^* \}$ is null.
	
To see that there are exactly two orbits, one argues as follows.  First, suppose $W$ is two-dimensional isotropic.  We claim that there is $g \in \ul{H}_7'$ so that $g W \subseteq V_7$.  From this claim, the lemma follows from Lemma 2.5 of \cite{AWZ18}, by using the action of $G_2 \subseteq H_7'$.
	
To see that there is $g \in \ul{H}_7'$ with $g W \subseteq V_7$, write $x, y$ for a basis of $W$.  Applying Lemma \ref{lem:Spin7Move}, we can assume $x = \epsilon_1$.  Because $(x,y) = 0$, we get that $y \in \mm{*}{*}{*}{0}$, and thus may assume $y \in \mm{0}{*}{*}{0} \subseteq V_7$.  Acting by $H_6'$, it is then clear that we can move all of $W$ into $V_7$, as claimed.  This proves the lemma. \end{proof}

\begin{prop}\label{prop:orbits} Each of the four classes of isotropic two-spaces $W$ above makes up finitely many $\ul{H}(k)$-orbits, which are characterized as follows:
	\begin{enumerate}
		\item $pr_{E}(W) = E$;
		\item $pr_{E}(W)$ is one-dimensional anisotropic;
		\item $pr_{E}(W) = (k,0)$ or $(0,k)$ and $pr_{\Theta}(W)$ is one-dimensional isotropic;
		\item $pr_{E}(W) = (k,0)$ or $(0,k)$ and $pr_{\Theta}(W)$ is two-dimensional isotropic and null;
		\item $pr_{E}(W) = (k,0)$ or $(0,k)$ and $pr_{\Theta}(W)$ is two-dimensional isotropic and non-null;
		\item $pr_{E}(W) = 0$ and $W \subseteq \Theta$ is non-null;
		\item $pr_{E}(W) = 0$ and $W \subseteq \Theta$ is null.
\end{enumerate}\end{prop}
\begin{proof} Suppose first that $pr_{E}(W) = E$.  Define $W'= pr_{\Theta}(W)$; this is a non-degenerate two-space.  By Lemma \ref{lem:Spin7Move}, we can use $\ul{H}_7'\subseteq \ul{H}$ to move one element of $W'$ to $1$, so we assume without loss of generality that $1 \in W'$.  Let $u \in W'$ span the perpendicular space $1$; thus, $u \in V_7$.  Because $W'$ combines with $pr_{E}(W) = E$ to make an isotropic two space, $n_{\Theta}(u) \neq 0$ is determined by $E$.  Because $G_2$ acts transitively on such elements $u$, we see that there is one $H$-orbit in this case, as claimed.
	
Next suppose that $pr_{E}(W)$ is one-dimensional and anisotropic.  We may assume that $pr_{E}(W) = k 1$, and then that $W$ contains $(1, 1) \in \Theta \oplus E$.  Let $y \in W \cap \Theta$, so that $y$ is isotropic and perpendicular to $1$, i.e., $y \in V_7$.  Then, because $G_2$ acts transitively on the isotropic lines in $V_7$, we see that there is one orbit of such spaces.
	
If $pr_{E}(W) = 0$, then $W \subseteq \Theta$ is two-dimensional isotropic, and thus we have handled these cases by Lemma \ref{lem:Spin72}.  This completes the possible cases when $E$ is anisotropic, i.e., when $E$ is a field.
	
Thus now assume that $pr_{E}(W)$ is one-dimensional isotropic.  Then $pr_{\Theta}(W)$ is isotropic, and is either one or two-dimensional.  Suppose that $pr_{\Theta}(W)$ is one-dimensional isotropic.  By Lemma \ref{lem:Spin7Move} above, there is one $\ul{H}_7'$-orbit of such lines, thus one orbit in this case.  If $pr_{\Theta}(W)$ is two-dimensional isotropic, then by Lemma \ref{lem:Spin72}, we have two $\ul{H}_7'$ orbits.  This completes the proof of the proposition. \end{proof}

We next compute the stabilizers of the above two-spaces in $\ul{H}$.  These stabilizer computations enable us to apply the results of \cite[Section 5]{AWZ18} to check the convergence of the integrals $I_i(\phi,s)$ and $J_i(\phi,s)$.  In order to prove the \emph{vanishing} of the integrals $I_i(\phi,s)$ and $J_i(\phi,s)$ for $i >1$, we will need to make a different stabilizer computation, which we do in the next subsection.  See Remark \ref{advantage}.

To state the result on the various stabilizers, we require the following notation regarding parabolic subgroups.  Denote by $Z_{\ul{G}} \simeq \GL_1$ the one-dimensional center of $\ul{G}$.  For a nonzero isotropic element $y \in V_7$, denote by $P_{G_2}(y)= M_{G_2}(y)N_{G_2}(y)$ the parabolic subgroup of $G_2$ stabilizing the line $k y$ and by $P_7(y)$ the maximal (Siegel) parabolic subgroup of $\ul{H}_7$ stabilizing $k y$.  For a null isotropic two-dimensional subspace $W$ of $\Theta$, denote by $P_{7,W}$ the maximal (Heisenberg) parabolic subgroup of $\ul{H}_7$ stabilizing $W$.

We also require a notation for a certain non-maximal parabolic subgroup of $\ul{H}_7$.  For this, denote by $P_{7,G}(e_3^*)$ the non-maximal parabolic subgroup of $\ul{H}_7$ that stabilizes the filtration $k e_3^* \subseteq \Span\{e_3^*, e_1, e_2\}$, so that $P_{7,G}(e_3^*) \supseteq P_{G_2}(e_3^*)$.  Let $N_{7,G}(e_3^*)$ be the unipotent radical of $P_{7,G}(e_3^*)$; it is 3-step, $N_{7,G}(e_3^*) \supseteq N_{7,G}(e_3^*)' \supseteq N_{7,G}(e_3^*)''$, with $\dim_k N_{7,G}(e_3^*)/N_{7,G}(e_3^*)' = 4$ and the other two successive quotients of dimension two.

\begin{lem}\label{lem:Stab1} Except in the case where $pr_{\Theta}(W)$ is two-dimensional isotropic and non-null, the groups $\ul{H}_W$ are as follows:
	\begin{enumerate}
		\item Suppose $pr_{E}(W)=E$. Then the map $\ul{H} \rightarrow \ul{H}_7$ induces an isomorphism of $\ul{H}_W$ with a Levi subgroup of the Siegel parabolic of $\ul{H}_7$.  In particular, $\ul{H}_W \simeq \GL_3 \times \GL_1$.
		\item Suppose $pr_{E}(W)$ is one-dimensional anisotropic.  For concreteness, suppose that $W$ is spanned by $(1,1) \in \Theta \oplus E$ and the isotropic element $e_3^* \in V_7 \subseteq \Theta$.  The group $\ul{H}_W$ is $Z_{\ul{G}} \times (M_{G_2}(e_3^*) \ltimes N)$ where $N$ is the six-dimensional subgroup $N_{G_2}(e_3^*)N_{7,G}(e_3^*)'$ of $N_{7,G}(e_3^*)$.
		\item Suppose that both $pr_{E}(W)$ and $pr_{\Theta}(W) = ky$ are one-dimensional isotropic.  Then $\ul{H}_W \simeq P_{7}(y) \times \GL_1$ is the inverse image of $P_7(y)\subset \ul{H}_7$ under the map $\ul{H} \rightarrow \ul{H}_7$.
		\item Suppose that $pr_{E}(W)$ is one-dimensional isotropic, $Y=pr_{\Theta}(W)$ is two-dimensional isotropic and null, and $y \in Y$ is nonzero.  Denote by $Q_{7,y,Y} \subseteq P_{7,Y}$ a non-maximal parabolic subgroup of $\ul{H}_7$ that stabilizes a flag $ky \subseteq Y$ with $Y$ two-dimensional isotropic and null.  Then the map $\ul{H}_W \rightarrow \ul{H}_7$ induces an isomorphism of $\ul{H}_W$ with a parabolic subgroup of $\ul{H}_7$ of the form $Q_{7,y,Y}$.
		\item Suppose that $pr_{E}(W) = 0$, and $pr_{\Theta}(W)$ is two-dimensional isotropic and null.  Then $\ul{H}_W$ is the inverse image of $P_{7,W}\subset \ul{H}_7$ under the map $\ul{H} \rightarrow \ul{H}_7$.
	\end{enumerate}
\end{lem}
\begin{proof} For the first item, we may assume $W = \Span\{\epsilon_1 + (1,0),\epsilon_2 + (0,1)\}$.  Thus any element of $\ul{H}_W$ must stabilize both the line $k \epsilon_1$ and $k\epsilon_2$.  The stabilizers of these two lines in $\ul{H}_7$ is the Levi subgroup of a Siegel parabolic stabilizing, say, $k \epsilon_1$.  This proves that the image of $\ul{H}_W$ in $\ul{H}_7$ lands in the Levi subgroup $\GL_3 \times \GL_1$.  That this map induces an isomorphism $\ul{H}_W \rightarrow \GL_3 \times \GL_1$ is now clear once accounting for the action of $E^\times$ on $(1,0)$ and $(0,1)$.
	
For the second item, accounting for the action of $Z_{\ul{G}}$, we must compute the subgroup of $P_{7}(e_3^*)$ that also stabilizes $\Span\{1,e_3^*\}$.  By acting by the unipotent elements $\exp(x \epsilon_1 \wedge e_3^*) \in \ul{H}_7'$ for $x \in k$, we can reduce the stabilizer to $P_{G_2}(e_3^*)$.  Denote by $N_{G_2}(e_3^*)'$ the three-dimensional commutator subgroup of $N_{G_2}(e_3^*)$.  Because the elements $\exp(x \epsilon_1 \wedge e_3^*)$ span the one-dimensional space $N_{7,G}(e_3^*)'/N_{G_2}(e_3^*)'$, this completes the computation of the stabilizer in this case.

The third, fourth and fifth items are handled immediately. \end{proof}

We now must compute the stabilizers in the case that $pr_{\Theta}(W)$ is two-dimensional isotropic and non-null.  The work is done in the following lemma, which is easily proved once stated.
\begin{lem}\label{lem:nonnullStab} Suppose that $W=\Span\{x,y\}$ is a two-dimensional isotropic but non-null subspace of $\Theta$.  Set $b = x y^*$ and $U(b)= \{z \in \Theta: bz = 0\}$.
\begin{enumerate}
	\item The space $U(b)$ is a four-dimensional isotropic subspace of $\Theta$, that comes equipped with the symplectic form $\langle z_1, z_2 \rangle$ defined by $z_1 z_2^* = \langle z_1, z_2 \rangle b$.
	\item Denote by $Q_{U(b)}$ the subgroup of $\ul{H}_7$ that stabilizes $U(b)$.  Then $Q_{U(b)} = L_{U(b)}V_{U(b)}$ is a parabolic subgroup of $\ul{H}_7$ with Levi subgroup $L_{U(b)}= \GL_1 \times \mathrm{GSpin}(5) = \GL_1 \times \mathrm{GSp}_4$.  The unipotent radical $V_{U(b)}$ is abelian of dimension $5$ and the map $Q_{U(b)} \rightarrow \mathrm{GSp}_4$ is induced by the action of $Q_{U(b)}$ on $U(b)$.
	\item The stabilizer of $W$ inside $\ul{H}_7$ is $L' V_{U(b)}$ with $L' = \GL_1 \times (\GL_2 \times \GL_2)^0 \subseteq \GL_1 \times \mathrm{GSp}_4$.  Here $(\GL_2 \times \GL_2)^0$ is the subgroup of pairs $(g_1,g_2) \in \GL_2 \times \GL_2$ with $\det(g_1)=\det(g_2)$.	
\end{enumerate}
\end{lem}

Applying Lemma \ref{lem:nonnullStab}, one obtains the following for the stabilizers $\ul{H}_W$ in case $pr_{\Theta}(W)$ is two-dimensional isotropic and non-null.
\begin{lem}\label{lem:Stab2} Let $L'$ and $V=V_{U(b)}$ be as in Lemma \ref{lem:nonnullStab}.
\begin{enumerate}
	\item Suppose that $pr_{E}(W) = 0$ and $pr_{\Theta}(W)$ is two-dimensional isotropic and non-null.  Then $\ul{H}_W$ is the inverse image  of $L' \ltimes V\subset \ul{H}_7$ under the map $\ul{H} \rightarrow \ul{H}_7$.
	\item Suppose that $pr_{E}(W)$ is one-dimensional isotropic and $pr_{\Theta}(W)$ is two-dimensional isotropic and non-null.  Denote by $B$ a Borel subgroup of $\GL_2$ and set $L'' := \GL_1 \times (B \times \GL_2)^0 \subseteq L'$. Then the map $\ul{H}_W \rightarrow \ul{H}_7$ induces an isomorphism $\ul{H}_W \simeq L'' \ltimes V$.
\end{enumerate}
\end{lem}
\begin{proof} This follows immediately from Lemma \ref{lem:nonnullStab}.\end{proof}

\begin{prop}\label{prop:goodPairs}
For every isotropic two-space $W$ of $V$, the pair $(\ul{H},\ul{H}_W)$ is a good pair in the sense of \cite[Section 5]{AWZ18}.
\end{prop}

\begin{proof}
We first consider the cases when $pr_{\Theta}(W)$ is not two-dimensional isotropic and non-null. By Lemma \ref{lem:Stab1}, we can find a parabolic subgroup $P_{\ul{H}}=M_{\ul{H}}N_{\ul{H}}$ of $\ul{H}$ and a closed subgroup $M'$ of $M_{\ul{H}}$ ($M'=1$ for case (3) and (5) in Lemma \ref{lem:Stab1}; $M'\simeq \GL_1$ for case (1) and (4) in Lemma \ref{lem:Stab1}; $M'\simeq \GL_1\times \GL_1$ for case (2) in Lemma \ref{lem:Stab1}) such that the following two conditions hold.
\begin{enumerate}
\item $\ul{H}_W=(\ul{H}_W\cap M_{\ul{H}})\ltimes(\ul{H}_W\cap N_{\ul{H}})$.
\item $M_{\ul{H}}=(\ul{H}_W\cap M_{\ul{H}})\times M'$.
\end{enumerate}
Then we know that $(\ul{H},\ul{H}_W)$ is a good pair by Corollary 5.9 of \cite{AWZ18}.

Now we consider the cases when $pr_{\Theta}(W)$ is two-dimensional isotropic and non-null. If $pr_{E}(W)$ is one-dimensional isotropic, by Lemma \ref{lem:Stab2}, we can still find a parabolic subgroup $P_{\ul{H}}=M_{\ul{H}}N_{\ul{H}}$ of $\ul{H}$ and a closed subgroup $M'$ of $M_{\ul{H}}$ such that condition (1) and (2) above hold. In fact, $P_{\ul{H}}$ is the parabolic subgroup whose Levi part is isomorphic to $\mathrm{GSpin_3}\times \GL_1\times \GL_1\times \GL_1=\mathrm{GSp_2}\times \GL_1\times \GL_1\times \GL_1$, and $M'\simeq \GL_1$. Then we know that $(\ul{H},\ul{H}_W)$ is a good pair by Corollary 5.9 of \cite{AWZ18}.

Hence the only case left is when $pr_{E}(W) = 0$ and $pr_{\Theta}(W)$ is two-dimensional isotropic and non-null. By Proposition 5.8(4) of \cite{AWZ18} and Lemma \ref{lem:Stab2}(1) above, in order to show that $(\ul{H},\ul{H}_W)$ is a good pair, it is enough to prove the following lemma.
\end{proof}

\begin{lem}\label{good pair}
$(\GSp_4\ltimes U,(\GL_2 \times \GL_2)^0\ltimes U')$ is a good pair. Here $U$ is some unipotent group and $U'\subset U$ is a closed subgroup.
\end{lem}

\begin{proof}
The proof is very similar to the argument in Section 5.4 of our previous paper \cite{AWZ18}, we will write it in the Appendix.
\end{proof}

\subsection{Vanishing and reduction of period}\label{subsec:van} In this subsection we prove the following result, which immediately implies the vanishing of the integrals $I_i(\phi,s)$ and $J_i(\phi,s)$ for $i > 1$.

\begin{prop}\label{prop:VanW} Suppose that $W$ is an isotropic two-dimensional subspace of $V$, but exclude the case that $W \subseteq \Theta$ is isotropic and null.  Suppose that $\beta: \ul{P}(k)N(\BA)\backslash \ul{G}(\BA) \rightarrow \C$ is a measurable function, with $m \mapsto \beta(mg)$ a cuspidal function on $\ul{M}(\BA)$ for almost every $g \in \ul{G}(\BA)$, and that the integral
\[
	\mathcal{P}_W(\beta) = \int_{Z_{\ul{G}}(\BA)\ul{H}_W(k)\backslash \ul{H}(\BA)}{\beta(h)\,dh}
\]
converges absolutely.  Then $\mathcal{P}_W(\beta) = 0$.\end{prop}
\begin{proof}
We consider the vanishing of the various orbits one-by-one.  For an isotropic two-dimensional subspace $W$ of $V=\Theta \oplus E$, we make and recall the following notations:
\begin{itemize}
	\item $\ul{P}_W = \ul{M}_W \ul{N}_W \subseteq \GSO(V)$ the Heisenberg parabolic that is the stabilizer of $W$, with unipotent radical $\ul{N}_W$ and Levi subgroup $\ul{M}_W$.
	\item $\ul{H}_W \subseteq \ul{H}$ the stabilizer of $W$ inside $\ul{H}$, i.e. $\ul{H}_W = \ul{H} \cap \ul{P}_W$.
	\item $H_W'$ the image of $\ul{H}_W$ inside the reductive quotient of $\ul{P}_W$.
\end{itemize}
First consider the case that $pr_{E}(W) = E$.  Then we may assume that $W = \Span\{ \epsilon_1 + (1,0), \epsilon_2 + (0,1) \}.$  Set $W' = \{ \epsilon_1 - (1,0), \epsilon_2 - (0,1) \}$ and denote by $V_6$ the subspace of $\Theta$ perpendicular to $k \epsilon_1 \oplus k \epsilon_2$.  Then $V= W' \oplus V_6 \oplus W$.  The semisimple part of $H_W'$ in this case is $\SL_3$, acting on $V_6$ as the direct sum of the standard representation and its dual.  The vanishing of this orbit thus follows from the following lemma.
\begin{lem} Suppose $\alpha$ is a cusp form on $\SO(V_6)$.  Embed $\SL_3 \subseteq \SO(V_6)$ as the semisimple part of the Levi of a Siegel parabolic.  Then the period $\int_{[\SL_3]}{\alpha(h)\,dh} = 0$. \end{lem}
\begin{proof} One first Fourier expands $\alpha$ along the abelian unipotent radical of the Siegel parabolic.  By standard unfolding arguments, one quickly sees that the integral vanishes by the cuspidality of $\alpha$ along the unipotent radical of the maximal parabolic of $\SO(V_6)$ that stabilizes an isotropic line. \end{proof}

Next suppose that $pr_E(W)$ is one-dimensional and anisotropic.  Then we may assume that $W = \Span\{ 1_{\Theta} + 1_{E}, y\}$ with $y \in V_7$ isotropic.  Let $y' \in V_7$ be isotropic with $(y,y')=1$, and denote by $V_5(y)$ the perpendicular space to $\Span\{ y, y'\}$ inside $V_7$.  Set $W' = \Span\{ 1_{\Theta}-1_E, y' \}$ and $V_6^W = V_5(y) \oplus k(1,-1)$.  Then $V= W' \oplus V_6^W \oplus W$.  Moreover, the stabilizer $\ul{H}_W$ contains the parabolic subgroup $P_{G_2}(y)$.  To check the vanishing of this orbit, we must consider the image $H_W'$ inside $\SO(V_6^W) \times \GL(W)$.

To do this, we first consider the image of $P_{G_2}(y)$ inside $\SO(V_5(y))$. We have the following lemma, which is easily checked.
\begin{lem} Let $V_3(y) \supseteq ky$ be the three-dimensional isotropic subspace of $V_7$ stabilized by $P_{G_2}(y)$, and set $P''$ the parabolic subgroup of $\SO(V_5(y))$ that stabilizes $V_3(y)/ky$.  Then the image of $P_{G_2}(y)$ inside of $\SO(V_5(y))$ is $P''$. \end{lem}

Denote by $P'''$ the derived subgroup of $P''$.  Then the image of $P'''$ in $\SO(V_6^W) \times \GL(W)$ is contained in $\SO(V_6^W)$.  We are thus left to consider the $P'''$ periods of cusp forms on $\SO(V_6)$, which we do in the following lemma.
\begin{lem} The $P'''$ periods of cusp forms on $\SO(V_6)$ vanish.\end{lem}
\begin{proof} Denote by $N'$ the unipotent radical of the parabolic subgroup of $\SO(V_6^W)$ that stabilizes the line $ke_1$ in $V_5(y)$.  Then it is simple to show that $P'''$-period of $\SO(V_6)$-cusp forms vanish by cuspidality along $N'$.\end{proof}

Next we suppose that $pr_{E}(W)$ and $pr_{\Theta}(W)$ are each one-dimensional isotropic.  Then we may assume $W = \Span\{ \epsilon_1, (1,0)\}$.  Define $W' = \Span\{ \epsilon_2, (0,1)\}$ and $V_6 \subseteq \Theta$ the perpendicular space to $\Span\{\epsilon_1,\epsilon_2\}$.  Then $V = W' \oplus V_6 \oplus W$.  It is clear that $H_W'$ includes the $\SL_3$ acting on $V_6$, and thus these orbits vanish just as the first case above.

Now suppose that $pr_{E}(W)$ is one-dimensional isotropic and $pr_{\Theta}(W)$ is two-dimensional isotropic and null.  Then we may assume that $W = \Span\{ e_1 + (1,0), e_3^* \}$.  Then
\[W^\perp = \Span\{ e_1 + (1,0), e_3^*, e_1, \epsilon_1, \epsilon_2, e_2^*, e_2, e_1^*+(0,1)\}.\]
Consider the four elements $e_1 \wedge e_2^*$, $u_0 \wedge e_1 + e_2^* \wedge e_3^*$, $\epsilon_2 \wedge e_1$, $u_0 \wedge e_3^* + e_1 \wedge e_2$ of $K$.  Denote by $X$ the Lie subaglebra of $K$ generated by these elements and by $N'$ the unipotent subgroup of $\ul{H}_7'$ whose Lie algebra is $X$.  We have the following lemma, which implies the vanishing for this orbit.
\begin{lem} The group $N'$ acts as the identity on $W,$ and the image of $N'$ in $\SO(W^\perp/W)$ is the unipotent radical of the parabolic subgroup of $\SO(W^\perp/W)$ that fixes the isotropic line $ke_1$.\end{lem}
\begin{proof} One sees easily that $X$ annihilates $e_1$ and $e_3^*$.  It follows that $N'$ acts as the identity on $W$, and that the induced action on the six-dimensional space $W^\perp/W$ fixes the isotropic vector $e_1$.  Moreover, one computes immediately $X \cdot (e_1^* + (0,1)) = \Span\{\epsilon_1, \epsilon_2, e_2^*, e_2 \}$.  From this it follows that the image of $N'$ in $\SO(W^\perp/W)$ is the entire unipotent radical of the parabolic subgroup stabilizing $e_1$, proving the lemma. \end{proof}

Now assume that $pr_{E}(W)$ is one-dimensional isotropic and $pr_{\Theta}(W)$ is two-dimensional isotropic and non-null.  Then we may assume that $W = \Span\{ \epsilon_1, e_3 + (1,0)\}$.  Then
\[W^\perp = \Span\{e_1, e_2, e_1^*, e_2^*, e_3^*+(0,1), e_3, e_3+ (1,0), \epsilon_1 \}.\]
Consider the element $\epsilon_1 \wedge e_3^*$ of $K$.  This element is nilpotent and annihilates $e_1, e_2, e_1^*, e_2^*$, $e_3^*+(1,0)$, and $\epsilon_1$.  Moreover $(\epsilon_1 \wedge e_3^*)(e_3) = - \epsilon_1$.  It follows that $\exp(x \epsilon_1 \wedge e_3^*)$ acts as the identity on $W^\perp/W$, and acts on $W$ as the unipotent radical of the a Borel subgroup of $\GL(W)$.  Consequently, this orbit vanishes by cuspidality on $\GL_2$.

Finally suppose that $W \subseteq \Theta$ is isotropic and non-null.  Then we may assume $W = \Span\{ \epsilon_1, e_3 \}$, so that $W^\perp = \Span\{ (1,0), (0,1), e_1, e_2, e_1^*, e_2^*, \epsilon_1, e_3 \}$.  Again, consider the elements $\exp(x \epsilon_1 \wedge e_3^*)$ of $\ul{H}_7'$.  It is immediate that they act as the identity on $W^\perp/W$ and act on $W$ as the unipotent radical of a Borel subgroup.  Consequently, this orbit also vanishes by cuspidality along $\GL_2$. \end{proof}

Combining Proposition \ref{prop:VanW}, Proposition \ref{prop:goodPairs}, and Proposition \ref{prop:orbits}, we arrive at the following.  Recall that $H = \ul{H}\cap \ul{M}$, where $\ul{P}=\ul{M}\ul{N}$ is the Heisenberg parabolic subgroup of $\ul{G}$ that stabilizes a two-dimensional isotropic and null subspace of $\Theta \subseteq V$.
\[
\mathcal{P}_{\ul{H}}(\Lambda^T E(\phi,s)) = I_1(\phi,s)+J_1(\phi,s)
\]
where
\begin{align*}
I_1(\phi,s) &= \int_{Z_G(\BA)H(k)\backslash \ul{H}(\BA)}{(1-\widehat{\tau_{\ul{P}}}(H_{\ul{P}}(h)-T)) e^{\langle s \omega_{\ul{P}}, H_{\ul{P}}(h)\rangle} \phi(h)\,dh} \\ &= \frac{e^{(s-s_0)T}}{s-s_0}\int_{K_{\ul{H}}}\int_{Z(\BA)H(k)\back H(\BA)} \phi(hk)dhdk,
\end{align*}
and similarly
\begin{align*}
J_1(\phi,s) &= \int_{Z_G(\BA)H(k)\backslash \ul{H}(\BA)}{\widehat{\tau_{\ul{P}}}(H_{\ul{P}}(h)-T) e^{\langle -s \omega_{\ul{P}}, H_{\ul{P}}(h)\rangle} M(s)\phi(h)\,dh} \\ &= \frac{e^{(-s-s_0)T}}{-s-s_0}\int_{K_{\ul{H}}}\int_{Z(\BA)H(k)\back H(\BA)} M(s)\phi(hk)dhdk.
\end{align*}

Then we compute the constant $s_0$ in Method 2. Like Method 1, we have $s_0=c(1-2c_{\ul{P}}^{\ul{H}})$. In this case, although the unipotent group $\ul{N}$ is not abelian, it is easy to see that $c_{\ul{P}}^{\ul{H}}=\frac{8}{14}$. On the other hand, by Proposition \ref{the constant c}(4), we have $c=\frac{7}{2}$. This implies that $s_0=\frac{1}{2}$. As a result, we have proved that
$$\mathcal{P}_{\ul{H}}(\Lambda^T E(\phi,s)) =\frac{e^{(s-\frac{1}{2})T}}{s-\frac{1}{2}}\int_{K_{\ul{H}}}\int_{Z(\BA)H(k)\back H(\BA)} \phi(hk)dhdk +\frac{e^{(-s-\frac{1}{2})T}}{-s-\frac{1}{2}}\int_{K_{\ul{H}}}\int_{Z(\BA)H(k)\back H(\BA)} M(s)\phi(hk)dhdk.$$
By taking the residue at $s=\frac{1}{2}$, we have
$$\mathcal{P}_{\ul{H}}(\Lambda^T Res_{s=\frac{1}{2}}E(\phi,s)) =\int_{K_{\ul{H}}}\int_{Z(\BA)H(k)\back H(\BA)} \phi(hk)dhdk -e^{-T}\int_{K_{\ul{H}}}\int_{Z(\BA)H(k)\back H(\BA)} Res_{s=\frac{1}{2}}M(s)\phi(hk)dhdk.$$
This proves Theorem \ref{thm:GSO(10)}.

To conclude this subsection, we make explicit the period integral $\mathcal{P}_H$ in terms of the isomorphism $\PGL_4 \simeq \mathrm{PGSO}_6$.  More precisely, the map $H \rightarrow \ul{M} \simeq \GL_2 \times \GSO(6)$ induces
\[
H/Z \rightarrow \PGL_2 \times \mathrm{PGSO}(6) \simeq \PGL_2 \times \PGL_4.
\]
We have already noted that $H/Z \simeq (\GL_2 \times \GL_2)/\Delta(\GL_1)$.  In the following lemma, we make explicit the induced map
\begin{equation}\label{eq:Agps}
  (\GL_2 \times \GL_2)/\Delta(\GL_1) \rightarrow \PGL_2\times \PGL_4.
\end{equation}

\begin{lem} The map \eqref{eq:Agps} is induced by the map $\GL_2 \times \GL_2 \rightarrow \GL_4 \times \GL_2$ given by $(a,b) \mapsto (\mm{a}{}{}{b}, a)$ in $2 \times 2$ block form.\end{lem}
\begin{proof} From Lemma \ref{lem:gl2gl2} and Lemma \ref{lem:levi}, the map $H \rightarrow \ul{M}= \GL_2 \times \GSO(6)$ is given by
\[(g,h,(\lambda_1,\lambda_2)) \mapsto (g,j(g,h,\lambda))\]
where $j(g,h,\lambda)$ acts on $V_6 = M_2(k) \oplus E$ as $(m,\mu)\mapsto (gmh',\lambda \mu)$. Up to the action of $Z = \GL_1 \times \GL_1$ which sits inside $H$ as triples $(z,w,(zw,zw))$, we can assume that $(g,h,\lambda) = (g,h,(\det(g),\det(h)))$.  Denote by $V_4 = V_2 \oplus V_2$ the decomposition of the defining representation of $\GL_4$ into two $\GL_2$ representations. Recall that our map $\GL_4 \rightarrow \GSO(6)$ is induced by the exterior square representation. The element $\mm{g}{}{}{h} \in \GL_4$ acts on $V_6 =\wedge^2(V_4)$ by $(m,\mu)\mapsto (gmh',(\det(g),\det(h))\mu)$ for an appropriate choice of basis. The lemma follows.\end{proof}

\subsection{The local result}\label{subsection GL(4)GL(2) local}
Let $F$ be a local field (archimedean or p-adic), and $D/F$ be the unique quaternion algebra if $F\neq \BC$. Let
$$G(F)=\GL_4(F)\times \GL_2(F),\;H(F)=\left\{\begin{pmatrix} a&0\\0&b\end{pmatrix}\times \begin{pmatrix}a\end{pmatrix}|\; a,b\in \GL_2(F)\right\}$$
as in the previous subsections. Let $\pi$ be an irreducible smooth representation of $G(F)$ with trivial central character (we can also consider the nontrivial central character case, but we assume it is trivial here for simplicity), define the multiplicity
$$m(\pi)=\dim(\Hom_{H(F)} (\pi,1)).$$
Similarly, if $F\neq \BC$, we can define the quaternion version of the model $(G_{D},H_{D})$ with $G_D(F)=\GL_2(D)\times \GL_1(D)$ and $H_D(F)\simeq \GL_1(D)\times \GL_1(D)$. We can also define the multiplicity $m(\pi_D)$ for all irreducible smooth representations of $G_D(F)$ with trivial central character.

Assume that $F$ is $p$-adic. Let $\pi=\pi_1\otimes \pi_2$ (resp. $\pi_D=\pi_{1,D}\otimes \pi_{2,D}$) be an irreducible tempered representation of $G(F)$ (resp. $G_D(F)$) with trivial central character. We define the geometric multiplicity
$$m_{geom}(\pi)=c_{\pi_1}(1)c_{\pi_2}(1)+\Sigma_{T\in \CT_{ell}(\GL_2)} |W(\GL_2,T)|^{-1} \int_{T(F)/Z_{\GL_2}(F)} D^{\GL_2(F)}(t)^2 c_{\pi_1}(\begin{pmatrix} t&0\\0&t\end{pmatrix}) \theta_{\pi_2}(t)dt.$$
Here $\CT_{ell}(\GL_2)$ is the set of all the maximal elliptic tori of $\GL_2(F)$ (up to conjugation), $W(\GL_2,T)$ is the Weyl group, $Z_{\GL_2}$ is the center of $\GL_2$, $dt$ is the Haar measure on $T(F)/Z_{\GL_2}(F)$ such that $vol(T(F)/Z_{\GL_2}(F))=1$, $D^{\GL_2(F)}(t)$ is the Weyl determinant, $\theta_{\pi_i}$ is the distribution character of $\pi_i$, $c_{\pi_1}(1)$ is the regular germ of $\theta_{\pi_1}$ at the identity element, and $c_{\pi_1}(\begin{pmatrix} t&0\\0&t\end{pmatrix})$ is the regular germ of $\theta_{\pi_1}$ at $\begin{pmatrix} t&0\\0&t\end{pmatrix}$. We refer the readers to Section 4.5 of \cite{B15} for the definition of regular germs. Similarly, we can also define the quaternion version of the geometric multiplicity
$$m_{geom}(\pi_D)=\Sigma_{T_D\in \CT_{ell}(\GL_1(D))} |W(\GL_1(D),T)|^{-1} \int_{T_D(F)/Z_{\GL_1(D)}(F)} D^{\GL_1(D)}(t_D)^2 c_{\pi_{1,D}}(\begin{pmatrix} t_D&0\\0&t_D\end{pmatrix}) \theta_{\pi_{2,D}}(t_D)dt_D.$$
Note the for the quaternion case, we don't need to include the regular germ at the identity element because the group is not quasi-split.

The proof of the following theorem follows from a similar (but easier) argument as in the Ginzburg-Rallis model case (\cite{Wan15}, \cite{Wan16}, \cite{Wan16b}), we will skip it here. In fact, the argument for the current model is more similar to the argument for the ``middle model" (which is a reduced model of the Ginzburg-Rallis model) defined in Appendix B of \cite{Wan15}. But it is easier since $H$ is reductive.

\begin{thm}
\begin{enumerate}
\item Assume that $F$ is p-adic. Let $\pi=\pi_1\otimes \pi_2$ (resp. $\pi_D=\pi_{1,D}\otimes \pi_{2,D}$) be an irreducible tempered representation of $G(F)$ (resp. $G_D(F)$) with trivial central character. Then we have a multiplicity formula
    $$m(\pi)=m_{geom}(\pi),\;m(\pi_D)=m_{geom}(\pi_D).$$
\item Assume that $F$ is p-adic. Let $\pi=\pi_1\otimes \pi_2$ be an irreducible tempered representation of $G(F)$ with trivial central character, and let $\pi_D$ be the Jacquet-Langlands correspondence of $\pi$ from $G(F)$ to $G_{D}(F)$ if it exists; otherwise let $\pi_D=0$. Then we have
    $$m(\pi)+m(\pi_D)=1.$$
    In other words, the summation of the multiplicities over every tempered local Vogan L-packet is equal to 1.
\item The statement in (2) also holds when $F=\BR$.
\item When $F=\BC$, the multiplicity $m(\pi)=1$ for all irreducible tempered representation $\pi$ of $G(F)$ with trivial central character.
\end{enumerate}
\end{thm}

\begin{rmk}
As in the Ginzburg-Rallis model case, we can make the epsilon dichotomy conjecture for this model. To be specific, let $\pi=\pi_1\otimes \pi_2$ be an irreducible tempered representation of $G(F)$ with trivial central character, and let $\pi_D$ be the Jacquet-Langlands correspondence of $\pi$ from $G(F)$ to $G_{D}(F)$ if it exists; otherwise let $\pi_D=0$. Then the conjecture states that
$$m(\pi)=1\iff \epsilon(1/2,\pi,\rho_X)=1,\;m(\pi)=0\iff \epsilon(1/2,\pi,\rho_X)=-1.$$
Here $\rho_X=\wedge^2\otimes std$ is the 12-dimensional representation of ${}^LG=\GL_4(\BC)\times \GL_2(\BC)$ as in the previous subsections.

By a similar argument as in the Ginzburg-Rallis model case (\cite{Wan16}, \cite{Wan16b}), we can prove this conjecture when $F$ is archimedean. And when $F$ is p-adic, we can prove this conjecture when $\pi$ is not a discrete series.
\end{rmk}

\begin{rmk}
In general, we expect the results above hold for all generic representations of $G(F)$.
\end{rmk}

\appendix

\section{The proof of Lemma \ref{good pair}}
In this appendix, we are going to prove Lemma \ref{good pair}. Let $G=G_0\ltimes U$ where $G_0=\GSp_4$ and $U$ is some unipotent group. Let $H=H_0\ltimes U'$ be a subgroup of $G$ with $U'$ being a subgroup of $U$ and $H_0=(\GL_2\times \GL_2)^0\subset \GSp_4=G_0$. Our goal is to prove the following lemma.

\begin{lem}\label{good pair lemma}
The pair $(G,H)$ is a good pair.
\end{lem}

The proof is very similar to the argument in Section 5.4 of our previous paper \cite{AWZ18}, we only include it here for completion. We use the same notations as in Section 5 of \cite{AWZ18}. We need some preparation. Let $w_2=\begin{pmatrix}0&1\\1&0\end{pmatrix},\;J_2=\begin{pmatrix}0&1\\-1&0\end{pmatrix}$, and $J_4=\begin{pmatrix}0&w_2\\-w_2&0\end{pmatrix}$. We define the groups $\GSp_4$ to be
$$\GSp_4=\{g\in \GL_4|\;g^tJ_4g=\lambda J_4\; \text{for some}\;\lambda\in \GL_1\}.$$
The embedding $(\GL_2\times \GL_2)^0\rightarrow \GSp_4$ is given by
$$\begin{pmatrix}a_1&b_1\\c_1&d_1\end{pmatrix}\times \begin{pmatrix}a_2&b_2\\c_2&d_2\end{pmatrix}\in (\GL_2\times \GL_2)^0\mapsto \begin{pmatrix}a_1&0&0&b_1\\0&a_2&b_2&0\\0& c_2&d_2&0\\c_1&0&0&d_1\end{pmatrix}\in \GSp_4.$$

Let $B=TN$ be the upper triangular Borel subgroup of $G_0=\GSp_4$ with $T$ being the group of diagonal elements in $B$. Then $B_H=B\cap H_0$ is a Borel subgroup of $H_0=(\GL_2\times \GL_2)^0$ with the Levi decomposition $B_H=TN_H$ where $N_H=N\cap H_0$. Let
$$N'=\{\begin{pmatrix}1&a_{12}&a_{13}&a_{14}\\0&1&a_{23}&a_{24}\\0&0&1&a_{34}\\0&0&0&1 \end{pmatrix}\in N |\;a_{14}=a_{23}=0\} $$
be a closed subvariety of $N$ (note that it is not a group). The map
$$N_H\times N'\rightarrow N:\;(n,n')\mapsto nn'$$
is an isomorphism of varieties.

\begin{lem}\label{U1 part}
For all $h\in H_0(\BA_{\bar{k}})$ and $n'\in N'(\BA)$, we have
$$||hn'||_G\sim  ||h||_G \cdot ||n'||_G.$$
\end{lem}

\begin{proof}
By the Iwasawa decomposition, it is enough to consider the case when $h=tn$ with $t\in T(\BA_{\bar{k}})$ and $n\in N_H(\BA_{\bar{k}})$. Since $N=N_HN'$, $B=TN$ is a parabolic subgroup of $G_0$ and $B_H=TN_H$ is a parabolic subgroup of $H_0$, we have
$$||hn'||_G=||tnn'||_G\sim ||t||_G \cdot ||nn'||_G\sim ||t||_G \cdot ||n||_G \cdot ||n'||_G\sim ||tn||_G \cdot ||n'||_G=||h||_G \cdot ||n'||_G.$$
This proves the lemma.
\end{proof}

Now we are ready to prove Lemma \ref{good pair lemma}. For $g\in G(\BA)$, we want to show that
$$\inf_{\gamma\in H(\bar{k})} ||\gamma g||_G\gg \inf_{\gamma\in H(k)} ||\gamma g||_G.$$
By the Iwasawa decomposition, it is enough to consider the case when $g=utnn'$ with $u\in U(\BA),t\in T(\BA), n\in N_H(\BA)$ and $n'\in N'(\BA)$. Since $TN_H\in H_0$, we can write $g$ as $uh_0n'$ with $u\in U(\BA),h_0\in H_0(\BA)$ and $n'\in H(\BA)$. In order to prove Lemma \ref{good pair}, it is enough to prove the following lemma.

\begin{lem}
For all $u\in U(\BA),h_0\in H_0(\BA)$ and $n'\in N'(\BA)$, we have
\begin{equation}\label{G_2 1}
\inf_{\gamma \in H_0(\bar{k}),\nu \in U'(\bar{k})} ||\nu\gamma uh_0n'||_G \gg
\inf_{\gamma \in H_0(k),\nu\in U'(k)} ||\nu\gamma uh_0n'||_G.
\end{equation}
\end{lem}

\begin{proof}
The proof is exactly the same as the proof of Proposition 5.12 of \cite{AWZ18}. All we need to do is replace Lemma 5.11 of loc. cit. by Lemma \ref{U1 part} above. This finishes the proof of Lemma \ref{good pair}.
\end{proof}

\bibliographystyle{amsalpha}
\providecommand{\bysame}{\leavevmode\hbox to3em{\hrulefill}\thinspace}
\providecommand{\MR}{\relax\ifhmode\unskip\space\fi MR }
\providecommand{\MRhref}[2]{%
  \href{http://www.ams.org/mathscinet-getitem?mr=#1}{#2}
}
\providecommand{\href}[2]{#2}

\end{document}